\documentclass[letterpaper,11pt,oneside,reqno]{article}

\usepackage[sorting=nyt,style=alphabetic,backend=bibtex,hyperref=true,doi=false,maxbibnames=9,maxcitenames=4,eprint=false]{biblatex}

\makeatletter
\def\blx@maxline{77}
\makeatother
\usepackage{amsmath,amssymb,amsthm,amsfonts}
\usepackage{graphicx,color}
\usepackage{hyperref}
\usepackage{upgreek}
\usepackage[mathscr]{euscript}

\allowdisplaybreaks
\numberwithin{equation}{section}

\usepackage{tikz}
\usetikzlibrary{
	shapes,
	arrows,
	positioning,
	decorations.markings,
	decorations.pathmorphing
}

\usepackage{array}
\usepackage{cleveref}
\usepackage{enumerate}

\usepackage[DIV=13]{typearea}

\newcommand{\RoadblockSet}{\mathbf{B}}

\newcommand{\SpeedEssRange}{\Xi}
\newcommand{\SpeedEssRangeCirc}{\Xi^\circ}

\newcommand{\iu}{{\mathbf{i}\mkern1mu}}

\renewcommand{\Re}{\mathop{\mathrm{Re}}}
\renewcommand{\Im}{\mathop{\mathrm{Im}}}

\newcommand{\discnu}{\nu}
\newcommand{\discbeta}{\beta}
\newcommand{\disczcr}{\mathsf{z}}
\newcommand{\dischlim}{\mathsf{h}}

\newcommand{\discKernel}{\mathsf{K}}

\newcommand{\contKernel}{\mathcal{K}}

\newtheorem{proposition}{Proposition}[section]
\newtheorem{lemma}[proposition]{Lemma}

\newtheorem{theorem}[proposition]{Theorem}
\theoremstyle{definition}
\newtheorem{definition}[proposition]{Definition}

\newtheorem{remark}[proposition]{Remark}
\begin{document}
\title{Generalizations of TASEP in discrete and continuous inhomogeneous space}
\author{Alisa Knizel\thanks{knizel@math.columbia.edu}, 
Leonid Petrov\thanks{lenia.petrov@gmail.com}, 
and Axel Saenz\thanks{ais6a@virginia.edu}}

\date{}

\maketitle

\begin{abstract}
	We investigate a rich new class of exactly solvable particle systems generalizing the Totally Asymmetric Simple Exclusion Process (TASEP). Our particle systems can be thought of as new exactly solvable examples of tandem queues, directed first- or last-passage percolation models, or Robinson-Schensted-Knuth type systems with random input. One of the novel features of the particle systems is the presence of spatial inhomogeneity which can lead to the formation of traffic jams.
	
	For systems with special step-like initial data, we find explicit limit shapes, describe hydrodynamic evolution, and obtain asymptotic fluctuation results which put the systems into the Kardar-Parisi-Zhang universality class. At a critical scaling around a traffic jam in the continuous space TASEP, we observe deformations of the Tracy-Widom distribution and the extended Airy kernel, revealing the finer structure of this novel type of phase transitions.

	A homogeneous version of a discrete space system we consider is a one-parameter deformation of the geometric last-passage percolation, and we obtain extensions of the limit shape parabola and the corresponding asymptotic fluctuation results.
	
	The exact solvability and asymptotic behavior results are powered by a new nontrivial connection to Schur measures and processes. 
\end{abstract}

\setcounter{tocdepth}{1}
\tableofcontents
\setcounter{tocdepth}{3}

\section{Introduction}
\label{sec:intro}

\subsection{Discrete time TASEP}
\label{sub:tasep_intro}

The paper's main goal is two-fold:
\begin{enumerate}[$\bullet$]
	\item 
		We
		introduce new stochastic particle 
		systems in discrete and continuous inhomogeneous
		space generalizing the well-known
		Totally Asymmetric Simple Exclusion Process
		(TASEP), 
		and express their observables 
		(with arbitrary inhomogeneity)
		through
		Schur measures, 
		a widely used tool for getting asymptotic fluctuations
		in a variety of stochastic systems in one and two spatial dimensions;
	\item 
		In a continuous space system which we call the 
		\emph{continuous space TASEP}, we study the effect of 
		spatial inhomogeneity on the fluctuation distribution around the traffic jam,
		and obtain a phase transition of a novel type. 
\end{enumerate}

We begin by recalling the original TASEP,
and in the next subsection define its extension which gives rise
to new exactly solvable systems in inhomogeneous space.

\medskip

The TASEP is
one of the most studied
nonequilibrium particle systems 
\cite{Spitzer1970},
\cite{krug1991boundary},
\cite{johansson2000shape}, with applications ranging from
protein synthesis 
\cite{macdonald1968bioASEP},
\cite{TASEP_protein_2011}
to traffic modeling \cite{helbing2001trafficSurvey}.
TASEP in discrete time is a Markov process on particle
configurations in $\mathbb{Z}$ 
(with at most one particle per site)
which evolves as follows.
During each discrete time step $T-1\to T$,
every particle flips an independent $p$-coin to decide whether it
wants to jump one
step to the right. 
Suppose the
coin flip for some particle indicates
a jump attempt. 
If the site to the right is vacant, the particle makes
the jump, otherwise it remains in the same 
position.\footnote{The standard continuous time TASEP (likely the version most familiar to the 
reader) is obtained from this discrete time process by 
scaling time by $p^{-1}$ and sending $p\to0$.}
See \Cref{fig:discrete_usual_TASEP} for an illustration.

\begin{figure}[htbp]
	\centering
	\includegraphics[height=140pt]{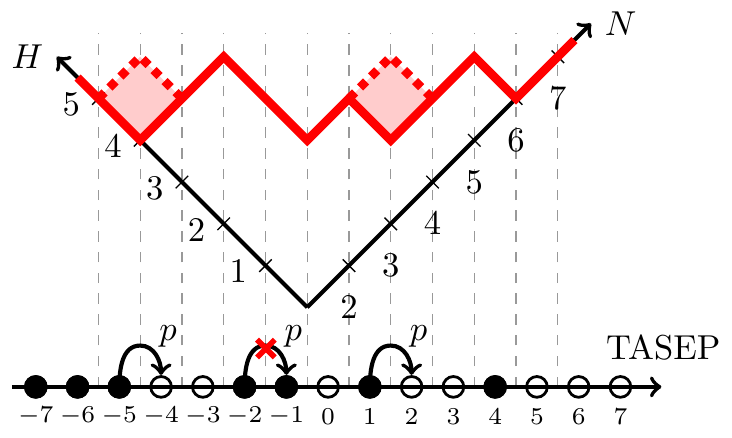}
	\caption{Discrete time TASEP with parallel update
		and its interpretation as a geometric corner growth.
		In this time step three particles make a jump attempt
		but one of them is blocked.}
	\label{fig:discrete_usual_TASEP}
\end{figure}

Start the TASEP from the \emph{step initial configuration}
under which the particles occupy every site of $\mathbb{Z}_{<0}$,
and there are no particles in $\mathbb{Z}_{\ge0}$.
Let $h(T,x)$ be the random \emph{height function} of the TASEP, 
that is, the number of particles to the right of $x\in \mathbb{Z}$ at time $T$.
At the level of Law of Large Numbers,
the height function grows linearly with time,
and its macroscopic shape evolves
according to the hydrodynamic 
equation
\cite{Liggett1985}, 
\cite{spohn1991large},
\cite{Liggett1999}.
The first Central Limit Theorem type result 
on fluctuations of the height functions
was obtained about two decades ago:
\begin{theorem}[\cite{johansson2000shape}]
	\label{thm:Johansson_intro}
	There exist functions $c_1(\kappa), c_2(\kappa)$ such that
	\begin{equation*}
		\lim_{T\to\infty}
		\mathop{\mathrm{Prob}}
		\left( \frac{h(T,\lfloor\kappa T\rfloor)-c_1(\kappa)T}{c_2(\kappa)T^{1/3}}>-r \right)
		=
		F_{GUE}(r),\qquad r \in \mathbb{R},
	\end{equation*}
	where $F_{GUE}$ is the GUE Tracy-Widom distribution 
	\cite{tracy_widom1994level_airy}.
\end{theorem}

In particular, 
TASEP fluctuations live on the
on $T^{1/3}$ scale, in contrast with
the $T^{1/2}$ scale observed in probabilistic systems
based on sums of independent random variables.
This result puts TASEP into the Kardar-Parisi-Zhang (KPZ) universality class
\cite{FerrariSpohnHandbook},
\cite{CorwinKPZ},
\cite{halpin2015kpzCocktail},
\cite{QuastelSpohnKPZ2015},
\cite{Corwin2016Notices}.

\medskip

There has been much development in further understanding the asymptotic
behavior of TASEP and related models, including 
effects of different initial conditions
and different particle speeds
\cite{Its-Tracy-Widom-2001},
\cite{Gravner-Tracy-Widom-2002a},
\cite{PhahoferSpohn2002},
\cite{imamura2005polynuclear},
\cite{BorodinFPS2007},
\cite{BorodinEtAl2009TwoSpeed}, 
\cite{matetski2017kpz}.
Much of this work relies on 
exact solvability of TASEP
which is powered by the algebraic structure
of Schur measures and processes
\cite{okounkov2001infinite},
\cite{okounkov2003correlation}.
An extension of \Cref{thm:Johansson_intro}
to ASEP (in which particles can jump in both directions)
was
proved 
a decade ago
in the pioneering work of Tracy and Widom 
\cite{TW_ASEP2}. This
has brought new exciting tools
of
Macdonald polynomials,
Bethe Ansatz, and Yang-Baxter equation into
the study of stochastic interacting particle systems
\cite{BorodinCorwin2011Macdonald},
\cite{BorodinPetrov2016_Hom_Lectures}.

\medskip

One important aspect of TASEP asymptotics
that has been quite hard to understand 
deals with running TASEP in \emph{inhomogeneous space}.
By this we mean that
each particle's jumping probability $p=p_x$
depends on the particle's current location $x$.
For the inhomogeneous space TASEP the 
exact solvability (connections to Schur measures and processes or Bethe Ansatz)
seems to break down.
Recent progress has been made 
in a particular case of the \emph{slow bond TASEP}.
Namely, if $p_x=1$ everywhere except
$p_{0}=1-\varepsilon$, then for any $\varepsilon>0$
the macroscopic speed of the TASEP at $0$ decreases
\cite{Basuetal2014_slowbond}
(see also the previous works \cite{janowsky1992slow_bond},
\cite{seppalainen2001slow_bond},
\cite{costin2012blockage}).
A Central Limit Theorem for $T^{1/2}$ Gaussian
fluctuations in the slow bond TASEP was established in
\cite{basu2017invariant}.

\subsection{Doubly geometric corner growth in discrete space}
\label{sub:DGCG_intro}

Let us reinterpret the TASEP with step initial configuration described above 
as a geometric corner growth model.
The corner growth is a discrete time Markov process on the space of 
weakly decreasing
\emph{height functions} (or \emph{interfaces}) $H\colon \mathbb{Z}_{\ge1}\to\mathbb{Z}_{\ge0}$
such that $H(1)=+\infty$ and $H(N)=0$ for large enough $N$.
Initially, we have $H_0(N)=0$ for all $N\ge2$,
and
at each discrete time step we independently add a $1\times 1$
box to every
inner corner of the interface with probability $p$.
Adding a box corresponds to a jump of one particle in the TASEP.
See \Cref{fig:discrete_usual_TASEP},
where the interface is rotated by $45^\circ$ to match
with the particle system.

We are now in a position to 
describe an inhomogeneous extension of TASEP in this corner growth language,
after specifying the parameter families.

\begin{definition}[Discrete parameters]
	\label{def:discrete_parameters}
	The discrete systems we consider depend on the following 
	parameters:
	\begin{equation}\label{eq:discrete_parameters}
		\begin{array}{ll}
			a_i\in(0,+\infty), &\quad i=1,2,\ldots; 
			\\
			\beta_t\in(0,+\infty),&\quad t=1,2,\ldots ;
			\\
			\nu_j\in\bigl[-\inf_{t\ge 1, i\ge1}(\beta_t a_i),1\bigl),& \quad j=2,3,\ldots .
		\end{array}
	\end{equation}
	The parameters in each of the families are assumed to 
	be uniformly bounded away from 
	the open boundaries of the corresponding intervals.\footnote{Throughout
	most of the paper the parameters $\nu_j$ are additionally assumed nonnegative,
	but the DGCG model makes sense under the weaker restrictions
	$\nu_j+\beta_ta_i\ge0$ for all $i,t,j$.}
\end{definition}

The \emph{doubly geometric inhomogeneous corner growth model} (DGCG, for short)
is, by definition, a discrete time Markov chain $H_T(N)$
on the space of height functions, where $N$ is the spatial variable and $T$ means discrete time.
\begin{figure}[htpb]
	\centering
	\includegraphics[width=.35\textwidth]{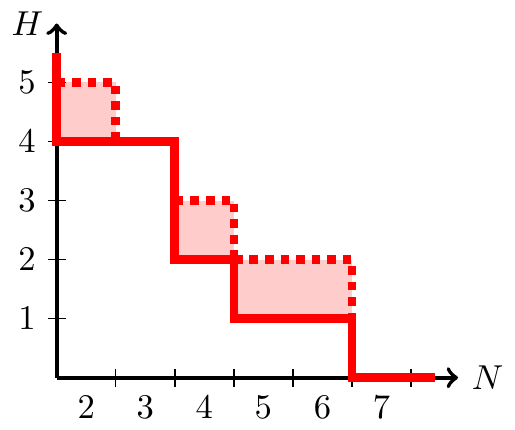}
	\caption{A possible step in DGCG.
	The inner corners before the step are at locations $2,4,5$, and~$7$.}
	\label{fig:time_dep_corner}
\end{figure}

The random growth proceeds as follows.
Let $2=N_1<\ldots<N_k$ be all \emph{inner corners} of $H_T$, 
i.e., all locations at which $H_T(N_i-1)>H_T(N_i)$. 
During the time step $T\to T+1$,
at every inner corner $N_i$
we independently add a $1\times 1$
box (i.e., increase the interface at $N_i$ by one) with probability 
\begin{equation}\label{eq:P_add_box_independently}
	\mathop{\mathrm{Prob}}\left( \textnormal{add a box at inner corner $N_i$ at step $T\to T+1$} \right)
	=
	\frac{\beta_{t+1} a_{N_i-1}}{1+\beta_{t+1} a_{N_i-1}}.
\end{equation}
If a box at $N_{i}$ is added, we also instantaneously add an independent random number 
$\le N_{i+1}-N_i-1$ (with $N_{k+1}=+\infty$, by agreement)
of boxes 
to the right of it according to the truncated inhomogeneous
geometric distribution
\begin{multline}
	\label{eq:add_boxes_p_inhom_geom_distr_parameter}
	\mathop{\mathrm{Prob}}
	\left( \textnormal{add $0\le m\le N_{i+1}-N_i-1$ more boxes} \right)
	\\=
	\begin{cases}
		\mathsf{p}(0)\mathsf{p}(1)\ldots 
		\mathsf{p}(m-1)\bigl( 1-\mathsf{p}(m) \bigr),& 0\le m<N_{i+1}-N_i-1;\\ 
		\mathsf{p}(0)\mathsf{p}(1)\ldots \mathsf{p}(N_{i+1}-N_i-2),& m=N_{i+1}-N_i-1,\\ 
	\end{cases}
\end{multline}
where 
\begin{equation}
	\label{eq:p_inhom_geom_distr_parameter}
	\mathsf{p}(r)=p_{T+1,N_i}(r):=\frac{\nu_{r+N_i}+\beta_{t+1} a_{r+N_i}}{1+\beta_{t+1} a_{r+N_i}}
\end{equation}
(note that this quantity is nonnegative, as it should be).
See \Cref{fig:time_dep_corner} for an illustration.

\medskip

In the simpler homogeneous case $a_i\equiv 1$, $\beta_t\equiv \beta$, $\nu_j\equiv \nu$
(note that setting $a_i$ to the particular constant $1$ 
does not restrict the generality of the homogeneous model),
the random growth $H_T(N)$ 
uses two independent identically distributed families of geometric random variables
(hence the name 
``doubly geometric corner growth''):
\begin{enumerate}[$\bullet$]
	\item 
		A new $1\times 1$ box is added after a 
		geometric waiting time with probability of success $\frac{\beta}{1+\beta}$.
	\item 
		If a box is added, 
		we also instantaneously add an independent random number of 
		$1\times 1$ boxes to the right of the added box according to the 
		truncated geometric
		distribution
		\begin{equation}\label{eq:DGCG_add_boxes_very_intro}
			\mathop{\mathrm{Prob}}
			(\textnormal{add $0\le m\le M$ more boxes})
			=
			\begin{cases}
				\Bigl( \frac{\nu+\beta}{1+\beta} \Bigr)^{m}
				\frac{1-\nu}{1+\beta},& 0\le m < M;\\[12pt]
				\Bigl( \frac{\nu+\beta}{1+\beta} \Bigr)^{M}, & m=M,
			\end{cases}
		\end{equation}
		where $M$ is the maximal number of 
		boxes which can be added without overhanging.
\end{enumerate}
\begin{remark}
	When we formally set $\nu=-\beta$ and 
	$p=\frac{\beta}{1+\beta}$,
	the homogeneous DGCG model becomes the usual TASEP
	(in its geometric corner growth formulation).
	Indeed, for $\nu=-\beta$ no extra boxes are instantaneously added to the randomly
	growing interface.
	In \Cref{sec:hom_DGCG_asymp}
	we discuss the relation between 
	the limit shape of the usual 
	geometric corner growth and the homogeneous DGCG model.
\end{remark}

The homogeneous DGCG
was suggested in 
\cite{derbyshev2012totally},
\cite{Povolotsky2013}
and further studied (on a ring) in
\cite{povolotsky2015gen_tasep}.
Similar tandem queuing and first-passage percolation 
models also appeared earlier in \cite{woelki2005steady}, \cite{Matrin-batch-2009}.

\subsection{Continuous space TASEP}
\label{sub:cont_space_intro}

Let us now describe our second and main model, the \emph{continuous space TASEP}.
It is a continuous time Markov process on the space of finite particle configurations
in $\mathbb{R}_{>0}$.
The particles are ordered, and the process preserves the ordering.
More than one particle per site is allowed,
and one should think that particles at the same site
form a vertical stack (consisting of~$\ge1$ particles). 
It is convenient to think that there
is an infinite stack of particles at location $0$.

The process depends on a speed function $\xi(y)$, $y\in \mathbb{R}_{\ge0}$,
which is assumed positive, piecewise continuous, 
and bounded away from $0$ and $+\infty$. We also need a 
scale parameter $L>0$ which will later go to infinity.
The process evolves as follows:

\begin{definition}[Evolution of the continuous space TASEP]
	\label{def:DGCG_very_intro}
	New particles leave the infinite stack at $0$ at 
	rate\footnote{We say
	that a certain event has rate $\mu>0$ if it repeats after independent random
	time intervals
	which have exponential distribution with rate $\mu$ (and mean
	$\mu^{-1}$).}
	$\xi(0)$.
	If there are particles in a stack located at
	$x\in \mathbb{R}_{>0}$, then
	one particle may (independently) decide to leave this stack at rate $\xi(x)$.
	Almost
	surely at each
	moment in time only one particle can start moving.
	Finally, the moving particle instantaneously jumps 
	to the right by a
	random distance
	$\min(Y,x^{(r)}-x)$, 
	where $Y$ is an independent
	exponential random variable with mean $1/L$,
	and $x^{(r)}$ is the coordinate of the nearest 
	stack to the right of the one at $x$ ($x^{(r)}=+\infty$
	if there are no stacks to the right of $x$).
	In other words, if the desired moving distance 
	is too large, then the moving particle joins the stack immediately to the 
	right of its old location.

	See \Cref{fig:cont_TASEP_pic_intro} for an illustration.
\end{definition}

\begin{figure}[htbp]
	\centering
	\includegraphics[width=.8\textwidth]{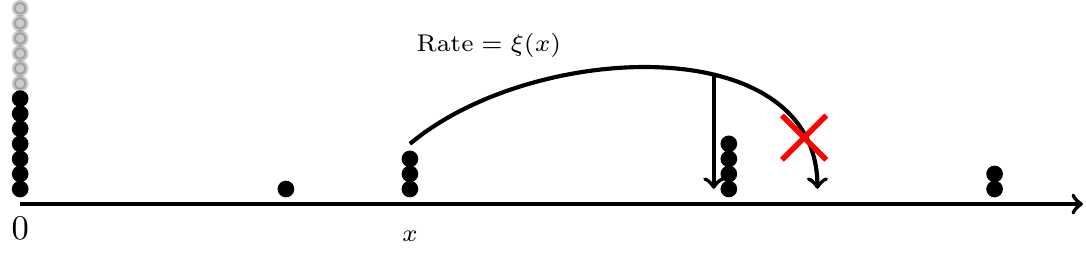}
	\caption{A possible jump in the
	continuous space TASEP.}
	\label{fig:cont_TASEP_pic_intro}
\end{figure}

The continuous space TASEP arises from the DGCG in a certain 
Poisson type limit transition which preserves 
exact solvability.
We study asymptotic behavior of 
the continuous space TASEP in an arbitrary landscape
described by the function
$\xi(\cdot)$.
We obtain the 
limit shape
and investigate 
fluctuations
and phase transitions at 
points of discontinuous decrease in $\xi$.
These points 
can be interpreted as 
traffic accidents, road work, 
or drastic changes in the landscape,
and may lead to traffic jams.
By a traffic jam we mean the presence of a large number of particles in a small 
interval, which corresponds to a discontinuity of the macroscopic
height function.

\begin{remark}
	It is possible
	to add obstacles of another type to the continuous space TASEP.
	These are fixed sites $b\in \mathbb{R}_{>0}$
	(interpreted as traffic lights or roadblocks)
	which with some positive probability
	capture particles flying over them (precise definition in \Cref{sub:cont_space_scaling}).
	Roadblocks
	may create shocks of 
	Baik-Ben
	Arous-P\'ech\'e type. 
	The corresponding asymptotic results
	are given in \Cref{sec:asymptotic_results}.
\end{remark}

\subsection{Results}

Let $H_T(N)$ be the height function (=~interface) of DGCG with the 
initial condition $H_0(N)=0$ for $N\ge2$.
In the continuous space TASEP, let 
$\mathcal{H}(t,\chi)$
count the number of particles to the right of the location $\chi$
at time $t$ (when initially the line $\mathbb{R}_{>0}$
has no particles). The first main result
of the paper connect both families of random variables
$\left\{ H_T(N) \right\}_{T}$
and $\left\{ \mathcal{H}(t,\chi) \right\}_{t}$
(for fixed $N$ and $\chi$, respectively)
to determinantal processes. 
In particular, the joint distribution
of $\left\{ H_{T_j}(N+1) \right\}$
coincides with the joint distribution of the 
leftmost points in a certain Schur process
depending on the parameters $a_1,\ldots,a_N $, 
$\left\{ \beta_t \right\}$, and $\nu_2,\ldots,\nu_N $.
The determinantal structure of the continuous space TASEP's
height function $\left\{ \mathcal{H}(t,\chi) \right\}_{t}$
is obtained as a limit from the DGCG case.
See \Cref{sub:new_Schur_connection_discrete_system_determinantal_structure,sub:new_cont_TASEP} for detailed formulations of structural results.

\medskip

Our second group of results concern asymptotic analysis.
Using the determinantal structure, we 
investigate the asymptotic behavior of the 
continuous space TASEP, that is, study 
$\mathcal{H}(\theta L,\chi)$ as $L\to\infty$
and the speed function $\xi(\cdot)$ is fixed
(there is no need to scale the continuous space).
Our
asymptotic results are the following:
\begin{enumerate}[$\bullet$]
\item
	(Law of Large Numbers;
	\Cref{thm:continuous_intro_limit_shape_theorem}) 
		We show that there exists a deterministic
		limit (in probability)
		of the 
		rescaled height function
		$L^{-1}\mathcal{H}(\theta L,\chi)$ as $L\to+\infty$.
		The limit shapes is a Legendre dual of an explicit function
		involving an integral over the inhomogeneous space.

	\item (Hydrodynamic equations; \Cref{sec:app_hydrodynamics})
		We present informal derivations of
		hydrodynamic partial differential equations 
		for the limiting densities in DGCG and the continuous space TASEP.
		This is done by constructing families of local translation invariant
		stationary distributions of arbitrary density, and computing the 
		flux (also called current) of particles.

	\item (Central Limit type Theorem; \Cref{thm:cont_TASEP_fluctuations_in_all_regimes_intro}) 
		We show that generically
		the fluctuations of 
		the height function around the limit shape
		are of order $L^{1/3}$
		and are governed by the GUE Tracy-Widom distribution
		as in \Cref{thm:Johansson_intro}.
		We also consider the corresponding fluctuations at a single location 
		and different times, leading to the Airy$_2$ process.
		In the presence of shocks caused by 
		roadblocks we observe a phase transition of 
		Baik-Ben Arous-P\'ech\'e type.

	\item (Fluctuations in traffic jams;
		\Cref{thm:traffic_jam_deformed_GUE}) 
		The most striking
		feature of our asymptotic results is a phase
		transition of a new type in the continuous space TASEP.
		Namely, there is a
		transition in fluctuation distribution
		as one approaches a 
		point of 
		discontinuous decrease in the speed function $\xi(\cdot)$
		from the right. 
		There is a critical distance from the jump discontinuity of $\xi(\cdot)$
		at which the fluctuations are 
		governed by a deformation of the GUE Tracy-Widom distribution.
		This deformation can in principle be also obtained in a limit of
		an inhomogeneous
		last-passage percolation, or in a multiparameter Wishart-like random 
		matrix model. Both models were considered in \cite{BorodinPeche2009},
		and our kernel for the deformed GUE Tracy-Widom distribution
		is a particular case of formula~(6) in that paper.
\end{enumerate}

We leave a detailed investigation of the DGCG model
(including phase transitions in fluctuations)
for a future work. Here we only consider the homogeneous DGCG
which depends on two parameters $\beta>0$ and $\nu\in[-\beta,1)$
and is a one-parameter extension of the 
standard corner growth model. 
We show (\Cref{sec:hom_DGCG_asymp})
that the limit shape in the homogeneous DGCG is a one-parameter
deformation of the corner growth's limit shape parabola,
and obtain the corresponding GUE Tracy-Widom fluctuations.

\subsection{Methods}

Since the seminal 
works
\cite{baik1999distribution},
\cite{johansson2000shape},
\cite{Borodin2000b},
\cite{okounkov2001infinite},
\cite{okounkov2003correlation}
about two decades ago,
Schur measures and processes proved
to be a very successful tool in the asymptotic analysis of a large
class of interacting particle systems and 
models of statistical mechanics. 
These methods of Integrable Probability 
also serve as our
main analytic tool.
However, the connection between the models we
consider and Schur processes is not that apparent. 
We consider 
establishing and utilizing this connection 
an important part of the paper. 
From this point of view, DGCG and continuous TASEP extend the
field of classical models solved by means of Schur functions.

Curiously, it became possible to find this connection to Schur processes
only due to recent
developments in the study of stochastic higher spin six
vertex models. Namely, 
the 
continuous space TASEP is a $q\searrow 0$ degeneration
of the inhomogeneous exponential jump model studied in
\cite{BorodinPetrov2016Exp}. The methods used in that paper 
involved computing $q$-moments of the height function of the model,
and break down for $q=0$ (see \Cref{sub:compare_q} below for more detail). 
Here we apply a different approach
based on a nontrivial coupling
\cite{OrrPetrov2016}
between 
the stochastic higher spin six vertex model
and $q$-Whittaker measures and processes. 
This coupling survives passing to the $q\searrow0$ limit
and produces a coupling between DGCG and Schur processes,
which
circumvents the issue of not having observables
of $q$-moment type for $q=0$.
Moreover, at $q=0$ the $q$-Whittaker processes
turn into the Schur ones which possess determinantal structure
\cite{okounkov2001infinite}, \cite{okounkov2003correlation}.

The passage from DGCG to the continuous space TASEP 
preserves the determinantal structure coming from the Schur measures.
The determinantal process associated with the continuous
space TASEP lives on infinite
particle configurations
and depends on the arbitrary speed function $\xi(\cdot)$.
In particular cases this 
limit transition has
appeared in 
\cite{borodin2007asymptotics},
\cite{borodinDuits2011GUE},
\cite{BO2016_ASEP}.
In full generality this limit of Schur measures and processes seems new.

To obtain our asymptotic results, 
we perform analysis of the 
correlation kernel (written in a double contour integral form)
by the steepest descent method. Because of the 
presence of inhomogeneity parameters
in the kernel, the steepest descent analysis
requires several difficult technical estimates.

We also note that
using the determinantal methods 
of Schur measures and processes
we are able to analyze the asymptotic behavior of 
joint distributions of the height function at different times
(of either DGCG or the continuous space TASEP)
at a single location. 
It is interesting that the Schur structure
we employ
does not cover joint 
distributions at several space locations
(see \Cref{sub:compare_q} for more discussion).
A companion paper 
\cite{Petrov2017push} 
deals with a simpler model in inhomogeneous space
in which an analysis of certain
joint distributions across space and time
is possible.

\subsection{Equivalent formulations}
\label{sub:equiv_intro}

Both the DGCG and the continuous space TASEP 
possess a number of equivalent formulations
and interpretations
most of which mimic equivalences known for the usual TASEP.

The doubly geometric corner growth model 
has the following interpretations:
\begin{enumerate}[$\bullet$]
	\item A corner growth model, the original definition
		in \Cref{sub:DGCG_intro};
	\item 
		A generalization of the classical TASEP 
		from \Cref{sub:tasep_intro}
		in which the
		jumping distance of each particle is the product of 
		independent
		Bernoulli and the geometric random variables:\footnote{Throughout 
			the paper $\mathbf{1}_{A}$ stands for the
		indicator of an event $A$. By $\mathbf{1}$ (without subscripts) we will also
		mean the identity operator.}
		\begin{equation}
			\mathop{\mathrm{Prob}}(j)=
			\frac{\mathbf{1}_{j=0}}{1+\beta}
			+
			\frac{\beta\,\mathbf{1}_{j\ge1}}{1+\beta}
			\left( \frac{\nu+\beta}{1+\beta} \right)^{j-1}
			\frac{1-\nu}{1+\beta}
			,
			\qquad 
			j\in \mathbb{Z}_{\ge0}.
			\label{eq:gB_distribution_intro}
		\end{equation}
		Jumping over the particle to the right is forbidden.
		See 
		\Cref{fig:gB_TASEP} for an illustration,
		and
		\Cref{sub:gb_TASEP} for more detail.
		We call \eqref{eq:gB_distribution_intro} the 
		geometric-Bernoulli distribution (or \emph{gB distribution}, for short).

		\begin{figure}[htbp]
			\centering
			\includegraphics[height=140pt]{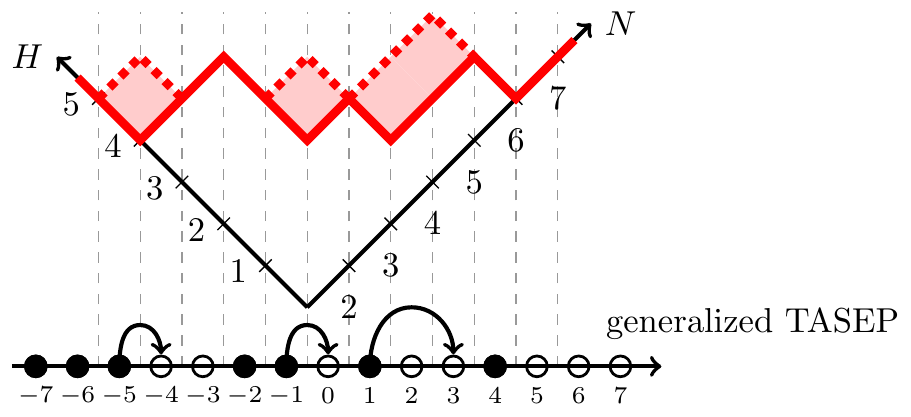}
			\caption{DGCG model and its matching to a generalization of TASEP
				which we call the gB-TASEP.}
			\label{fig:gB_TASEP}
		\end{figure}

	\item 
		Via the exclusion/zero range duality
		(essentially, by looking at the growing DGCG
		interface in the $(H,N)$ coordinates)
		the DGCG can be interpreted as a zero range process
		with the gB hopping distribution.
	\item 
		A directed last-passage percolation model
		with a random environment type modification
		(\Cref{sub:LPP_formulation}).
	\item 
		A directed first-passage percolation model
		on a strict-weak lattice with 
		independent gB distributed weights
		(\Cref{sub:gb_RSK}).
		This interpretation is closely related to 
		applying the column Robinson-Schensted-Knuth (RSK) correspondence
		to a random matrix with independent gB distributed entries.
		Limit shapes for this (homogeneous)
		model were considered previously
		in \cite{Matrin-batch-2009}.
	\item 
		A free fermion type degeneration of the stochastic higher spin 
		six vertex model 
		studied in, e.g., 
		\cite{Borodin2014vertex},
		\cite{CorwinPetrov2015},
		\cite{BorodinPetrov2016inhom}.
	\item 
		Via a coupling of
		\cite{OrrPetrov2016}, 
		certain observables of the (free fermion degenerate)
		stochastic higher spin six vertex model
		are mapped to those in a TASEP with time-mixed geometric and Bernoulli
		steps. The latter is directly linked to Schur processes providing
		a crucial ingredient for exact solvability of the DGCG.
\end{enumerate}
The last two interpretations are
explained in \Cref{sec:models},
and are crucially employed 
in the 
proof of the determinantal structure of DGCG and continuous 
space TASEP in \Cref{sec:det_structure}.

\medskip

In the limit to the continuous space TASEP
the first-passage percolation model coming out of DGCG 
turns into a semi-discrete directed first-passage percolation,
with a modification that each point of a Poisson process
has an additional independent exponential weight.
See \Cref{sub:contin_FPP_RSK}.

Moreover, the continuous space TASEP has a natural formulation
as a continuous time
tandem queuing system. 
The jobs (=~particles) enter the system
according to a Poisson clock at $0$. Each 
point of the real line is a server 
with exponential service times (and the rate depends 
on the server's coordinate).
The job processed at one server is sent to the right (according
to an exponential random distance with mean $1/L$)
and either joins the queue at the nearest server on the right,
or forms a new queue.

\subsection{Related work on spatially inhomogeneous systems}

The study of interacting particle systems in inhomogeneous space started with
numerical and hydrodynamic analysis. 
Numerical simulations 
were mainly motivated by applications
to traffic modeling 
\cite{krug1996phase},
\cite{krug1999simulation},
\cite{krug2000phase},
\cite{dong2008understanding}, 
\cite{helbing2001trafficSurvey}.

The hydrodynamic treatment of interacting particle systems
is the main tool of 
their asymptotic analysis 
\cite{Liggett1985}, 
\cite{liggett1976coupling},
\cite{andjel1982invariant},
\cite{Andjel1984},
\cite{spohn1991large},
\cite{Liggett1999}
in the absence of exact formulas.
This technique allows to 
prove the law of large numbers and
write down a macroscopic PDE for the limit shape of the height function.
Hydrodynamic methods have been successfully applied to 
spatially inhomogeneous systems
including TASEP in, e.g.,
\cite{Landim1996hydrodynamics}, 
\cite{seppalainen1999existence},
\cite{rolla2008last},
\cite{Seppalainen_Discont_TASEP_2010},
\cite{calder2015directed}.

Limit shapes of
directed last-passage percolation
in random inhomogeneous environment
have been studied in
\cite{seppalainenKrug1999hydrodynamics} and more recently in
\cite{emrah2016limit},
\cite{CiechGeorgiou2018}.
Other spatially inhomogeneous systems were considered in, e.g.,
\cite{ben1994kinetics},
\cite{thompson2010zero},
\cite{blank2011exclusion}, 
\cite{blank2012discrete},
with focus on condensation/clustering effects
and understanding of phase diagrams.

A stochastic partial differential
equation limit of the spatially inhomogeneous
ASEP was obtained recently in 
\cite{corwin2018spde}. 
This limit regime to 
an SPDE differs from the one we consider since one needs to scale down the ASEP asymmetry,
while we work in a totally asymmetric setting from the beginning.

Rigorously proving
asymptotic results on
fluctuations in interacting particle systems 
in the KPZ universality class
typically require 
exact formulas. 
A first example of such a result is
\Cref{thm:Johansson_intro}
of \cite{johansson2000shape} which essentially utilizes 
Schur measures. 
In the presence of spatial inhomogeneity, however, 
integrable structures in systems like TASEP break down.
In fact, the understanding of asymptotic fluctuations remains a challenge for
most spatially inhomogeneous systems
in the KPZ class.
An exception is the Gaussian fluctuation behavior in the slow bond TASEP
established recently in \cite{basu2017invariant}.
In contrast, inserting particle-dependent inhomogeneity parameters
(i.e., when particles have different speeds)
preserves most of the structure which allows
to get asymptotic fluctuations,
e.g., see
\cite{BaikBBPTASEP}, 
\cite{BorodinEtAl2009TwoSpeed}, 
\cite{Duits2011GFF}, 
\cite{barraquand2015phase}.

In principle, the $(\textnormal{time})^{1/3}$
scale of fluctuations in certain spatially 
inhomogeneous zero range processes 
may be established as in 
\cite{Balasz_Komjathy_Seppalainen},
but this does not give access to fluctuation distributions.
The previous work
\cite{BorodinPetrov2016Exp}
is a first example of rigorous fluctuation
asymptotics (to the point of establishing Tracy-Widom fluctuation distributions)
in a spatially inhomogeneous TASEP-like particle system
(which is a $q$-deformation of our continuous space TASEP).
The present work improves on the results of 
\cite{BorodinPetrov2016Exp}
by treating joint fluctuations in the $q=0$ system
and looking at fluctuations close to traffic jams.
Overall, in this paper we
explore a whole new family of natural exactly solvable 
systems with spatial inhomogeneity.

\subsection{Outline}
The paper is organized as follows.
In \Cref{sec:models}
we describe how the DGCG model is related to a (free fermion) stochastic higher spin six
vertex model, and get the continuous space TASEP as a Poisson-type limit of DGCG.
We also recall the (degeneration of) the 
result of \cite{OrrPetrov2016}
linking the stochastic vertex model to a TASEP with mixed geometric and Bernoulli steps. 
In \Cref{sec:det_structure} we 
show how the latter connection leads to a determinantal structure
in both the DGCG and continuous space TASEP models. 
In \Cref{sec:asymptotic_results} we formulate the asymptotic
results about the continuous space TASEP and the homogeneous DGCG, 
and prove them in \Cref{sec:asymptotics}. 
In \Cref{sec:hom_DGCG_asymp} we discuss the homogeneous
version of the DGCG model, obtain its limit shape and fluctuations,
and show that they present a one-parameter extension of the 
celebrated geometric corner growth model.

In \Cref{sec:app_B_combinatorics} we discuss in detail a
number of equivalent combinatorial formulations of the DGCG 
and the continuous space TASEP.
\Cref{sec:app_hydrodynamics} presents informal derivations of hydrodynamic
partial differential equations. 
\Cref{sec:app_B1_TW_etc_distributions} contains the definitions of
various fluctuation kernels appearing in the paper.

\paragraph{Notation.}
Throughout the paper
$C,C_i,c,c_j$ stand for positive constants 
which are independent of the main asymptotic parameter $L\to+\infty$.
The values of the constants might change from line to line.

\section{Stochastic vertex models and particle systems}
\label{sec:models}

Here we explain how the DGCG and continuous space TASEP
defined in \Cref{sub:DGCG_intro,sub:cont_space_intro}
are related to a certain stochastic vertex model.
Joint distributions of the height function in the latter model 
are coupled to a TASEP with time-mixed geometric and Bernoulli steps
via results of \cite{OrrPetrov2016} which we also recall.

\subsection{Schur vertex model}
\label{sub:Schur_vertex_model}

We begin by describing a \emph{stochastic vertex model}
whose height function coincides with the DGCG interface $H_N(T)$.
Both models depend on the parameters $a_i,\nu_j,\beta_t$
from \Cref{def:discrete_parameters}.

First we recall a $q$-dependent inhomogeneous stochastic higher spin six vertex
introduced in \cite{BorodinPetrov2016inhom}. 
We follow the notation of \cite{OrrPetrov2016} with the agreement that 
the parameters $u_i$ in the latter paper are expressed through our parameters as
$u_t\equiv -\beta_t>0$, $t=1,2,\ldots $.
The stochastic higher spin six vertex model is a probability distribution on the set of
infinite oriented up-right paths drawn in $(N,T)\in\mathbb{Z}_{\ge 2}\times \mathbb{Z}_{\ge1}$, 
with all paths
starting from a left-to-right arrow entering at 
some of the points 
$\{(2,T)\colon T\in\mathbb{Z}_{\ge 1}\}$ on the left boundary.
No paths enter through the bottom boundary.
Paths cannot share horizontal pieces, but common
vertices and vertical pieces are allowed. 
The probability distribution on this
set of paths is constructed in a Markovian way. 
First, we flip independent coins with probability of success 
$a_1\beta_T/(1+a_1\beta_T)$, $t\in \mathbb{Z}_{\ge1}$, 
and for each success start a path at the point $(2,T)$ on the left boundary.

\begin{figure}[htpb]
	\centering
	\includegraphics[width=.3\textwidth]{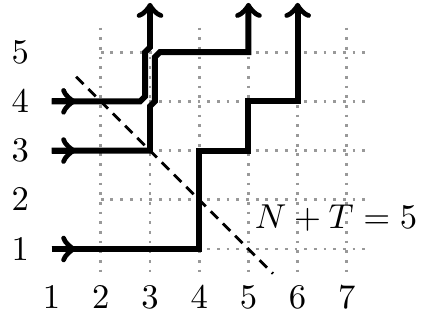}
	\caption{Sampling a path configuration inductively.}
	\label{fig:HS}
\end{figure}

Then, assume that we have already
defined the configuration inside the triangle $\{(N,T)\colon N+T\le n\}$, where
$n\ge2$. For each vertex $(N,T)$ with $N+T=n$, we know the number of incoming
arrows (from below and from the left) into this vertex. Sample, independently
for each such vertex, the number of outgoing arrows according to the stochastic
vertex weights $\mathsf{L}^{(q)}_{a_N,\nu_N,\beta_T}$ 
given in \Cref{def:q_vertex_weights} below.
In this way the path configuration is
now defined inside the larger triangle $\{(N,T)\colon N+T\le n+1\}$, and we can
continue inductively. See \Cref{fig:HS} for an illustration.

\begin{definition}
	\label{def:q_vertex_weights}
	The ($q$-dependent) vertex weights is a 
	collection
	$\mathsf{L}^{(q)}_{a,\nu,\beta}(i_1,j_1;i_2,j_2)$,
	$i_1,i_2\in\mathbb{Z}_{\ge0}$, $j_1,j_2\in\{0,1\}$, 
	where
	$i_1$ and
	$j_1$ are the numbers of arrows entering the vertex, respectively, from below
	and from the left, and $i_2$ and $j_2$ are the numbers of arrows leaving the
	vertex, respectively, upwards and to the right. 
	The concrete expressions for $\mathsf{L}^{(q)}_{a,\nu,\beta}$ are given in the 
	following table:
	\begin{equation*}
		\begin{tabular}{c|c|c|c|c}
		&
		\includegraphics[width=.14\textwidth]{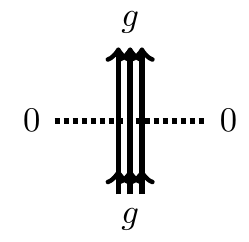}
		&
		\includegraphics[width=.14\textwidth]{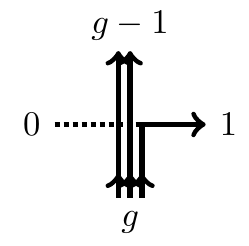}
		&
		\includegraphics[width=.14\textwidth]{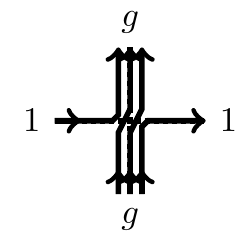}
		&
		\includegraphics[width=.14\textwidth]{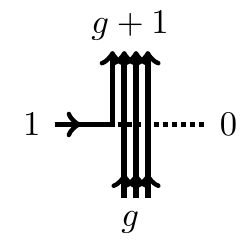}
		\\
		\hline\rule{0pt}{20pt}
		$\mathsf{L}^{(q)}_{a,\nu,\beta}$&
		$\dfrac{1+a\beta q^{g}}{1+a\beta}$&
		$\dfrac{a\beta(1-q^{g})}{1+a\beta}$&
		$\dfrac{\nu q^{g}+a\beta}{1+a\beta}$&
		$\dfrac{1-\nu q^{g}}{1+a\beta}$
		\phantom{\Bigg|}
		\end{tabular}
	\end{equation*}
	\\
	Here
	$g\in\mathbb{Z}_{\ge0}$ is arbitrary. Note that the weight automatically vanishes
	at the forbidden configuration $(0,0;-1,1)$.

	We impose 
	the arrow
	preservation property:
	$\mathsf{L}^{(q)}_{a,\nu,\beta}(i_1,j_1;i_2,j_2)$
	vanishes unless $i_1+j_1=i_2+j_2$ (i.e., the number of outgoing arrows is
	the same as the number of incoming ones). 
	Moreover, the weights are stochastic:
	\begin{equation}
		\label{eq:L_stochasticity_condition}
		\sum_{i_2,j_2\in\mathbb{Z}_{\ge0}\colon i_2+j_2=i_1+j_1}
		\mathsf{L}^{(q)}_{a,\nu,\beta}(i_1,j_1;i_2,j_2)=1,
		\qquad 
		\mathsf{L}^{(q)}_{a,\nu,\beta}(i_1,j_1;i_2,j_2)\ge0.
	\end{equation}
	The nonnegativity of the weights holds if $q\in[0,1)$,
	$a,\beta\in(0,+\infty)$, and $\nu\ge -a \beta$.
	We can thus interpret
	$\mathsf{L}^{(q)}_{a,\nu,\beta}(i_1,j_1;i_2,j_2)$ as a (conditional) probability that there are $i_2$
	and $j_2$ arrows leaving the vertex given that there are $i_1$ and $j_1$
	arrows entering the vertex.
\end{definition}

The weights $\mathsf{L}^{(q)}_{a,\nu,\beta}$ remain stochastic
when setting $q=0$.
The new vertex weights depend on whether $i_1$
is zero or not,
and are given in \Cref{fig:Schur_vertex_weights}.
We call the corresponding stochastic higher spin six vertex model the 
\emph{Schur vertex model} due to its connections with Schur measures
which we explore later.

\begin{figure}[h]
	\centering
	\begin{tabular}{c|c|c|c|c}
	&
	\includegraphics[width=.14\textwidth]{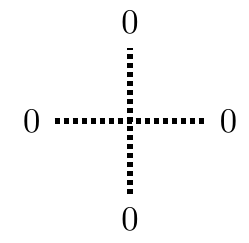}
	&
	\includegraphics[width=.14\textwidth]{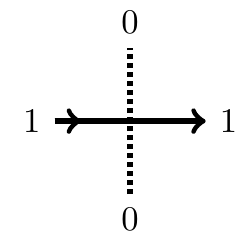}
	&
	\includegraphics[width=.14\textwidth]{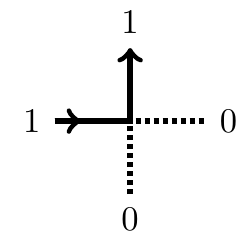}
	\\
	\hline
	\rule{0pt}{20pt}
	$\mathsf{L}^{(q=0)}_{a,\nu,\beta}$
	&$1$
	&$\dfrac{\nu+a\beta}{1+a\beta}$
	&$\dfrac{1-\nu}{1+a\beta}$
	\phantom{\Bigg|}
	\\
	\hline
	&
	\includegraphics[width=.14\textwidth]{vertexoo}
	&
	\includegraphics[width=.14\textwidth]{vertexol}
	&
	\includegraphics[width=.14\textwidth]{vertexll}
	&
	\includegraphics[width=.14\textwidth]{vertexlo}
	\\
	\hline\rule{0pt}{20pt}
	$\mathsf{L}^{(q=0)}_{a,\nu,\beta}$&
	$\dfrac{1}{1+a\beta}$&
	$\dfrac{a\beta}{1+a\beta}$&
	$\dfrac{a\beta}{1+a\beta}$&
	$\dfrac{1}{1+a\beta}$
	\phantom{\Bigg|}
	\end{tabular}
	\caption{The vertex weights for $q=0$. Everywhere in the second row we have $g\ge1$.}
	\label{fig:Schur_vertex_weights}
\end{figure}
\begin{figure}[htbp]
	\centering
	\includegraphics[width=.5\textwidth]{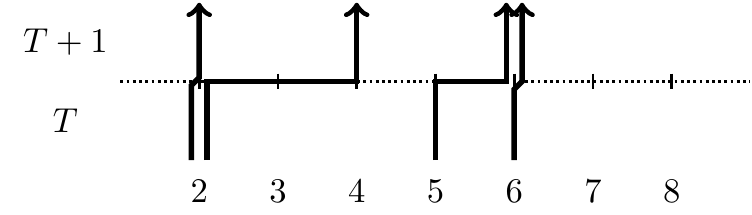}
	\caption{A step $T\to T+1$ in the Schur vertex model viewed as a parallel update.
	The path at $2$ decides to travel by $2$. 
	The path at $5$ starts traveling, but when it reaches $6$ it has to stop. The path at $6$ decides not to travel.}
	\label{fig:traveling_paths}
\end{figure}

One crucial observation regarding the $q=0$ 
weights in \Cref{fig:Schur_vertex_weights}
is that $\mathsf{L}^{(q=0)}_{a,\nu,\beta}(i_1,j_1;i_2,j_2)$
depends on $j_1$ only if $i_1=0$. 
That is, the evolution $T\to T+1$ 
in the Schur vertex can be regarded as a \emph{parallel update}
(for this reason one can say that setting $q=0$ means a ``free fermion'' degeneration).
In particular, each nonempty cluster of paths at each horizontal coordinate $N$
independently decides (with probability $a_N\beta_{T+1}/(1+a_N\beta_{T+1})$)
to emit one path which travels to the right. 
This traveling path then makes a random number of steps to the right, 
at each step deciding to continue or to stop with probabilities corresponding to the vertices
$(0,1;0,1)$ or $(0,1;1,0)$, respectively.
If the path reaches the neighboring cluster of paths on the right, then it has to stop. 
See \Cref{fig:traveling_paths} for an illustration.
This establishes a correspondence between the Schur vertex model
and the DGCG model from
\Cref{sub:DGCG_intro}:

\begin{proposition}
	\label{prop:Schur_vertex_model_corner_growth_correspondence}
	The height function of the $q=0$ vertex model
	\begin{equation*}
		H_T(N)=\#\{\textnormal{paths which are $\ge N$ at vertical coordinate $T$}\}
	\end{equation*}
	is the same as $H_T(N)$ in the DGCG model.
\end{proposition}

\subsection{TASEP with mixed geometric and Bernoulli steps}
\label{sub:mixed_TASEP}

This subsection is essentially a citation (and a $q=0$ degeneration) of \cite{OrrPetrov2016}
mapping the Schur vertex model to a TASEP with mixed steps.
We continue to work with the parameters $a_i,\beta_t,\nu_j$
as in \Cref{def:discrete_parameters}, but in addition require that $\nu_j\ge0$.
In the mixed TASEP, the inhomogeneity is put onto \emph{particles}, not \emph{space}:
each particle $Y_i$ is assigned the parameter $a_i$.

\begin{definition}
	\label{def:geometric_step}
	The \emph{geometric step} with parameter $\alpha>0$ such that $a_i\alpha<1$ for all $i$ 
	applied to a configuration
	$\vec Y=(Y_1>Y_2>\ldots )$ in $\mathbb{Z}$ (with at most particle per site
	and densely packed at~$-\infty$)
	is defined as follows.
	Each particle $Y_j$ with an empty site to the right
	(almost surely there are finitely many such particles at any finite time)
	samples an independent geometric random variable $\mathsf{g}_j$ with distribution
	\begin{equation*}
		\mathop{\mathrm{Prob}}(\mathsf{g}_j=m)=
		(a_j\alpha)^m(1-a_j\alpha)
		,
		\qquad 
		m\in \mathbb{Z}_{\ge0},
	\end{equation*}
	and 
	jumps by $\min(\mathsf{g}_j,Y_{j-1}-Y_{j}-1)$ steps to the right
	(with $Y_0=+\infty$ by agreement).
	See \Cref{fig:gB_TASEP} in the Introduction
	for an illustration of a possible jump (though note that the jump's distribution
	differs from the one in the figure).
	When $\alpha=0$, the geometric step does not change the configuration.
\end{definition}

\begin{definition}
	\label{def:Bernoulli_step}
	Under the \emph{Bernoulli step} with parameter $\beta>0$, the configuration
	$\vec Y$
	is randomly updated as follows.
	First, each particle $Y_j$ tosses an independent coin with probability of success 
	$a_j\beta/(1+a_j\beta)$. Then, sequentially for $j=1,2,\ldots $,
	the particle $Y_j$ jumps to the right by one 
	if its coin is a success and the destination is unoccupied.
	If the coin is a failure or the destination is occupied, the 
	particle $Y_j$ stays put. (The first particle $Y_1$ moves with probability $a_1\beta/(1+a_1\beta)$
	since there are no particles to the right of it.)
	Since the probability of success is strictly less than $1$, 
	the jumps eventually stop because the configuration is densely packed at $-\infty$.

	Note that this Bernoulli step has \emph{sequential update} as opposed to the
	parallel update in the discrete time TASEP
	discussed in \Cref{sub:tasep_intro}.
\end{definition}

\begin{definition}
	\label{def:mixed_TASEP}
	The \emph{mixed TASEP} $\{Y_j(N-1;T)\}$ with 
	parameters 
	$a_i>0$, $\beta_t>0$, $\nu_j\in[0,1)$, and $N\in \mathbb{Z}_{\ge1}$
	is a discrete time Markov process on particle configurations 
	on $\mathbb{Z}$ (with at most one particle per site)
	defined as follows.
	Starts from 
	the step initial configuration
	$Y_j(0;0)=-j$, $j\in \mathbb{Z}_{\ge1}$ and first
	make $N-1$ geometric steps with parameters $\nu_2/a_2,\ldots, \nu_N/a_N$
	(some of these parameters might be zero; the corresponding geometric steps 
	do not change the configuration).
	Let $\vec Y(N-1;0)$ denote the configuration after these geometric steps.
	Then make $T$ Bernoulli steps with parameters $\beta_1,\ldots,\beta_T $, and denote the 
	resulting configuration by 
	$\vec Y(N-1;T)$.
\end{definition}

\begin{theorem}[\cite{OrrPetrov2016}]
	\label{prop:relation_mixed_TASEP_to_higher_spin_corner_growth}
	Fix $N\in \mathbb{Z}_{\ge1}$ and $0\le T_1\le \ldots \le T_{\ell}$.
	We have the following equality of joint distributions
	between the Schur vertex model and the 
	mixed TASEP:
	\begin{equation}\label{eq:relation_mixed_TASEP_to_higher_spin_corner_growth}
		\left\{ H_{T_j}(N) \right\}_{j=1}^{\ell}
		\stackrel{d}{=}
		\left\{ Y_N(N-1;T_j)+N \right\}_{j=1}^{\ell}.
	\end{equation}
\end{theorem}
\begin{proof}
	This follows by setting $q=0$ in Theorem 1.1 (or Theorem 5.9)
	in \cite{OrrPetrov2016}. 
	Note that in contrast with the observables
	of $q$-moment type,
	setting $q=0$ in these equalities in distribution is perfectly justified,
	and leads to the desired result (cf. \Cref{sub:compare_q} below for more discussion).
\end{proof}

Together
\Cref{prop:Schur_vertex_model_corner_growth_correspondence} and 
\Cref{prop:relation_mixed_TASEP_to_higher_spin_corner_growth} 
link the joint
distributions of the DGCG (at a single location and different times)
to those in the mixed TASEP. The latter are known to be certain observables
of Schur processes. 
In this way we see that the DGCG possesses a determinantal structure.
The structure is described in detail in \Cref{sec:det_structure} below.

\subsection{Continuous space TASEP as a limit of DGCG}
\label{sub:cont_space_scaling}

Let us now explain how the DGCG (equivalently, the Schur vertex model)
converges to the continuous space TASEP. We will consider a more general process which includes
roadblocks.
Thus, the continuous space TASEP is a continuous time
Markov process $\{X(t)\}_{t\geq 0}$ on the space
$$
	\mathcal{X}:=\{(x_1\geq x_2\ge\dots \geq x_k>0)\colon x_i\in \mathbb R \text{
		and } k\in
	\mathbb Z_{\geq 0} \text{ is arbitrary}\}
$$
of finite particle configurations on $\mathbb R_{>0}.$ 
The particles are ordered, and the process preserves this ordering.
However, more than one particle per site it allowed.

The Markov process $X(t)$ on $\mathcal{X}$ depends on the following
data:
\begin{enumerate}[$\bullet$]
\item {\it Distance} parameter $L>0$ 
	(going to infinity in our asymptotic regimes);
\item {\it Speed function} $\xi(y),$ $y\in \mathbb R_{\geq 0},$ which is assumed to be positive, piecewise continuous,
have left and right limits, and uniformly bounded away from $0$ and $+\infty;$
\item Discrete set $\RoadblockSet\subset \mathbb R_{>0}$
	(whose elements will be referred to as
	\emph{roadblocks})
	without
	accumulation points 
	such that there are finitely many points of $\RoadblockSet$
	in a right neighborhood of $0$.
	Fix a function $p: \mathbf B\rightarrow (0,1)$.
\end{enumerate}
The process $X(t)$ evolves as follows: 
\begin{enumerate}[$\bullet$]
	\item New particles
	enter $\mathbb R_{>0}$ (leaving $0$) at 
	rate
	$\xi(0);$ 
\item If at some time $t>0$ there are particles at a location $x \in \mathbb{R}_{>0}$,
	then one particle decides to leave this location at rate $\xi(x)$
	(these events occur independently for each occupied location). Almost surely at each
	moment in time only one particle can start moving;
	\item 
		The moving particle (say, $x_j$) instantaneously jumps to the right by some random
		distance $x_j(t)-x_j(t-)=\min(Y, x_{j-1}(t-)-x_j(t-))$ (by agreement, $x_0\equiv+\infty$).
		The distribution of $Y$ is as follows:
		$$
			\mathop{\mathrm{Prob}}
			(Y \geq y )
			=
			e^{-L y}\prod\limits_{b \in \mathcal{\mathbf B},
			\text{ }x_j(t-)<b<x_j(t-)+y} p(b).
		$$
\end{enumerate}
This completes the definition of the continuous space TASEP.
See \Cref{fig:cont_TASEP_pic} for an illustration. 

\begin{figure}[htbp]
	\centering
	\includegraphics[width=.8\textwidth]{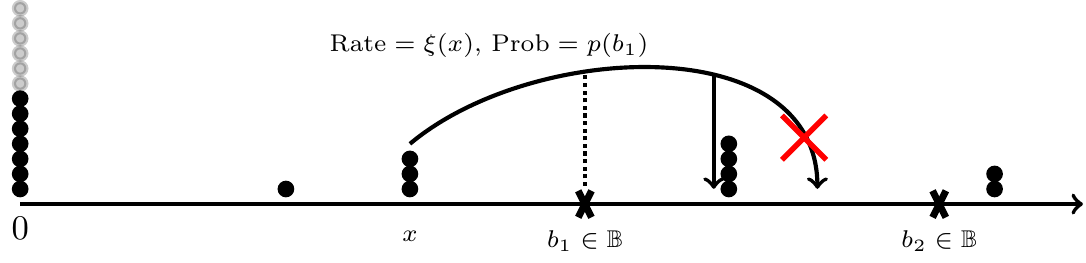}
	\caption{A possible jump in the 
		continuous space TASEP $X(t)$.
		The jump occurs at rate $\xi(x)$.
		The moving particle overcomes the roadblock at $b_1$ with probability $p(b_1)$,
		and joins the next stack
		because the particles preserve order.}
	\label{fig:cont_TASEP_pic}
\end{figure}

\medskip

We define the {\it height function} 
of the process $X(t)$
by
\begin{equation*}
\mathcal{H}(t,\chi):=\#\{\text{particles } x_i \text { at time $t$ such
that } x_i\geq \chi\}.
\end{equation*}
The height function $\mathcal{H}(t,\chi)$ is almost surely weakly decreasing in $\chi\in \mathbb{R}_{>0}$
and $\lim_{\chi\rightarrow +\infty}\mathcal{H}(t,\chi)=0$. 
Additionally,
it is very convenient to assume there are 
infinitely many particles at location~$0$, so that $\mathcal{H}(t,0)\equiv+\infty$.

\medskip

Let us now describe the regime in which the DGCG converges to the continuous space TASEP.
Let $\varepsilon>0$ be a small parameter,
and set $\beta_t=\varepsilon$ for all $t$.
Scale the discrete time and space of the DGCG~as 
\begin{equation*}
	T=\lfloor \varepsilon^{-1}t \rfloor,
	\qquad 
	N=\lfloor \varepsilon^{-1}\chi \rfloor.
\end{equation*}
To define the scaling of the $a_i$'s and the $\nu_j$'s,
denote $\mathbf B^{\varepsilon}=\{
\lfloor \varepsilon^{-1}b \rfloor, b\in \RoadblockSet  \}\subset \mathbb Z_{\geq 1}$.
Set
\begin{equation}
	\label{eq:DGCG_to_cTASEP_scaling_no_roadblocks}
	a_1=\xi(0),
	\qquad a_j=\xi(j\varepsilon), \qquad 
	\nu_j=e^{-L\varepsilon},\qquad j\in \mathbb{Z}_{\ge2}\setminus \RoadblockSet^{\varepsilon},
\end{equation}
and
\begin{equation}
	\label{eq:DGCG_to_cTASEP_scaling_roadblocks}
	a_i=\xi(b),\qquad 
	\nu_i=p(b),\qquad 
	\textnormal{where $i=\lfloor \varepsilon^{-1}b \rfloor $ for
	$b\in \RoadblockSet$}.
\end{equation}
In particular, all $\nu_j$ can be chosen nonnegative, and
$\nu_j\to1$ for almost all $j$.
The roadblocks correspond to the indices $i$ such that $\nu_i <1$.
Note that if $\xi(\cdot)$ is discontinuous at $0$
then the rate at which particles are added to the system from the infinite stack at $0$ is different from 
$\lim_{\chi\to0+}\xi(\chi)$.

\begin{theorem}
	\label{thm:convergence_discrete_to_continuous_TASEP_process_only}
	As $\varepsilon\to0$ under the scalings described above,
	the DGCG height function 
	converges to the one for the continuous space TASEP
	as
	$H_T(N)\to \mathcal{H}(t,\chi)$, in the sense of finite-dimensional
	distributions, jointly for all $(t,\chi)$.
\end{theorem}
\begin{proof}
	First, pass to the Poisson-type continuous time limit
	$\beta_t\equiv \beta\to 0$ in the DGCG, 
	keeping the space and all other parameters $a_i,\nu_j$ intact.
	Interpret this intermediate continuous time DGCG as a particle system
	on $\mathbb{Z}_{\ge1}$, with 
	$H_T(N)-H_T(N-1)$ particles at each $N\ge2$,
	and infinitely many particles at $1$.
	Then new particles are added to the continuous time DGCG
	at rate $\beta^{-1}\frac{a_1 \beta}{1+a_1\beta}=a_1+O(\beta)$
	(see, e.g., the second line of \Cref{fig:Schur_vertex_weights})
	
	Now take the $\varepsilon$-dependent parameters $a_i,\nu_j$
	as above
	in the continuous time DGCG.
	We can 
	couple this DGCG (for all $\varepsilon>0$)
	and continuous space TASEP such that 
	they have the same number of 
	particles at each time. This 
	is possible since particles are added to 
	both systems according to Poisson processes of rate
	$a_1=\xi(0)$.
	This coupling reduces
	the problem to finite particle systems,
	and one readily sees that all transition probabilities 
	in DGCG converge to those in the continuous space TASEP
	(geometric random variables in DGCG become the 
	exponential ones in the definition of the continuous space TASEP).
\end{proof}

\Cref{thm:convergence_discrete_to_continuous_TASEP_process_only} 
thus brings the Schur process type determinantal structure 
from the DGCG to the continuous space TASEP. 

\subsection{Comments}
\label{sub:compare_q}

Let us make two detailed comments 
on the determinantal structure of 
the DGCG and the continuous space TASEP 
which is outlined above (detailed formulations of the 
determinantal structure are given in \Cref{sec:det_structure} below).

\paragraph{Limit as $q\to0$ of previously known formulas.}

First, we compare the existing methods to solve the $q$-deformations 
of the systems considered in the present paper.
In the $q$-deformed setting, 
\cite{CorwinPetrov2015}, \cite{BorodinPetrov2016inhom}, \cite{BorodinPetrov2016Exp}
obtain formulas of two types:
\begin{enumerate}[$\bullet$]
	\item The $q$-moments of the height function
		$\mathop{\mathbb{E}} q^{k H_T^{(q)}(N)}$, $k\in \mathbb{Z}_{\ge1}$
		(where $H_T^{(q)}$ is the height function of the $q$-dependent
		vertex model from \Cref{sub:Schur_vertex_model}),
		are expressed as $k$-fold nested contour integrals of elementary functions
		(for shortness, we do not specify the contours):
		\begin{multline*}
			\mathop{\mathbb{E}} q^{k H_T^{(q)}(N)}
			=
			\frac{q^{k(k-1)/2}}{(2\pi\sqrt{-1})^2}
			\oint\ldots\oint 
			\frac{dw_1 \ldots dw_k }{w_1\ldots w_k}			
			\prod_{1\le i<j\le k}\frac{w_i-w_j}{w_i-qw_j}
			\\\times
			\prod_{r=1}^{k}
			\left( 
				\frac{1}{1-w_r/a_1}
				\prod_{j=2}^{N}
				\frac{a_j-\nu_jw_r}{a_j-w_r}
				\prod_{j=1}^{T}\frac{1+q\beta_j w_r}{1+\beta_jw_r}
			\right).
		\end{multline*}
	\item The $q$-Laplace transform\footnote{Here $(a;q)_{\infty}=(1-a)(1-aq)(1-aq^2)\ldots $
		is the infinite $q$-Pochhammer symbol.}
		$\mathop{\mathbb{E}}\bigr((\zeta q^{H_T^{(q)}(N)};q)_{\infty}\bigl)^{-1}$
		is written as a Fredholm determinant
		$\det(1+K_\zeta^{(q)})$
		of a kernel which itself has a single contour integral representation:
		\begin{equation*}
			K_\zeta^{(q)}(w,w')=\frac{1}{2\pi\sqrt{-1}}
			\int \Gamma(-u)\Gamma(1+u)(-\zeta)^{u}\frac{g(w)}{g(q^u w)}\frac{du}{q^u w-w'},
		\end{equation*}
		where 
		$g(w)$ contains infinite $q$-Pochhammer symbols
		and is such that $g(w_r)/g(qw_r)$ is equal to the $r$-th term in the 
		product in the above $q$-moment formula. Again, to shorten the exposition we do not specify 
		the integration contour in $K_\zeta^{(q)}$ or the space on which this kernel acts.
\end{enumerate}

Both the $q$-moment and the Fredholm determinantal formulas characterize
the distribution of $H_T^{(q)}(N)$ uniquely. 
As $q\to 0$, the height functions $H_T^{(q)}(N)$ converge to the DGCG height function
(denote it by $H_T^{(q=0)}(N)$ in this subsection). However, at $q=0$ both the observables
$\mathop{\mathbb{E}} q^{k H_T^{(q)}(N)}$
and 
$\mathop{\mathbb{E}}\bigr((\zeta q^{H_T^{(q)}(N)};q)_{\infty}\bigl)^{-1}$
provide almost no information about the distribution of $H_T^{(q=0)}(N)$.

In principle, before passing to the $q\to0$
limit, one could invert the $q$-Laplace transform to express the distribution
of $H_T^{(q)}(N)$ in a form which survives the $q\to0$ transition.
This inversion would involve taking an extra contour integral of the Fredholm
determinant $\det(1+K_\zeta^{(q)})$
(e.g., see \cite[Proposition 3.1.1]{BorodinCorwin2011Macdonald}),
and the result would contain $q$ in a very nontrivial manner. 
Instead of passing to the 
$q\to 0$ limit in this rather complicated Fredholm determinant,
we utilize the connection 
of the $q$-dependent vertex model to the 
$q$-Whittaker processes
found in 
\cite{OrrPetrov2016} 
which easily survives
the $q=0$ degeneration.
In this way we relate $H_T^{(q=0)}(N)$ to Schur processes
(which are the $q=0$ limits of the $q$-Whittaker processes),
and then obtain asymptotic results by working with determinantal
processes.

\paragraph{Joint distributions at different space locations.}

Let us now discuss a limitation of the determinantal structure
in describing the joint distributions of the height function
$H_T(N)$ (or $\mathcal{H}(t,\chi)$) across different spatial locations.

The $q=0$ degeneration of the results
of 
\cite{OrrPetrov2016} implies a more
general equality of joint distributions than 
\eqref{eq:relation_mixed_TASEP_to_higher_spin_corner_growth}.
Let us describe the simplest nontrivial example.
The joint distribution of $H_T(N_1)$ and $H_T(N_2)$,
$N_1<N_2$, can be described as follows. 
First, we have $H_T(N_1)\stackrel{d}{=}Y_{N_1}(N_1-1;T)+N_1$, 
where $\vec Y$ is the mixed TASEP from \Cref{def:mixed_TASEP}.
Take the random configuration 
\begin{equation*}
	\vec Y(N_1-1;T)=(Y_{1}(N_1-1;T)>Y_2(N_1-1;T)>\ldots ),
\end{equation*}
and apply to it $N_2-N_1$ additional geometric steps with parameters $\nu_{N_1+1}/a_{N_1+1},
\ldots,\nu_{N_2}/a_{N_2} $. Denote the resulting configuration by 
$\vec Y'$. (In fact, the distribution of $\vec Y'$ coincides with that of 
$\vec Y(N_2-1;T)$ from \Cref{def:mixed_TASEP}, but note that the order of 
geometric and Bernoulli steps in $\vec Y'$ is not the same as in
$\vec Y(N_2-1;T)$.)
Then we have
\begin{equation*}
	\left\{ H_T(N_1),H_T(N_2) \right\}\stackrel{d}{=}
	\left\{ Y_{N_1}(N_1-1;T)+N_1,Y'_{N_2}+N_2 \right\}.
\end{equation*}
The joint distribution in the right-hand side is 
\emph{not} given by a
marginal of a Schur processes.

Joint distributions in TASEP corresponding to increasing both the particle's number and the 
time are known as \emph{time-like} (see, e.g., \cite{derrida1991dynamics}, 
\cite{Ferrari_Airy_Survey} about the terminology). Their asymptotic analysis is typically 
much harder than the one of the \emph{space-like} joint distributions
(which for TASEP are related to marginals of Schur processes).
Asymptotic analysis of two-time 
time-like joint distribution in the last-passage percolation
was performed recently in
\cite{johansson2015two},
\cite{johansson2018two}.
(See also references to related non-rigorous and experimental
work in the latter paper.)
In the present work we do not consider
joint distributions of the height function $H_T(N)$
involving more than one space location.

\section{Determinantal structure via Schur processes}
\label{sec:det_structure}

In this section we derive the 
determinantal structure of the 
DGCG and the continuous space TASEP. First, we recall the Schur processes 
and their determinantal structure
(as applied to our concrete situation).
Then, using 
\Cref{prop:Schur_vertex_model_corner_growth_correspondence} and 
\Cref{prop:relation_mixed_TASEP_to_higher_spin_corner_growth},
we obtain determinantal formulas for the DGCG model.
A limit to continuous space then leads to determinantal formulas
for the continuous space TASEP.
Throughout the section the parameters $a_i,\beta_t,\nu_j$
are assumed to satisfy 
\eqref{eq:discrete_parameters}, with an additional restriction $\nu_j\ge0$.

\subsection{Schur processes}
\label{sub:Schur_processes_new}

\subsubsection{Young diagrams}
\label{appsub:Young_diagrams}

A partition is a nonincreasing integer sequence of the form 
$\lambda=(\lambda_1\ge\ldots\ge\lambda_{\ell(\lambda)}>0)$. The number of nonzero
parts $\ell(\lambda)$ (which must be finite) is called the length of a
partition. Partitions are represented by Young diagrams, such that
$\lambda_1,\lambda_2,\ldots $ denote the lengths of the successive rows. 
The column
lengths of a Young diagram are denoted by $\lambda_1'\ge\lambda_2'\ge\ldots$. 
They form a transposed Young diagram $\lambda'$.
See
\Cref{fig:YD_example}.
The set of all partitions (equivalently, Young diagrams)
is denoted by $\mathbb{Y}$.

\begin{figure}[htpb]
	\centering
	\includegraphics[width=.2\textwidth]{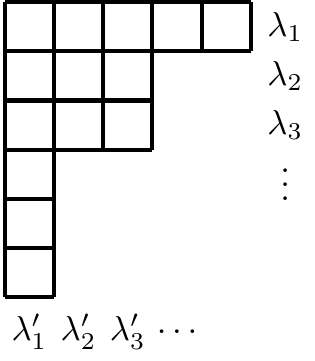}
	\caption{A Young diagram $\lambda=(5,3,3,1,1,1)$
	for which the
	transposed diagram is $\lambda'=(6,3,3,1,1)$.}
	\label{fig:YD_example}
\end{figure}

Let $\mu,\lambda$ be two Young diagrams. We say that $\lambda$ differs from $\mu$ by adding a 
horizontal strip (notation $\mu\prec\lambda$) iff $0\le \lambda_i'-\mu_i'\le 1$ for all $i$.
We say that $\lambda$ differs by $\mu$ by adding a vertical strip
(notation $\mu\prec'\lambda$) iff $\mu'\prec \lambda'$.

\subsubsection{Schur functions}
\label{appsub:Schur_functions}

For each Young diagram $\lambda$, let $s_\lambda$ be the corresponding Schur
symmetric function \cite[Ch. I.3]{Macdonald1995}. Evaluated at $N$
variables $u_1,\ldots,u_N$ (where $N\ge \ell(\lambda)$ is arbitrary),
$s_\lambda$ becomes the symmetric polynomial
\begin{equation}
	\label{eq:schur_def}
	s_\lambda(u_1,\ldots,u_N) =
	\frac{\det[u_i^{\lambda_j+N-j}]_{i,j=1}^N}{\det[u_i^{N-j}]_{i,j=1}^N}
\end{equation}
If $N<\ell(\lambda)$, then $s_\lambda(u_1,\ldots,u_N )=0$ by definition. When
all $u_i\ge0$, the value $s_\lambda(u_1,\ldots,u_N )$ is also nonnegative.
The Schur functions $s_\lambda$ form a linear basis in the algebra of symmetric functions $\Lambda$,
where $\lambda$ runs over all possible Young diagrams.

Along with evaluating Schur functions at finitely many variables, we also need
their general nonnegative specializations. That is, 
a nonnegative specialization is an algebra homomorphism $\rho\colon\Lambda\to\mathbb{C}$
such that $\rho(s_\lambda)\ge0$ for all Young diagrams $\lambda$.
Nonnegative specializations are classified by the Edrei--Thoma
theorem
\cite{Edrei1952}, \cite{Thoma1964}
(also see, e.g., \cite{borodin2016representations}).
They depend on infinitely many real parameters
$\vec\alpha=(\alpha_1\ge \alpha_2\ldots\ge0)$,
$\vec \beta=(\beta_1\ge\beta_2\ge\ldots\ge0)$,
and $\gamma\ge0$, with $\sum_{i}(\alpha_i+\beta_i)<\infty$,
and are determined 
by the Cauchy summation identity
\begin{equation}
	\label{eq:Cauchy}
	\sum_{\lambda\in \mathbb{Y}}
	s_\lambda(y_1,\ldots,y_n )
	\rho_{\vec\alpha,\vec\beta,\gamma}(s_\lambda)=
	\prod_{j=1}^{n}
	\Biggl(
		e^{\gamma y_j}
		\prod_{i=1}^{\infty}
		\frac{1+\beta_i y_j}{1-\alpha_i y_j}
	\Biggr),
\end{equation}
where $n\ge1$ is arbitrary, and $y_1,\ldots,y_n $ are 
regarded as formal variables.
We will write 
$s_\lambda(\vec\alpha;\vec\beta;\gamma)=
s_\lambda(\alpha_1,\alpha_2,\ldots;\beta_1,\beta_2,\ldots;\gamma)$
for $\rho_{\vec\alpha,\vec\beta,\gamma}(s_\lambda)$
and will continue to use notation $s_\lambda(\alpha_1,\ldots,\alpha_m )$
for the substitution of the variables $\alpha_1,\ldots,\alpha_m$ into $s_\lambda$
(which is the same as the specialization with finitely many $\alpha_i$'s and $\vec\beta=\vec0$, $\gamma=0$).

There are also skew Schur symmetric functions $s_{\lambda/\mu}$ which are
defined through
\begin{equation*}
	s_{\lambda}(u_1,\ldots,u_{N+M} )=
	\sum_{\mu\in \mathbb{Y}}s_{\lambda/\mu}(u_1,\ldots,u_N )s_{\mu}(u_{N+1},\ldots,u_{N+M} ).
\end{equation*}
The function $s_{\lambda/\mu}$ vanishes unless the Young diagram $\lambda$ contains
$\mu$ (notation: $\lambda\supset\mu$). 
Skew Schur functions satisfy a skew generalization of the Cauchy summation
identity:
\begin{equation}
	\label{eq:skew_Cauchy}
	\sum_{\nu\in \mathbb{Y}}s_{\nu/\mu}(x)s_{\nu/\lambda}(y)
	=
	\prod_{i,j}\frac{1}{1-x_iy_j}
	\sum_{\kappa\in \mathbb{Y}}
	s_{\mu/\kappa}(y)s_{\lambda/\kappa}(x)
	,
\end{equation}
where $\lambda,\mu$ are fixed and $x,y$ are two sets of variables.
The specializations $s_{\lambda/\mu}(\vec\alpha;\vec\beta;\gamma)$ are well-defined
and produce nonnegative numbers.
The skew Schur functions $s_{\lambda/\mu}(a,0,\ldots;\vec0;0 )$
and $s_{\lambda/\mu}(\vec0;\beta,0,\ldots;0 )$ vanish unless,
respectively,
$\mu\prec\lambda$ and $\mu\prec'\lambda$.
For the specialization with all zeros we have
we have $s_{\lambda/\mu}(\vec0;\vec0;0)=\mathbf{1}_{\lambda=\mu}$.

Taking the one-variable
specializations
$x=(\vec0;\beta,0,\ldots ;0)$ and
$y=(a,0,\ldots;\vec0;0 )$
in
\eqref{eq:skew_Cauchy}, we get the identity
\begin{equation}
	\label{eq:skew_Cauchy_particular}
	\sum_{\nu\in \mathbb{Y}}s_{\nu/\mu}(\vec0;\beta;0)s_{\nu/\lambda}(a;\vec0;0)
	=
	(1+a\beta)
	\sum_{\kappa\in \mathbb{Y}}
	s_{\mu/\kappa}(a;\vec0;0)s_{\lambda/\kappa}(\vec0;\beta;0),
\end{equation}
where $a,\beta\ge0$ are real numbers.
We refer to, e.g.,
\cite[Ch I]{Macdonald1995} for further details on ordinary and skew Schur functions.

\subsubsection{A field of Young diagrams}
\label{appsub:Schur_field}

Recall the discrete parameters $\{a_i\}_{i\ge1}$,
$\{\nu_i\}_{i\ge2}$, and $\{\beta_i\}_{i\ge1}$
(\Cref{def:discrete_parameters}), and fix $N\in\mathbb{Z}_{\ge1}$.
Consider a random field
of Young diagrams, that is, a probability distribution on an
array of Young diagrams 
$\{ \lambda^{(T,K)} \}_{T,K\in \mathbb{Z}_{\ge0}}$ (cf. \Cref{fig:field})
with the following properties:
\begin{enumerate}[\bf1.]
	\item (bottom boundary condition)
		For all $T\ge0$ we have $\lambda^{(T,0)}=\varnothing$.
	\item (left boundary condition)
		For all $M\in \mathbb{Z}_{\ge0}$,
		the joint distribution of the Young diagrams 
		$\lambda^{(0,K)}$, $0\le K\le M$, at the left boundary
		is given by the following ascending Schur process:
		\begin{multline}
			\label{eq:asc_Schur_left_boundary}
			\mathop{\mathrm{Prob}}
			\bigl( \lambda^{(0,0)}, \lambda^{(0,1)},\ldots,\lambda^{(0,M)}  \bigr)
			=
			\frac{1}{Z}
			s_{\lambda^{(0,1)}}(a_1)\\\times
			s_{\lambda^{(0,2)}/\lambda^{(0,1)}}(a_2)
			\ldots
			s_{\lambda^{(0,M)}/\lambda^{(0,M-1)}}(a_M)
			s_{\lambda^{(0,M)}}\Bigl(\frac{\nu_2}{a_2},\ldots,\frac{\nu_N}{a_N} \Bigr)
			,
		\end{multline}
		where $Z$ is the normalizing constant.
		In particular, this implies that 
		along the left edge each two consecutive Young diagrams
		$\lambda^{(0,j)}$ and $\lambda^{(0,j+1)}$ almost surely differ by adding a horizontal strip.
		In particular, $\ell(\lambda^{(0,j)})\le j$.

	\item (conditional distributions)
		For any $i,j\in \mathbb{Z}_{\ge1}$ consider the quadruple of neighboring Young diagrams
		$\kappa=\lambda^{(i-1,j-1)}$, $\lambda=\lambda^{(i,j-1)}$, $\mu=\lambda^{(i-1,j)}$,
		and $\nu=\lambda^{(i,j)}$
		(we use these notations to shorten the formulas; cf. \Cref{fig:quadruple}).
		The conditional distributions in this quadruple are as follows:\footnote{The first
		probabilities in \eqref{eq:quadruple_distributions}
		are conditional over the northwest quadrant with tip $\mu$ and the southeast quadrant
		with tip $\lambda$, and we require that the dependence on these quadrants 
		is only through their tips $\mu$ and $\lambda$, respectively. This
		can be viewed as a type of a two-dimensional Markov property.}
		\begin{equation}
			\label{eq:quadruple_distributions}
			\begin{split}
				&
				\mathop{\mathrm{Prob}}\bigl(\nu\mid \lambda^{(p,q)}\colon \textnormal{$p\le i-1, q\ge j-1$ or $p\ge i-1$, $q\le j-1$}\bigr)
				\\
				&
				\hspace{50pt}=
				\mathop{\mathrm{Prob}}(\nu\mid \lambda,\mu)
				=
				\frac{s_{\nu/\mu}(\vec0;\beta_i,0,\ldots;0 )s_{\nu/\lambda}(a_j,0,\ldots;\vec0;0 )}{Z_u^{(i,j)}}
				;
				\\
				&
				\mathop{\mathrm{Prob}}\bigl(\kappa\mid \lambda^{(p,q)}\colon \textnormal{$p\le i-1, q\ge j-1$ or $p\ge i-1$, $q\le j-1$}\bigr)
				\\
				&
				\hspace{50pt}=
				\mathop{\mathrm{Prob}}(\kappa\mid \lambda,\mu)
				=
				\frac{s_{\mu/\kappa}(a_j,0,\ldots;\vec0;0 )s_{\lambda/\kappa}(\vec0;\beta_j,0,\ldots;0 )}{Z_\ell^{(i,j)}}
				,
			\end{split}
		\end{equation}
		where $Z_u^{(i,j)},Z_\ell^{(i,j)}$ are normalizing constants.
		In particular, $\mu\prec'\nu$, $\kappa\prec'\lambda$,
		$\kappa\prec\mu$, and $\lambda\prec\nu$ almost surely, and 
		this implies that $\ell(\lambda^{(i,j)})\le j$ for all $i,j$.
		The skew Cauchy identity 
		\eqref{eq:skew_Cauchy_particular}
		implies that
		$Z_u^{(i,j)}=(1+\beta_ia_j)Z_{\ell}^{(i,j)}$.
\end{enumerate}

\begin{figure}[h]
	\centering
	\begin{tabular}{cccc}
											& $\mu$                         & $\prec'$           & $\nu$                             \\
		$\scriptstyle(a_j)$ & \rotatebox[origin=c]{90}{$\prec$} &                   & \rotatebox[origin=c]{90}{$\prec$} \\
											& $\kappa$                          & $\prec'$           & $\lambda$                             \\
											&                                   & $\scriptstyle(\beta_i)$ &
	\end{tabular}
	\caption{A quadruple of Young diagrams in the field $\lambda^{(T,K)}$.}
	\label{fig:quadruple}
\end{figure}

The above conditions \textbf{1}--\textbf{3} do not define a field 
$\lambda^{(T,K)}$
uniquely. Namely, while \eqref{eq:quadruple_distributions}
specifies the marginal distributions of $\kappa$ and $\nu$
(given $\mu,\lambda$),
it does not specify the joint distribution of $(\kappa,\nu)$
(given $\mu,\lambda$). 
It is possible to specify this joint distribution
such that 
\begin{enumerate}[$\bullet$]
	\item 
		the field $\{\lambda^{(T,K)}\}_{T,K\in \mathbb{Z}_{\ge0}}$ is well-defined
		(i.e., satisfies \textbf{1}--\textbf{3});
	\item 
		the scalar field 
		$\{\lambda^{(T,K)}_K\}_{T,K\in\mathbb{Z}_{\ge0}}$
		of the last parts of the partitions
		is marginally Markovian in the sense that
		its distribution does not depend on the distribution of the 
		other parts of the partitions.
\end{enumerate}
There are two main constructions of the field $\lambda^{(T,K)}$ satisfying \textbf{1}--\textbf{3} 
and with
marginally Markovian last parts. 
One involves the Robinson-Schensted-Knuth (RSK) correspondence 
and follows \cite{OConnell2003}, 
\cite{OConnell2003Trans},
see also \cite[Case B]{diekerWarren2008determinantal},
and another construction can be read off  
\cite{BorFerr2008DF}. 
The latter construction postulates that the joint distribution
of $\kappa,\nu$ given $\lambda,\mu$
is essentially the product of the marginal distributions
\eqref{eq:quadruple_distributions}, unless this violates conditions
in \Cref{fig:quadruple}
(in which case the product formula has to be corrected).
The RSK construction involves more complicated combinatorial rules for
stitching together the marginal distributions of $\kappa$ and $\nu$.
Either of these constructions of $\{\lambda^{(T,K)}\}$ 
works for our purposes, and we do not discuss further details.
The marginally Markovian evolution of the last parts $\lambda^{(T,K)}_K$ is
the discrete time TASEP with mixed geometric and Bernoulli steps which we describe in 
\Cref{sub:mixed_TASEP}.
In the rest of this section
we refer to $\{\lambda^{(T,K)}\}$ simply as \emph{the} random field of Young diagrams.

\begin{lemma}
	\label{lemma:marginal_distr}
	For any fixed $T,K\in \mathbb{Z}_{\ge0}$ the marginal distribution of the random Young diagram
	$\lambda=\lambda^{(T,K)}$ is given by the Schur measure 
	\begin{equation}
		\label{eq:Schur_measure_K_N_T}
		\mathop{\mathrm{Prob}}(\lambda)=
		\frac{1}{Z}
		s_\lambda(a_1,\ldots,a_K )
		s_\lambda
		\Bigl(
			\frac{\nu_2}{a_2},\ldots,\frac{\nu_N}{a_N} 
			;
			\beta_1,\ldots,\beta_T 
			;
			0
		\Bigr).
	\end{equation}
\end{lemma}
The normalizing constant in \eqref{eq:Schur_measure_K_N_T} 
is given by 
(cf. \eqref{eq:Cauchy})
\begin{equation*}
	Z=\prod_{i=1}^{K}
	\Biggl(
		\prod_{j=2}^{N}
		\frac{1}{1-a_i\nu_j/a_j}
		\prod_{j=1}^{T}(1+a_i\beta_j)
	\Biggr).
\end{equation*}
\begin{proof}[Idea of proof of \Cref{lemma:marginal_distr}]
	Follows by repeatedly applying 
	the skew Cauchy identity \eqref{eq:skew_Cauchy_particular}
	and arguing by induction on adding a box to
	grow the $T\times K$ rectangle.
	The additional specialization $(\nu_2/a_2,\ldots,\nu_N/a_N )$
	comes from the left boundary condition in the field
	$\{\lambda^{(T,K)}\}$.
\end{proof}

The notion of random fields of Young diagrams was introduced recently 
\cite{BufetovMatveev2017}, \cite{BufetovPetrovYB2017} 
to capture properties of coupled Schur processes.
This concept extends 
the work started with
\cite{OConnellPei2012},
\cite{BorodinPetrov2013NN}, and earlier 
applications of Robinson-Schensted-Knuth correspondences
to particle systems \cite{johansson2000shape}, 
\cite{OConnell2003}, \cite{OConnell2003Trans}.
In the next part we consider down-right joint distributions 
in the field $\lambda^{(K,T)}$ which are given by more general
Schur processes.

\begin{figure}[htpb]
	\centering
	\includegraphics[width=.6\textwidth]{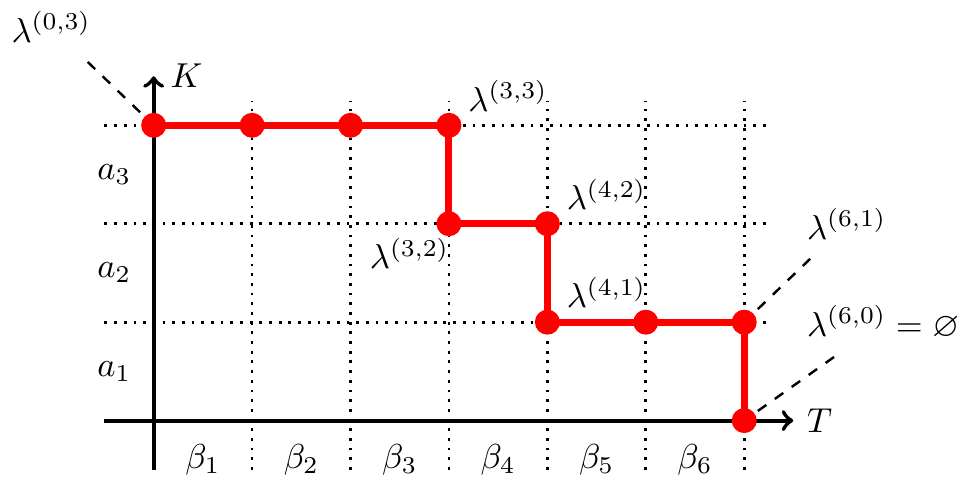}
	\caption{A graphical representation of the field 
	$\lambda^{(T,K)}$ of Young diagrams, and a down-right path.}
	\label{fig:field}
\end{figure}

\subsubsection{Schur processes and correlation kernels}
\label{appsub:Schur_processes}

Here we recall (at an appropriate level of generality) the definition and the
correlation kernel for Schur processes from \cite{okounkov2003correlation}.
Fix the parameters $N$ and $\{a_i\}$,
$\{\nu_i\}$, $\{\beta_i\}$.
Take a down-right path 
$\{(T_j,K_j)\}_{j=1}^{\ell}$, that is,
\begin{equation}
	\label{eq:app_A_down_right_path}
	K_1\ge \ldots\ge K_\ell=0,
	\qquad 
	0=T_1\le \ldots \le T_\ell,
\end{equation}
and, moreover, assume that the points $(T_j,K_j)$
are pairwise distinct.
A Schur process associated with this data is
a probability distribution on sequences $(\lambda;\mu)$ of Young diagrams
(see \Cref{fig:field} for an illustration)
\begin{equation}
	\label{eq:Schur_process_diagrams}
	\varnothing=\mu^{(1)}\subset \lambda^{(1)}\supset \mu^{(2)}\subset
	\lambda^{(2)}\supset\mu^{(3)}\subset \ldots \subset
	\lambda^{(\ell-1)}\supset\mu^{(\ell)}=\varnothing
\end{equation}
with probability weights
\begin{multline}
	\label{eq:Schur_process_weights}
	\mathop{\mathrm{Prob}}(\lambda;\mu)=
	\frac{1}{Z}
	s_{\lambda^{(1)}/\mu^{(1)}}\Bigl(
		\frac{\nu_2}{a_2},\ldots,\frac{\nu_N}{a_N} 
	\Bigr)
	s_{\lambda^{(1)}/\mu^{(2)}}(a_{(K_2,K_1]})
	s_{\lambda^{(2)}/\mu^{(2)}}(\beta_{(0,T_2]})
	\ldots \\\times
	s_{\lambda^{(\ell-1)}/\mu^{(\ell-1)}}(\beta_{(T_{\ell-1},T_\ell]})
	s_{\lambda^{(\ell-1)}/\mu^{(\ell)}}(a_{(0,K_{\ell-1}]})
	.
\end{multline}
Here
$a_{(u,v]}$ stands for the specialization $(a_{u+1},\ldots,a_v;\vec0;0 )$ corresponding to the 
vertical direction of the down-right path,
and $\beta_{(u,v]}$ means $(\vec0;\beta_{u+1},\ldots,\beta_v;0 )$ (this corresponds
to the horizontal direction).
Note that some of these specializations can be empty.
The normalizing constant in \eqref{eq:Schur_process_weights}
can be readily computed using the 
Cauchy identities.

As shown in \cite{okounkov2003correlation}, the Schur process 
\eqref{eq:Schur_process_weights} can be interpreted as a determinantal random
point process whose correlation kernel is expressed as a double contour
integral. (We refer to, 
e.g.,
\cite{Soshnikov2000}, \cite{peres2006determinantal}, \cite{Borodin2009},
for general definitions related to determinantal processes.)
To recall the result of \cite{okounkov2003correlation}, consider the particle configuration
\begin{equation}
	\label{eq:particle_configuration_for_Schur_process}
	\bigl\{
	\lambda^{(i)}_j-j \colon
	i=1,\ldots,\ell-1,\ j=1,2,\ldots \bigr\} \subset
	\underbrace{\mathbb{Z}\times\ldots\times\mathbb{Z}} _{\textnormal{$\ell-1$
			times}}
\end{equation}
corresponding to a sequence \eqref{eq:Schur_process_diagrams}
(where we sum over all the $\mu^{(j)}$'s).
The configurations $\lambda^{(i)}_j-j$, $j\ge1$, are infinite
and are densely packed at $-\infty$ (i.e., each partition
$\lambda^{(i)}$ is appended infinitely many zeroes).
Then, for any $m$ and any pairwise distinct locations $(r_p,x_p)$,
$p=1,\ldots,m$, where $1\le r_p\le \ell-1$ and $x_p\in \mathbb{Z}$, we have
\begin{multline*}
	\mathbb{P}\Bigl( \textnormal{there are points of the configuration
			\eqref{eq:particle_configuration_for_Schur_process} at each of the locations
			$(r_p,x_p)$} \Bigr)
			\\=\det \left[ \discKernel_{\mathsf{SP}}(r_p,x_p;r_q,x_q) \right]_{p,q=1}^{m}.
\end{multline*}
The kernel $\discKernel_{\mathsf{SP}}$ has the form
\begin{equation}
	\label{eq:kernel_Schur_general}
	\discKernel_{\mathsf{SP}}(i,x;j,y)= \frac{1}{(2\pi \iu)^2}\oint\oint
	\frac{dz\,dw}{z-w}\frac{w^{y}}{z^{x+1}}\frac{\Phi(i,z)}{\Phi(j,w)},
\end{equation}
where
\begin{equation*}
	\Phi(i,z)=
	\prod_{n=2}^{N}\frac{1}{1-z\nu_n/a_n}
	\prod_{t=1}^{T_i}(1+\beta_t z)
	\prod_{k=1}^{K_i}(1-z^{-1}a_k).
\end{equation*}
The integration contours in 
\eqref{eq:kernel_Schur_general}
are positively oriented simple closed curves 
around $0$
satisfying $|z|>|w|$ for $i\le j$ and $|z|<|w|$ for $i>j$.
Moreover, on the contours it must be $|z|<a_n/\nu_n$, $a_k<|w|<\beta_t^{-1}$
for all $n,t,k$ entering the products in \eqref{eq:kernel_Schur_general}.
In particular, the $w$ contour should encircle the $a_k$'s.
Thus, we have the following determinantal structure in the field $\{\lambda^{(T,K)}\}$:
\begin{proposition}
	\label{prop:kernel_field}
	For any $m\in \mathbb{Z}_{\ge1}$
	and any collection of pairwise distinct integer triplets
	$\{(T_i,K_i,x_i)\}_{i=1}^{m}$ such that
	$K_1\ge \ldots\ge K_m\ge0 $,
	$0\le T_1\le \ldots\le T_m$, we have
	\begin{multline}\label{eq:det_structure_K_field_appendix}
		\mathop{\mathrm{Prob}}
		\bigl(\textnormal{for all $i$, the configuration $\{\lambda^{(T_i,K_i)}_j-j\}_{j\ge1}\subset\mathbb{Z}$
		contains a particle at $x_i$}\bigl)
		\\
		=
		\det
		[\discKernel
		(T_p,K_p,x_p;T_q,K_q,x_q)]_{p,q=1}^{m},
	\end{multline}
	where the kernel is given by 
	\begin{multline}
		\label{eq:kernel_field}
		\discKernel
		(T,K,x;T',K',x')=
		\frac{1}{(2\pi \iu)^2}\oint\oint
		\frac{dz\,dw}{z-w}\frac{w^{x'}}{z^{x+1}}
		\prod_{n=2}^{N}\frac{1-w\nu_n/a_n}{1-z\nu_n/a_n}
		\\\times
		\frac{\prod_{t=1}^{T}(1+\beta_t z)}
		{\prod_{t=1}^{T'}(1+\beta_t w)}
		\frac{\prod_{k=1}^{K}(1-z^{-1}a_k)}
		{\prod_{k=1}^{K'}(1-w^{-1}a_k)}.
	\end{multline}
	The contours
	are positively oriented simple closed curves 
	around $0$ such that $w$ also encircles the $a_k$'s,
	$|z|>|w|$ for $T\le T'$, and
	$|z|<|w|$ for $T>T'$.
	Moreover, on the contours it must be $|z|<a_n/\nu_n$, $|w|<\beta_t^{-1}$
	for all $t,n$.
\end{proposition}
\begin{remark}
	In the description of the integration contours
	in \Cref{prop:kernel_field} we silently assumed that 
	the parameters $a_i,\beta_t,\nu_j$
	satisfy certain restrictions such that the 
	contours exist. In \Cref{prop:kernels_equivalent}
	below we deform the contours
	and lift these restrictions when $K=K'=N$ (this holds when we apply the 
	Schur process structure to DGCG).
\end{remark}

\subsubsection{Particles at the edge and Fredholm determinants}
\label{sub:Fredholm_from_Schur}

The joint distribution of the last parts of the partitions
$\{\lambda^{(T_i,K_i)}_{K_i}\}$
(which evolve in a marginally Markovian way)
for $(T_i,K_i)$ along a down-right path can be written 
in terms of a Fredholm determinant. 

Let us first recall Fredholm determinants on an abstract discrete space
$\mathfrak{X}$. Let $\discKernel(i,i')$, $i,i'\in \mathfrak{X}$, be a kernel on this space.
We define the Fredholm determinant of $\mathbf{1}+z\discKernel$, $z\in \mathbb{C}$, as
the infinite series
\begin{equation}
	\label{eq:K_Fredholm_determinant}
	\det(1+z\discKernel)_{\mathfrak{X}}=1+\sum_{r=1}^{\infty}
	\frac{z^r}{r!}\sum_{i_1\in\mathfrak{X}}\ldots
	\sum_{i_r\in\mathfrak{X}}
	\det\left[ \discKernel(i_p,i_q) \right]_{p,q=1}^{r}.
\end{equation}
One may view \eqref{eq:K_Fredholm_determinant} as a formal series, but in our
setting this series will converge numerically. Details on Fredholm
determinants may be found in \cite{Simon-trace-ideals} or
\cite{Bornemann_Fredholm2010}.

Fix a down-right path $\{(T_i,K_i)\}_{i=1}^{\ell}$ as in 
\eqref{eq:app_A_down_right_path},
and consider the space
\begin{equation*}
	\mathfrak{X}=
	\bigcup_{i=1}^{\ell-1}
	\bigl(\{T_i\}\times\{K_i\}
	\times 
	\mathbb{Z}\bigr).
\end{equation*}
According to \Cref{prop:kernel_field},
let us view 
$\{\lambda^{(T_i,K_i)}_j-j\colon i=1,\ldots,\ell-1,\; j=1,2,\ldots \}$
as a determinantal point process on $\mathfrak{X}$ with 
correlation
kernel 
$\discKernel(Y;Y')=
\discKernel(T,T,y;T',K',y')$, where 
$Y=(T,K,y),Y'=(T',K',y')\in \mathfrak{X}$.
Fix $\vec y=(y_1,\ldots,y_{\ell-1} )\in \mathbb{Z}^{\ell-1}$
and interpret 
\begin{equation*}
	\mathop{\mathrm{Prob}}
	\Bigl( \lambda^{(T_i,K_i)}_{K_i}-K_i>y_i\colon i=1,\ldots,\ell-1  \Bigr)
\end{equation*}
as the probability that the 
random point configuration
$\mathfrak{X}$ corresponding to our determinantal process
has no particles in the set 
\begin{equation*}
	\mathfrak{X}_{\vec y}
	:=
	\bigcup_{i=1}^{\ell-1}
	\bigl(
		\left\{ T_i \right\}\times\left\{ K_i \right\}
		\times
		\left\{ -K_i,-K_i+1,\ldots,y_i-1,y_i  \right\}
	\bigr)
	\subset \mathfrak{X}.
\end{equation*}
This probability can be written (e.g., see \cite{Soshnikov2000}) as the
Fredholm determinant
\begin{equation*}
	\det\left( \mathbf{1}- \chi_{\vec y}\discKernel\chi_{\vec y} \right)_{\mathfrak{X}},
\end{equation*}
where 
$\chi_{\vec y}(T_i,K_i,x)
=
\mathbf{1}_{-K_i\le x\le y_i}$, $i=1,\ldots,\ell-1 $,
is the indicator of $\mathfrak{X}_{\vec y}\subset \mathfrak{X}$
viewed as a projection operator acting on functions.

In particular, in the one-point case we get the following Fredholm determinant:
\begin{multline*}
	\mathop{\mathrm{Prob}}
	\bigl( \lambda^{(T,K)}_K-K>y \bigr)
	=
	\det(\mathbf{1}-\discKernel(T,K,\cdot;T,K,\cdot))_{\{-K,-K+1,\ldots,y-1,y \}}
	\\
	=1+\sum_{m=1}^{\infty}
	\frac{(-1)^m}{m!}
	\sum_{x_1=-K}^{y}\ldots 
	\sum_{x_m=-K}^{y}
	\det
	[\discKernel(T,K,x_p;T,K,x_q)]_{p,q=1}^{m},
\end{multline*}
where the last equality is the series expansion of the Fredholm determinant.

\subsection{Determinantal structure of DGCG}
\label{sub:new_Schur_connection_discrete_system_determinantal_structure}

Let us now apply the formalism of Schur processes to the DGCG model.
We will us the kernel $\discKernel$
\eqref{eq:kernel_field} with  
$K=K'=N$ and different integration contours. That is, define
\begin{equation}
	\label{eq:new_kernel_full}
	\begin{split}
		&\discKernel_N
		(T,x;T',x'):=
		-
		\frac{\mathbf{1}_{T>T'}\mathbf{1}_{x\ge x'}}{2\pi\iu}
		\oint \frac{\prod_{t=T'+1}^{T}(1+\beta_t z)}{z^{x-x'+1}}\,dz
		\\&\hspace{20pt}+
		\frac{1}{(2\pi \iu)^2}\oint\oint
		\frac{dz\,dw}{z-w}\frac{w^{x'+N}}{z^{x+N+1}}
		\prod_{n=2}^{N}\frac{1-w\nu_n/a_n}{1-z\nu_n/a_n}
		\frac{\prod_{t=1}^T(1+\beta_t z)}
		{\prod_{t=1}^{T'}(1+\beta_t w)}
		\prod_{k=1}^{N}
		\frac{a_k-z}
		{a_k-w},
	\end{split}
\end{equation}
where $T,T'\in\mathbb{Z}_{\ge0}$
and $x,x'\in\mathbb{Z}_{\ge -N}$.
In the single integral the contour is a small positively oriented circle around $0$,
and the contours in the double integral satisfy:
\begin{enumerate}[$\bullet$]
	\item the $z$ contour is a small positively oriented circle around $0$
		which must be to the left of all points 
		$a_n/\nu_n$;
	\item the $w$ contour is a positively oriented 
		simple closed curve around all the $a_k$'s
		which stays to the right of zero, all points 
		$-\beta_t^{-1}$, and the $z$ contour.
\end{enumerate}

\begin{proposition}
	\label{prop:kernels_equivalent}
	The integration contours in \eqref{eq:new_kernel_full}
	exist for all choices of parameters $a_i>0$, $\beta_t>0$, and $\nu_j\in[0,1)$.
	Moreover, $\discKernel_N(T,x;T',x')=\discKernel(T,N,x;T',N,x')$,
	where the latter is given in \eqref{eq:kernel_field}.
\end{proposition}
In other words, the deformation of contours from 
$\discKernel(T,N,x;T',N,x')$
to 
$\discKernel_N(T,x;T',x')$
provides an analytic continuation of the kernel
to the full range of parameters
$a_i>0$, $\beta_t>0$, $\nu_j\in[0,1)$.
\begin{proof}[Proof of \Cref{prop:kernels_equivalent}]
	The existence of the contours is straightforward.
	Let us explain how to deform the contours
	in
	$\discKernel(T,N,x;T',N,x')$
	to get the desired result.
	First, note that 
	the integrand is regular at $w=0$
	for $x'\ge -N$.
	Depending on the relative order of $T$ and $T'$, perform the 
	following contour deformations:
	\begin{enumerate}[$\bullet$]
		\item 
			For $T\le T'$, 
			the $w$ contour is inside the $z$ one
			in 
			\eqref{eq:kernel_field}.
			Drag the $z$ contour through the $w$ one,
			and turn $z$ into a 
			small circle around $0$.
			The $w$ contour then needs to 
			encircle only $\{a_i \}$ and not zero,
			as desired.
			This deformation of the contours 
			results in a single integral of the residue
			at $z=w$ over the new $w$ contour, 
			but since this contour does not include zero,
			the single integral vanishes.
			
		\item When $T> T'$, 
			the $z$ contour is inside the $w$ one
			in 
			\eqref{eq:kernel_field}.
			Make $z$ a small circle around $0$, 
			then drag the $w$ contour through the 
			$z$ one, and have the $w$ contour
			encircle $\{ a_i \}$ and not zero.
			This deformation brings a single integral 
			of the residue at $w=z$ over the new $z$ contour,
			and this is precisely the single integral we get in 
			\eqref{eq:new_kernel_full}.
	\end{enumerate}
	These contour deformations lead to the kernel $\discKernel_N$.
\end{proof}

Fix $\ell\ge1$ and $0\le T_1\le \ldots\le T_{\ell}$,
and define a determinantal point process
$\mathfrak{L}_{N,\ell}$
on 
$\mathfrak{X}:=\mathbb{Z}_{\ge -N}\times \{T_1,\ldots,T_\ell \}$
as follows. For any $m\ge1$ and $m$ pairwise distinct
points $(x_i,t_i)\in \mathfrak{X}$, set
\begin{multline}
	\label{eq:determinantal_pp_definition_intro_discrete}
	\mathop{\mathrm{Prob}}
	\bigl( 
		\textnormal{the random configuration $\mathfrak{L}_{N,\ell}$
		contains all points $(x_i,t_i)$, $i=1,\ldots,m $}
	\bigr)
	\\=
	\det\left[ \discKernel_N(t_i,x_i;t_j,x_j) \right]_{i,j=1}^{m}.
\end{multline}
In other words, $\mathfrak{L}_{N,\ell}$ is the $\mathbb{Z}_{\ge-N}$-part of the 
determinantal process $\lambda^{(i)}_j-j$ coming from the
Schur process as in \Cref{appsub:Schur_processes}
corresponding to the down-right path
$\{(T_j,K_j)\}=\{(T_j,N)\}$.
See \Cref{fig:det_PP} for an illustration.

\begin{figure}[htpb]
	\centering
	\includegraphics[width=.5\textwidth]{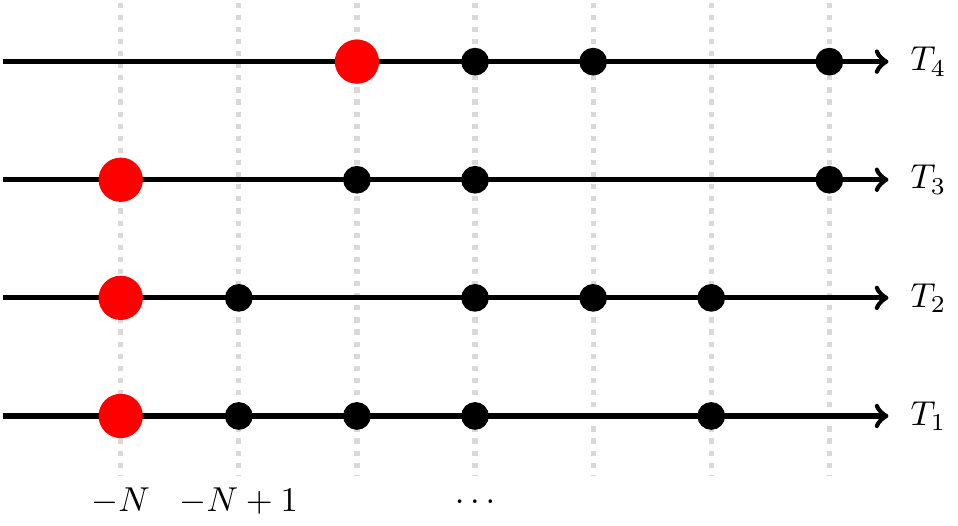}
	\caption{An example of a configuration of $\mathfrak{L}_{N,4}$.
	The leftmost particles are highlighted.}
	\label{fig:det_PP}
\end{figure}

\begin{theorem}
	\label{thm:det_structure_discrete_Schur}
	With the above notation, the 
	joint distribution of the height function of the DGCG
	\begin{equation*}
		\{H_{T_j}(N+1)-N\}_{j=1}^{\ell}
	\end{equation*}
	coincides with the joint distribution of the leftmost points
	of the determinantal point process
	$\mathfrak{L}_{N,\ell}$ on
	$\mathbb{Z}_{\ge -N}\times \{T_1,\ldots,T_\ell \}$.
\end{theorem}
\begin{proof}
	We know from
	\Cref{prop:Schur_vertex_model_corner_growth_correspondence} and 
	\Cref{prop:relation_mixed_TASEP_to_higher_spin_corner_growth}
	that 
	\begin{equation*}
		\{H_{T_j}(N+1)-N\}_{j=1}^{\ell}
		\stackrel{d}{=}
		\{Y_N(N-1;T_j)\}_{j=1}^{\ell},
	\end{equation*}
	where $\vec Y(N-1;T)$ is the mixed TASEP of \Cref{def:mixed_TASEP}.
	If we connect $\vec Y(N-1;T)$ to a field of random Young diagrams,
	then the desired statement would follow from the 
	determinantal structure of the Schur process described in
	\Cref{sub:new_Schur_connection_discrete_system_determinantal_structure}.
	
	The desired connection of the mixed 
	TASEP with particle-dependent inhomogeneity
	to Schur processes is in well-known
	and follows from 
	the column 
	Robinson-Schensted-Knuth (RSK) correspondences
	(see \cite{fulton1997young}, \cite{Stanley1999} 
		for details on RSK, and, e.g.,
		\cite{johansson2000shape}, 
		\cite{OConnell2003Trans}, 
		\cite{warrenwindridge2009some}
	for probabilistic applications of RSK to TASEPs)
	or, alternatively, from the results of 
	\cite{BorFerr2008DF}.
	The precise connection reads as follows.
	For any down-right path $\{(T_j,K_j)\}_{j=1}^{\ell}$
	\eqref{eq:app_A_down_right_path} we have the equality of the following
	joint distributions:
	\begin{equation}
		\label{eq:TASEP_Schur_correspondence}
		\left\{ Y_{K_j}(N-1;T_j)+K_j \right\}_{j=1}^{\ell-1}
		\stackrel{d}{=}
		\bigl\{ \lambda^{(T_j,K_j)}_{K_j} \bigr\}_{j=1}^{\ell-1},
	\end{equation}
	where $\vec Y$ is the mixed TASEP,
	and
	$\{\lambda(T,K)\}$ is the random field from
	\Cref{appsub:Schur_field}.
	In particular, the distribution of each particle $Y_K(N-1;T)$ 
	in the mixed TASEP is the same as of
	$\lambda_K-K$, where $\lambda_K$ is the 
	last part of a random partition $\lambda$
	chosen from the Schur measure 
	$\propto s_\lambda(a_1,\ldots, a_K;\vec0;0)s_\lambda( \nu_2/a_2,\ldots,\nu_N/a_N ;\beta_1,\ldots, \beta_T;0 )$.

	Taking $K_j\equiv N$ in
	\eqref{eq:TASEP_Schur_correspondence}
	and using \Cref{prop:kernel_field}
	(together with \Cref{prop:kernels_equivalent}
	for the contour deformation),
	we arrive at the claim.
\end{proof}

In particular, for $\ell=1$ \Cref{thm:det_structure_discrete_Schur}
implies the following
Fredholm determinantal expression for the distribution of the random variable
$H_T(N+1)$:
\begin{multline}
	\label{eq:one-point-Fredholm-discrete}
	\mathop{\mathrm{Prob}}
	\bigl(
		H_T(N+1)-N>y
	\bigr)
	=
	\det(\mathbf{1}-\discKernel_N(T,\cdot;T,\cdot))_{ \left\{ -N,-N+1,\ldots,y-1,y  \right\}}
	\\
	=
	1+
	\sum_{m=1}^{\infty}
	\frac{(-1)^m}{m!}
	\sum_{x_1=-N}^{y}\ldots \sum_{x_m=-N}^{y}
	\det
	[\discKernel_N(T,x_i;T,x_j)]_{i,j=1}^{m}.
\end{multline}
The second equality is the series expansion of the Fredholm determinant,
see \Cref{sub:Fredholm_from_Schur}.

\subsection{Determinantal structure of continuous space TASEP}
\label{sub:new_cont_TASEP}

Let us now describe the determinantal structure 
of the continuous space TASEP which 
follows by taking the continuous space
scaling of the DGCG results.
By $N+\mathfrak{L}_{N,\ell}$ denote the shift 
of the determinantal process $\mathfrak{L}_{N,\ell}$
from 
\Cref{sub:new_Schur_connection_discrete_system_determinantal_structure}
by $N$ to the right.

\begin{theorem}
	\label{thm:intro_convergence_discrete_to_continuous_TASEP}
	As $\varepsilon\to0$ and under the scaling described 
	in
	defined \Cref{sub:cont_space_scaling}, 
	$N+\mathfrak{L}_{N,\ell}$
	converges in the sense of finite dimensional distributions 
	to a determinantal point process $\widetilde{\mathfrak{L}}_{\ell}$
	on $\mathbb{Z}_{\ge0}\times\{t_1,\ldots,t_\ell \}$
	(where $T_i=\lfloor \varepsilon^{-1}t_i \rfloor $) with the kernel\,\footnote{Which 
	expresses the correlations of the process $\widetilde{\mathfrak{L}}_{\ell}$ 
	by analogy with
	\eqref{eq:determinantal_pp_definition_intro_discrete}.}
	\begin{align}
		&\contKernel(t, x; t', x')
		=
		-
		\mathbf{1}_{t>t'}\mathbf{1}_{x\ge x'}
		\frac{(t-t')^{x-x'}}{(x-x')!}
		\nonumber\\
		&\hspace{30pt}+
		\frac{1}{(2\pi\iu)^2} \oint \oint \frac{dw dz}{z-w}
		\frac{w^{x'}}{z^{x+1}}
		\exp
		\biggl\{
			t z-t'w+
			L \int \limits_0^\chi \left
			(\frac{w}{\xi (u)-w}-\frac{z}{\xi (u)-z} \right)du
		\biggr\}
		\label{eq:cont_kernel_full_intro}
		\\\nonumber
		&\hspace{190pt}
		\times
		\frac{(\xi(0)-z)}{(\xi(0)-w)}
		\prod\limits_{b\in\RoadblockSet\colon b < \chi} \frac{\xi(b)- p(b)w}{\xi(b)- p(b) z} 
		\cdot \frac{\xi(b)-z}{\xi(b)-w}.
	\end{align}
	The
	$z$ contour
	is a small positively oriented circle around $0$
	which must be to the left of all points 
	$\xi(y),$ $y\in[0,\chi]$.
	The $w$ 
	contour
	is a positively oriented simple closed curve around all
	points $\xi(y),$ $y\in[0,\chi]$
	which is also to the right of the $z$ contour.
	
	Correspondingly, the joint distribution 
	$\left\{ \mathcal{H}(t_i,\chi) \right\}_{i=1}^{\ell}$
	of the continuous space TASEP height function
	coincides with the joint distribution 
	of the leftmost particles of $\widetilde{\mathfrak{L}}_{\ell}$.
\end{theorem}
\begin{proof}
	The second part of the claim (that $\mathcal{H}(t_i,\chi)$
	are the leftmost points of $\widetilde{\mathfrak{L}}_{\ell}$)
	follows from the first part together with 
	\Cref{thm:convergence_discrete_to_continuous_TASEP_process_only}.
	Thus, it suffices to establish the convergence 
	of the correlation kernels
	$\discKernel_N$ \eqref{eq:new_kernel_full}
	to
	$\contKernel$ \eqref{eq:cont_kernel_full_intro}
	(which would
	imply the convergence of determinantal point processes
	in the sense of finite dimensional distributions
	since those are completely determined by the correlation
	kernels,
	cf. \cite{Soshnikov2000}).
	
	Because of the shift
	$N+\mathfrak{L}_{N,\ell}$
	we first subtract $N$ from $x,x'$ in 
	$\discKernel_N$, and then scale
	$a_i,\nu_j,\beta_t,T,N$ depending on $\varepsilon$.
	First, observe that the single integral
	in \eqref{eq:new_kernel_full}
	converges to the first term in \eqref{eq:cont_kernel_full_intro}:
	\begin{equation*}
		-
		\frac{\mathbf{1}_{T>T'}\mathbf{1}_{x\ge x'}}{2\pi\iu}
		\oint \frac{\prod_{t=T'+1}^{T}(1+\beta_t z)}{z^{x-x'+1}}\,dz
		\to
		-
		\frac{\mathbf{1}_{t>t'}\mathbf{1}_{x\ge x'}}{2\pi\iu}
		\oint 
		\frac{e^{z(t-t')}dz}{z^{x-x'+1}}
		=
		-
		\mathbf{1}_{t>t'}\mathbf{1}_{x\ge x'}
		\frac{(t-t')^{x-x'}}{(x-x')!}.
	\end{equation*}

	Next, let us look at the double integrals.
	Under our scaling the integration contours 
	readily match, so it remains to show the convergence of the
	integrands. 
	Keep
	$\dfrac{w^{x'}}{z^{x+1}(z-w)}$, and 
	also separate
	the factors $\dfrac{a_1-z}{a_1-w}=\dfrac{\xi(0)-z}{\xi(0)-w}$
	from the product over $k=1,\ldots, N$.
	These factors do not change with $\varepsilon$.
	Consider the limit as $\varepsilon\to0$ of the remaining factors
	in the integrand. 
	We have 
	\begin{equation*}
		\frac{\prod_{t=1}^T(1+\beta_t z)}
		{\prod_{t=1}^{T'}(1+\beta_t w)}
		=
		\frac{(1+\varepsilon z)^{\lfloor \varepsilon^{-1}t \rfloor }}
		{(1+\varepsilon w)^{\lfloor \varepsilon^{-1}t' \rfloor }}\to 
		e^{tz-t'w}.
	\end{equation*}
	In the product
	\begin{equation*}
		\prod_{n=2}^{N}
		\frac{a_n-w\nu_n}{a_n-w}
		\cdot
		\frac{a_n-z}{a_n-z\nu_n}
	\end{equation*}
	consider separately the factors corresponding to $n\in \RoadblockSet^{\varepsilon}$.
	We obtain for all sufficiently small $\varepsilon$:
	\begin{equation*}
		\prod_{2\le n\le \lfloor \varepsilon^{-1}\chi \rfloor ,\, n\in \RoadblockSet^{\varepsilon}}
		\frac{a_n-w\nu_n}{a_n-w}
		\cdot
		\frac{a_n-z}{a_n-z\nu_n}
		=
		\prod_{b\in\RoadblockSet\colon b < \chi}
		\frac{\xi(b)- p(b)w}
		{\xi(b)-w}
		\cdot 
		\frac{\xi(b)-z}
		{\xi(b)- p(b) z}
		,
	\end{equation*}
	and these factors also do not change with $\varepsilon$
	(there are finitely many roadblocks on $[0,\chi)$).
	Finally, 
	\begin{multline*}
		\prod_{2\le n\le N,\, n\notin\RoadblockSet^{\varepsilon}}
		\frac{a_n-w\nu_n}{a_n-w}
		=
		\exp
		\Biggl\{
			\sum_{2\le n\le \lfloor \varepsilon^{-1}\chi \rfloor ,\, n\notin\RoadblockSet^{\varepsilon}}
			\log\biggl( 
				\frac{\xi(n\varepsilon)-we^{-L\varepsilon}}{\xi(n\varepsilon)-w}
			\biggr)
		\Biggr\}
		\\=
		\exp
		\Biggl\{
			\varepsilon L
			\sum_{2\le n\le \lfloor \varepsilon^{-1}\chi \rfloor ,\, n\notin\RoadblockSet^{\varepsilon}}
			\biggl(
				\frac{w}{\xi(n\varepsilon)-w}
				+O(\varepsilon^2)
			\biggl)
		\Biggr\}
		\to
		\exp
		\biggl\{
			L\int_0^\chi
			\frac{w}{\xi(u)-w}\,du
		\biggr\}
		,
	\end{multline*}
	because the exclusion of finitely many points $n\in \RoadblockSet^{\varepsilon}$
	changes the value of the Riemann sums by $O(\varepsilon)$ which is negligible.
	A similar convergence to 
	the exponent of an integral holds for the 
	$z$ variable.
\end{proof}

\begin{remark}
	The limiting determinantal process $\widetilde{\mathfrak{L}}_{\ell}$ 
	in \Cref{thm:intro_convergence_discrete_to_continuous_TASEP}
	may be viewed as a new (and very general) limit of Schur measures
	and processes. Let us discuss the case $\ell=1$.
	The height function $H_T(N)$ is identified with the leftmost point of
	a determinantal point process 
	$N+\mathfrak{L}_{N,1}\subset \mathbb{Z}_{\ge0}$.
	This point process
	is the same as the 
	random point configuration $\{\lambda_j+N-j\}_{j=1}^{N}\subset\mathbb{Z}_{\ge0}$,
	where $\lambda$ is distributed as the Schur measure
	$\propto s_\lambda(a_1,\ldots,a_N )s_\lambda(\nu_2/a_2,\ldots,\nu_N/a_N;\beta_1,\ldots,\beta_T )$.
	\Cref{thm:intro_convergence_discrete_to_continuous_TASEP}
	states that under the scaling 
	$\beta_t\equiv \varepsilon$, $T=\lfloor \varepsilon^{-1}t \rfloor $,
	$N=\lfloor \varepsilon^{-1}\chi \rfloor $, 
	and \eqref{eq:DGCG_to_cTASEP_scaling_no_roadblocks}--\eqref{eq:DGCG_to_cTASEP_scaling_roadblocks}
	these Schur measures 
	converge to an infinite random configuration 
	$\widetilde{\mathfrak{L}}_{1}$
	on $\mathbb{Z}_{\ge0}$.
	
	This infinite random point configuration 
	$\widetilde{\mathfrak{L}}_{1}$
	is a determinantal
	process
	with kernel $\mathcal{K}$ \eqref{eq:cont_kernel_full_intro}
	whose leftmost point has the same distribution as 
	$\mathcal{H}(t,\chi)$.
	This general limit of Schur measures to infinite random point configurations on $\mathbb{Z}_{\ge0}$
	depending on $t,\chi$, $L$, arbitrary speed function $\xi(\cdot)$, 
	and the roadblocks as parameters
	appears to be new.
	Certain related 
	discrete infinite-particle limits 
	of Schur and Schur-type measures
	have appeared before in 
	\cite{borodin2007asymptotics},
	\cite{borodinDuits2011GUE},
	\cite{BO2016_ASEP}.
\end{remark}

\section{Asymptotics of continuous space TASEP. Formulations}
\label{sec:asymptotic_results}

\subsection{Limit shape}
\label{sub:limit_shape_cont_TASEP}

We consider the following limit regime for the continuous
space TASEP:
\begin{equation}
	\label{eq:continuous_scaling_regular_behaior}
	L\to+\infty, \quad  t=\theta L, 
	\quad 
	\textnormal{location $\chi>0$, the speed function $\xi(\cdot)$,
	and roadblocks are not scaled}.
\end{equation}
Here $\theta>0$ is the scaled time. 
Denote 
\begin{equation}\label{range_of_Speed_notation}
	\SpeedEssRange_\chi:=
	\mathop{\mathrm{EssRange}}\{\xi(\gamma)\colon 0< \gamma <\chi\}
	\cup\{\xi(0)\}
	\cup\bigcup_{b\in\RoadblockSet\colon 0<b<\chi}\{\xi(b)\},
	\qquad 
	\mathcal{W}_\chi:=\min 	\SpeedEssRange_\chi,
\end{equation}
where $\mathop{\mathrm{EssRange}}$ stands for the \emph{essential range}, i.e.,
the set of all points for which the preimage of any neighborhood under
$\xi$ has positive Lebesgue measure. 
Note that we include the values
of $\xi(\cdot)$ corresponding to $0$ and the roadblocks even if they do not
belong to the essential range. 
These values play a special role
because each of the point locations 
$\{0\}\cup\mathbf B$ contains at least one particle with nonzero probability. 
For future use, also set
\begin{equation}\label{essential_range_of_Speed_notation}
	\SpeedEssRangeCirc_\chi:=\textup{EssRange}\{\xi(\gamma)\colon 0< \gamma <\chi\}, \qquad 
	\mathcal{W}_\chi^\circ:=\min 	\SpeedEssRangeCirc_\chi.
\end{equation}

Consider equation
\begin{equation} \label{eq:critical_continuous}
	\int \limits_{0}^{\chi} 
	\frac{\xi (u)(\xi (u)+w) du}{(\xi (u)-w)^{3}}
	=\theta
\end{equation}
in $w\in(0,\mathcal{W}_\chi^\circ)$.
\begin{definition}
	\label{def:curver_part_continuous_regular}
	We say that the pair 
	$(\theta,\chi)\in \mathbb{R}_{>0}^2$ is in the 
	\emph{curved part} if 
	\begin{equation*}
		\int_0^\chi \frac{du}{\xi(u)}<\theta.
	\end{equation*}
	This inequality corresponds to comparing 
	both sides of
	\eqref{eq:critical_continuous} at
	$w=0$.
\end{definition}

\begin{lemma}\label{lemma:critical_point_exists_continuous}
	For $(\theta,\chi)$ in the curved part
	there exists a unique solution
	$w=\mathfrak{w}^\circ(\theta, \chi)$ to equation~\eqref{eq:critical_continuous}
	in 
	$w\in(0, \mathcal{W}_\chi^\circ)$.
	For fixed
	$\chi$ the function $\theta \mapsto \mathfrak{w}^\circ(\theta, \chi)$ is
	strictly increasing from zero, and 
	$\lim\limits_{\theta\rightarrow \infty}\mathfrak{w}^\circ(\theta, \chi)=\mathcal{W}_\chi^\circ$.
	For
	fixed $\theta$ the function $\chi \mapsto \mathfrak{w}^\circ(\theta, \chi)$
	is strictly decreasing to zero.
\end{lemma}
\begin{proof}
	Denote the left hand side of \eqref{eq:critical_continuous} by $I(w).$ Note
	that 
	\begin{equation*}
		\frac{\partial I(w)}{\partial w}=\int\limits_0^{\chi}
		\frac{2 \xi (u)(2 \xi (u)+w) du}{(\xi (u)-w)^{4}}>0.
	\end{equation*}
	Thus, $I(w)$ is strictly increasing on $(0, \mathcal{W}_\chi^\circ).$
	Since $\theta>I(0)$ by the assumption, and $I(w)\to+\infty$ as $w$ 
	approaches $\mathcal{W}_\chi^\circ$,
	there is a unique solution to \eqref{eq:critical_continuous} on the desired interval.
	The monotonicity properties of the solution are straightforward.
\end{proof}

\begin{definition}
	\label{def:limit_shape_continuous}
	Let $(\theta,\chi)\in \mathbb{R}_{>0}^{2}$.
	Define \emph{the limit shape}
	of the height function of the continuous space TASEP 
	as follows:
	\begin{equation*}
		\mathfrak{h}(\theta, \chi)
		:= 
		\begin{cases} 
			+\infty, &
			\textnormal{if $\chi=0$ and $\theta\ge0$};
			\\ 
			0,& 
			\textnormal{if $\chi>0$ and $(\theta,\chi)$ is not in the curved part} 
			;
			\\
			\displaystyle
			\theta\mkern2mu \mathfrak{w}(\theta,\chi)-\int
			\limits_{0}^{\chi}  \frac{\xi (u) \mathfrak{w}(\theta,\chi) du}{(\xi
			(u) - \mathfrak{w}(\theta,\chi))^{2}},
			& \textnormal{if $\chi>0$ and $(\theta,\chi)$ is in the curved part},
		\end{cases} 
	\end{equation*}
	where 
	\begin{equation}
		\label{eq:w_min_definition_which_phase}
		\mathfrak{w}(\theta, \chi):=
		\min\bigl(\mathfrak{w}^\circ(\theta, \chi),
		\mathcal{W}_\chi\bigr).
	\end{equation}
\end{definition}

Depending on which of the two expressions 
in the right-hand side of \eqref{eq:w_min_definition_which_phase}
produce the minimum, let us give the following definitions:
\begin{definition}
	\label{def:TW_G_BBP_phases}
	Assume that $(\theta,\chi)$ is in the curved part.
	If $\mathfrak{w}^\circ(\theta,\chi)<\mathcal{W}_\chi$,
	we say that the point $(\theta,\chi)$ is in the 
	\emph{Tracy-Widom phase}. 
	If 
	$\mathfrak{w}^\circ(\theta,\chi)>\mathcal{W}_\chi$,
	then 
	$(\theta,\chi)$ is in the 
	\emph{Gaussian phase}.
	If 
	$\mathfrak{w}^\circ(\theta,\chi)=\mathcal{W}_\chi$
	we say that $(\theta,\chi)$ is a
	\emph{BBP transition}.
	If $(\theta,\chi)$ is a transition point or is in the Gaussian phase,
	denote
	\begin{equation}
		\label{eq:m_x_multiplicity_definition_intro}
		m_\chi
		:=
		\#
		\Bigl\{ 
			y\in\left\{ 0 \right\}\cup\left\{ b\in\RoadblockSet\colon 0<b<\chi \right\} 
			\colon \xi(y)=\mathcal{W}_\chi
		\Bigr\}.
	\end{equation}
	The names of the phases match the fluctuation
	behavior observed in each phase, see 
	\Cref{sub:cont_TASEP_asympt_fluct_intro} below.
\end{definition}

\begin{theorem}
	\label{thm:continuous_intro_limit_shape_theorem}
	Under the scaling \eqref{eq:continuous_scaling_regular_behaior},
	we have the convergence 
	of the height function of the continuous space TASEP
	to the limiting height function
	of \Cref{def:limit_shape_continuous}:
	\begin{equation*}
		L^{-1}\mathcal{H}(\theta L,\chi)\to \mathfrak{h}(\theta,\chi)
		\ \text{in probability as $L\to+\infty$}.
	\end{equation*}
\end{theorem}

We prove \Cref{thm:continuous_intro_limit_shape_theorem} in 
\Cref{sub:cont_Fredholm_asymptotics}.

\subsection{Macroscopic properties of the limit shape}

Let us mention two macroscopic properties
of the limit shape of \Cref{def:limit_shape_continuous}.
For simplicity assume that there are 
no roadblocks.

First, one can 
check that 
the function 
$\mathfrak{h}(\theta,\chi)$ satisfies 
a natural hydrodynamic partial differential equation.
We write it down in
\Cref{sub:rmk_hydrodynamics_continuous},
and in \Cref{sub:rmk_hydrodynamics_discrete}
discuss its counterpart for the DGCG model.

Second,
as a function of $\theta$,
$\mathfrak{h}(\theta,\chi)$
can be represented as a Legendre dual of a certain explicit
function. 
Namely, 
let
\begin{equation}\label{eq:cont_G_L_function_def}
	G(v)=G(v;\theta,\chi,h):=-\theta v+h\log v
	+
	\mathscr{F}(v),
	\qquad 
	\mathscr{F}(v):=\displaystyle\int_0^\chi \frac{\xi(u)du}{\xi(u)-v}.
\end{equation}
We assume that $\chi$ is fixed, $h=\mathfrak{h}(\theta,\chi)$, 
and consider 
the behavior of $G$ as a function of $v$.
We have
\begin{equation*}
	-(v G''(v)+G'(v))=\theta-\mathscr{F}\,'(v)-v \mathscr{F}\,''(v)=
	\frac{\partial}{\partial v}\left( \theta v-v\mathscr{F}\,'(v) \right)
	=
	\theta-\int_0^\chi\frac{\xi(u)\left( \xi(u)+v \right)\,du}
	{(\xi(u)-v)^3}.
\end{equation*}
This expression vanishes at $v=\mathfrak{w}^\circ(\theta,\chi)$,
or, in other words, 
$v=\mathfrak{w}^\circ(\theta,\chi)$ is a critical point 
of $v\mapsto \theta v-v \mathscr{F}'(v)$.
From the proof of \Cref{lemma:critical_point_exists_continuous}
it follows that $(\theta v-v \mathscr{F}\,'(v))''=-
(v\mathscr{F}'(v))''<0$, so this critical point is a maximum.
Moreover, this maximum is unique on $(0,\mathcal{W}^\circ_{\chi})$
also by \Cref{lemma:critical_point_exists_continuous}.

At the same time, $\mathfrak{h}$
can be written as 
$\mathfrak{h}(\theta,\chi)=
\theta v-v\mathscr{F}\,'(v)
\,\big\vert_{v=\mathfrak{w}^\circ(\theta,\chi)}$.
Therefore, we have
\begin{equation*}
	\mathfrak{h}(\theta,\chi)
	=
	\max_{v\in[0,\mathcal{W}^\circ_{\chi})}
	\left( \theta v- v\mathscr{F}\,'(v) \right),
\end{equation*}
which is the Legendre dual of the function
$v\mapsto v \mathscr{F}\,'(v)
=\displaystyle\int_0^\chi \frac{v\,\xi(u)du}{(\xi(u)-v)^2}$.

Note that outside the curved part, i.e., when 
$\int_0^\chi \bigl(\xi(u)\bigr)^{-1}du\ge \theta$,
we have $\theta v-v \mathscr{F}'(v)\le 0$ for all $v\in[0, \mathcal{W}^\circ_\chi)$.
That is, the Legendre dual interpretation automatically takes
care of vanishing of the height function
outside the curved part.

\subsection{Asymptotic fluctuations in continuous space TASEP}
\label{sub:cont_TASEP_asympt_fluct_intro}

We now return to the general situation allowing roadblocks.
To formulate the results on fluctuations, let us denote
\begin{equation}\label{eq:cont_d_variance_TW_phase}
	\mathfrak{d}_{TW}=
	\mathfrak{d}_{TW}(\theta,\chi):=
	\left( \int_0^\chi\frac{\xi(u)(\mathfrak{w}^\circ(\theta,\chi)+2\xi(u))}
	{\mathfrak{w}^\circ(\theta,\chi)(\mathfrak{w}^\circ(\theta,\chi)-\xi(u))^4}\,du \right)^{1/3}>0
\end{equation}
and
\begin{equation}\label{eq:cont_d_variance_G_phase}
	\mathfrak{d}_{G}=
	\mathfrak{d}_{G}(\theta,\chi):=
	\left( \frac{\theta}{\mathcal{W}_\chi}-\int_0^\chi
		\frac{\xi(u)(\xi(u)+\mathcal{W}_\chi)}{(\xi(u)-\mathcal{W}_\chi)^3}\,du\right)^{1/2}
	>0
\end{equation}
(the expression under the square root in \eqref{eq:cont_d_variance_G_phase}
is strictly positive in the Gaussian phase
thanks to the monotonicity observed in the proof of 
\Cref{lemma:critical_point_exists_continuous}
and the fact that 
$\mathfrak{d}_{G}$ vanishes when
$\mathfrak{w}^\circ(\theta,\chi)=\mathcal{W}_\chi$, cf. 
\eqref{eq:critical_continuous}).
The kernels and distributions
in the next theorem are described in \Cref{sec:app_B1_TW_etc_distributions}.

\begin{theorem}
	\label{thm:cont_TASEP_fluctuations_in_all_regimes_intro}
	Fix arbitrary $\ell\in \mathbb{Z}_{\ge1}$.
	\begin{enumerate}[\bf1.]
		\item Let $(\theta,\chi)$ be in the Tracy-Widom phase. 
			Fix $s_1,\ldots,s_\ell,r_1,\ldots,r_\ell\in \mathbb{R}$,
			and denote 
			$$t_i:=\theta L+2 \mathfrak{w}^\circ(\theta,\chi)
			\mathfrak{d}_{TW}^2(\theta,\chi) s_i L^{2/3}.$$
			Then
			\begin{multline}
				\label{eq:Airy_multipoint_convergence_cont_TASEP}
				\lim_{L\to+\infty}
				\mathop{\mathrm{Prob}}
				\left( 
					\frac{
						\mathcal{H}(t_i,\chi)-L\mathfrak{h}(\theta,\chi)
						-2L^{2/3}(\mathfrak{w}^\circ(\theta,\chi))^2
						\mathfrak{d}_{TW}^2(\theta,\chi)s_i
					}
					{\mathfrak{w}^\circ(\theta,\chi)\mathfrak{d}_{TW}(\theta,\chi)L^{1/3}}
					> 
					s_i^2-r_i,\;
					i=1,\ldots,\ell 
				\right)
				\\=
				\det\left( \mathbf{1}-\mathsf{A}^{\mathrm{ext}} \right)_{
				\sqcup_{i=1}^{\ell} \{s_i\}\times(r_i,+\infty) }.
			\end{multline}
			In particular, for $\ell=1$ and $s_1=0$ we have
			convergence to the GUE Tracy-Widom distribution:
			\begin{equation*}
				\lim_{L\to+\infty}
				\mathop{\mathrm{Prob}}
				\left( \frac{\mathcal{H}(\theta L,\chi)-L\mathfrak{h}(\theta,\chi)}
				{\mathfrak{w}^\circ(\theta,\chi)\mathfrak{d}_{TW}(\theta,\chi)L^{1/3}}> -r \right)
				=
				F_{GUE}(r),\qquad r\in\mathbb{R}.
			\end{equation*}
		\item Let $(\theta,\chi)$ be at a BBP transition. 
			With $t_i$ as above, the probabilities
			in the left-hand side of 
			\eqref{eq:Airy_multipoint_convergence_cont_TASEP}
			converge to 
			\begin{equation*}
				\det(\mathbf{1}-\widetilde{B}^{\mathrm{ext}}_{m_\chi,(0,\ldots,0 )})
				_{\sqcup_{i=1}^{\ell}\{s_i\}\times(r_i,+\infty)}.
			\end{equation*}
			In particular, for $\ell=1$ we have the following single-time convergence
			to the BBP deformation
			of the GUE Tracy-Widom distribution:
			\begin{equation*}
				\lim_{L\to+\infty}
				\mathop{\mathrm{Prob}}
				\left( 
					\frac{\mathcal{H}(\theta L,\chi)-L
					\mathfrak{h}(\theta,\chi)}{\mathfrak{w}^\circ(\theta,\chi)
					\mathfrak{d}_{TW}(\theta,\chi)L^{1/3}}>-r 
				\right)=
				F_{m_\chi}(r),\qquad  r\in \mathbb{R}.
			\end{equation*}
		\item 
			Let $(\theta_1,\chi),\ldots,(\theta_\ell,\chi)$
			be in the Gaussian phase. Then for $r_1,\ldots,r_\ell\in \mathbb{R} $:
			\begin{equation}
				\lim_{L\to+\infty}
				\mathop{\mathrm{Prob}}
				\left( 
					\frac{\mathcal{H}(\theta_i L,\chi)-L\mathfrak{h}(\theta_i,\chi)}
					{\mathcal{W}_\chi L^{1/2}}>-
					\mathfrak{d}_G(\theta_i,\chi)
					r_i
				\right)
				=
				\det\left( \mathbf{1}-\mathsf{G}_{m_\chi} \right)
				_{\sqcup_{i=1}^{\ell}\{\theta_i\}\times (r_i,+\infty)},
				\label{eq:Gaussian_multipoint_convergence_cont_TASEP}
			\end{equation}
			where the kernel $\mathsf{G}_{m}$ on $\mathbb{R}\times \mathbb{R}$ 
			is expressed through
			\eqref{eq:G_limiting_kernel}
			as 
			$\mathsf{G}_m(\theta,h;\theta',h')
			=
			\widetilde{\mathsf{G}}^{\mathrm{ext}}
			_{m,\frac{\mathfrak{d}_G(\theta',\chi)}{\mathfrak{d}_G(\theta,\chi)}}
			(h;h')$.
			In particular, for $\ell=1$ we have the following Central Limit type theorem
			on convergence to the distribution of the largest
			eigenvalue of the GUE random matrix of size $m_\chi$:
			\begin{equation*}
				\lim_{L\to+\infty}
				\mathop{\mathrm{Prob}}
				\left( 
					\frac{\mathcal{H}(\theta L,\chi)-L\mathfrak{h}(\theta,\chi)}
					{\mathfrak{d}_G(\theta,\chi)\mathcal{W}_\chi L^{1/2}} >-r 
				\right)=G_{m_\chi}(r), 
				\qquad r\in \mathbb{R}.
			\end{equation*}
	\end{enumerate}
\end{theorem}

We prove \Cref{thm:cont_TASEP_fluctuations_in_all_regimes_intro} in
\Cref{sub:cont_Fredholm_asymptotics}.

\subsection{Fluctuation behavior around a traffic jam}
\label{sub:cont_TASEP_asympt_fluct_phase_transition_intro}

Let us now focus on phase transitions of another
type which are caused by 
decreasing jump discontinuities in the speed function
$\xi(\cdot)$ instead of roadblocks.
Let us focus on one such discontinuity 
at a given location $\boldsymbol\upchi>0$
with 
\begin{equation}\label{eq:traffic_jam_definition_cont}
	\lim_{u\to\boldsymbol\upchi-}\xi(u)>
	\xi^r:=\lim_{u\to\boldsymbol\upchi+}\xi(u).
\end{equation}
For simplicity let us assume that there are 
no roadblocks in the interval $[0,\boldsymbol\upchi+c)$ for some $c>0$.
The limiting height function $\mathfrak{h}(\theta,\chi)$
is continuous at $\chi=\boldsymbol\upchi$
if and only if 
$\mathfrak{w}^\circ(\theta,\boldsymbol\upchi)<\xi^r$ 
(cf. \Cref{lemma:critical_point_exists_continuous}).
Note that the value of 
$\mathfrak{w}^\circ(\theta,\boldsymbol\upchi)$
is determined only by the values of $\xi$ on $(0,\boldsymbol\upchi)$ and 
does not depend on $\xi^r$. 
Consider the equation
$\mathfrak{w}^\circ(\theta,\boldsymbol\upchi)=\xi^r$ 
which can be written as (see \eqref{eq:critical_continuous})
\begin{equation}
	\label{eq:critical_traffic_jam_equation_cont}
	\theta=\int_0^{\boldsymbol\upchi}
	\frac{\xi(u)(\xi(u)+\xi^r)}{(\xi(u)-\xi^r)^3}\,du.
\end{equation}
For fixed speed function 
$\xi(\cdot)$ satisfying \eqref{eq:traffic_jam_definition_cont}
let us call the right-hand side of \eqref{eq:critical_traffic_jam_equation_cont}
the
\emph{critical scaled time} $\theta_{\mathsf{cr}}$.
One readily sees that the height function $\mathfrak{h}$
is continuous at $\boldsymbol\upchi$
for $\theta<\theta_{\mathsf{cr}}$,
and becomes discontinuous 
for $\theta>\theta_{\mathsf{cr}}$.
Further analysis 
(performed in \Cref{sub:cont_phase_trans_fluctuations})
shows that for $\theta=\theta_{\mathsf{cr}}$
the height function is continuous at $\boldsymbol\upchi$
while its right derivative at $\boldsymbol\upchi$
is infinite. See
\Cref{fig:traffic_jam_transition} for an illustration.
From the limit shape result
it follows that
for $\theta>\theta_{\mathsf{cr}}$, at time $\theta L$
there are $O(L)$
particles in a small right neighborhood
of $\boldsymbol\upchi$. 
We thus say that the critical scaled time
$\theta_{\mathsf{cr}}$
corresponds to 
the formation of a \emph{traffic jam}.

\begin{figure}[htpb]
	\centering
	\includegraphics[width=.3\textwidth]{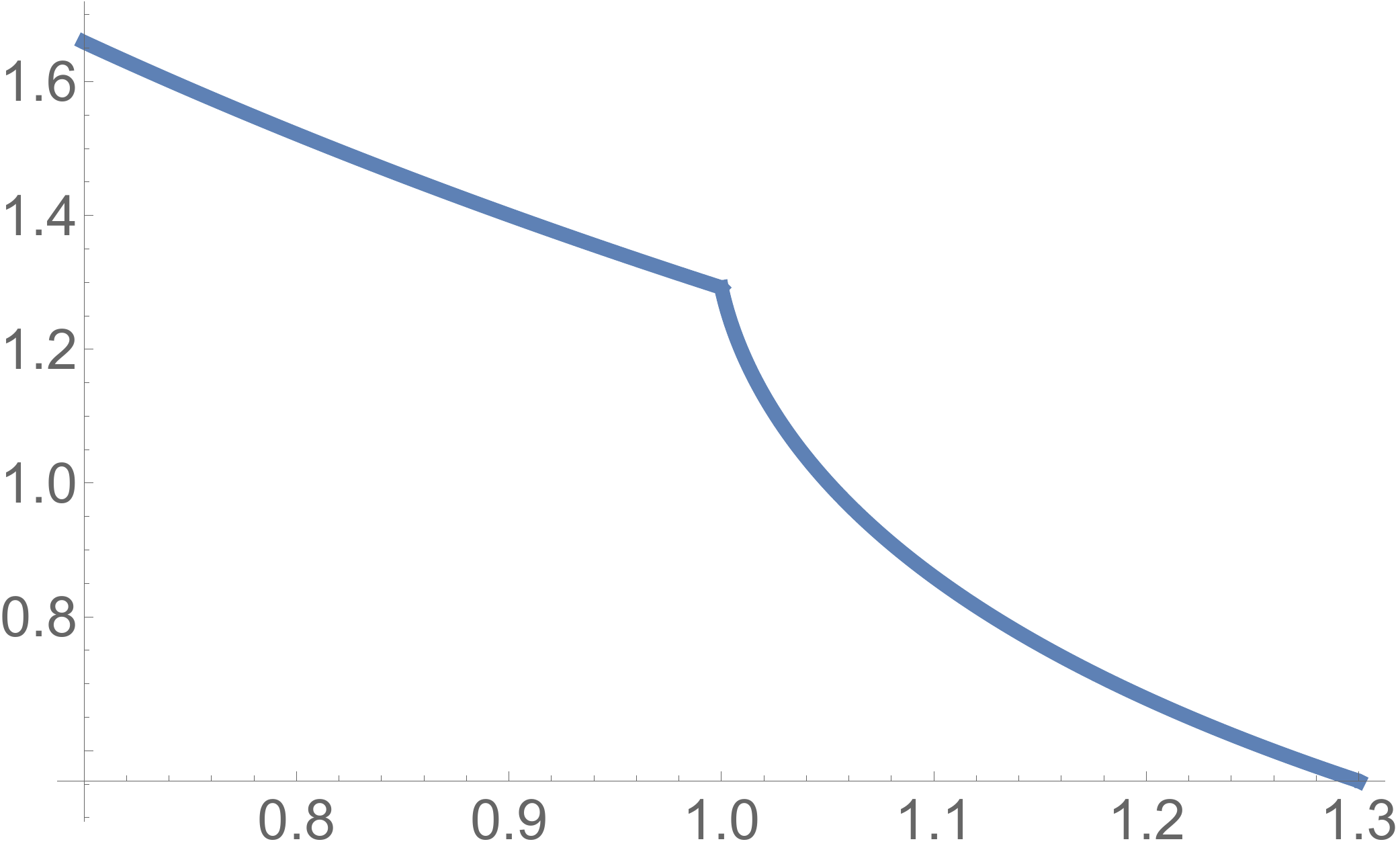}\qquad 
	\includegraphics[width=.3\textwidth]{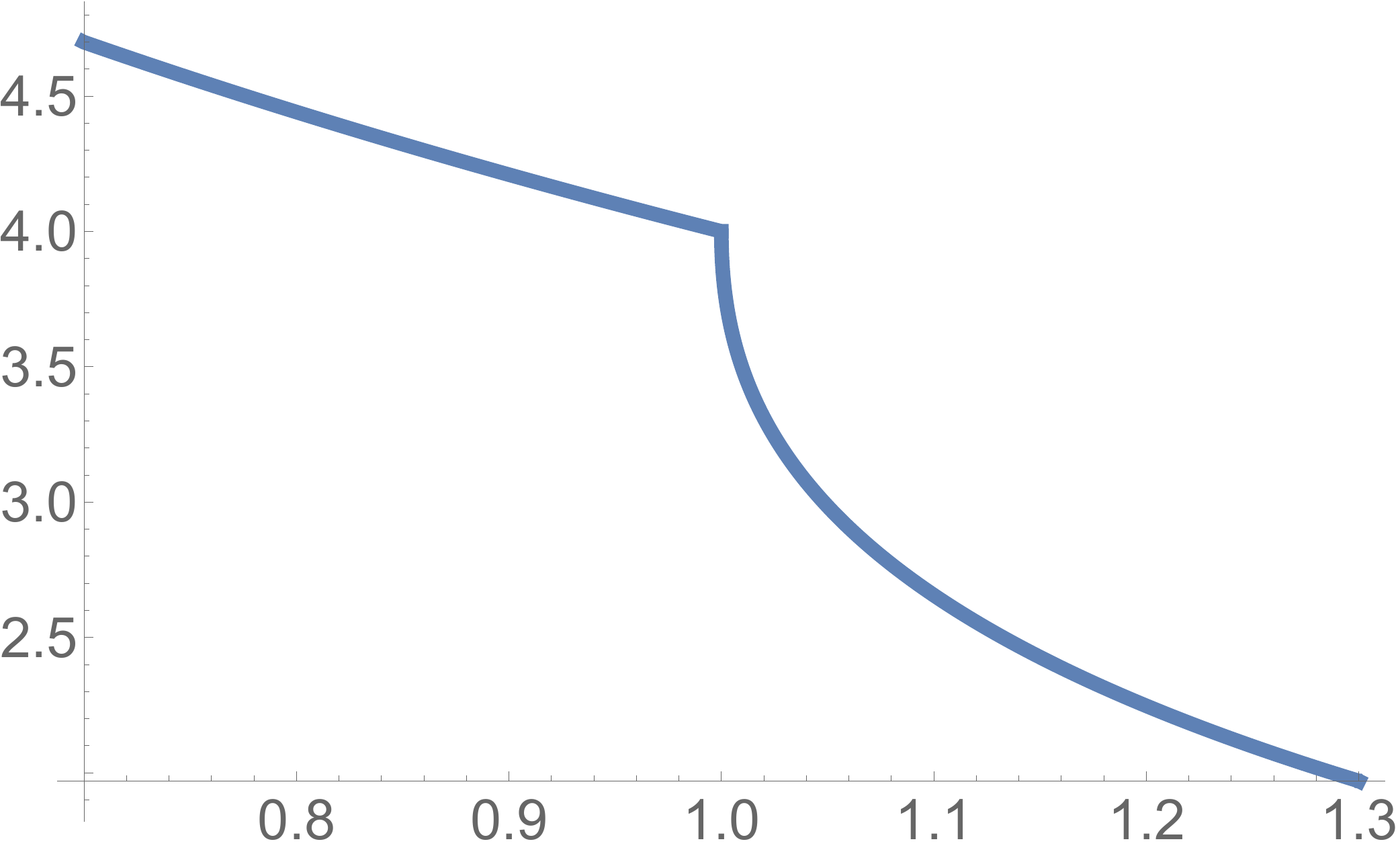}\qquad 
	\includegraphics[width=.3\textwidth]{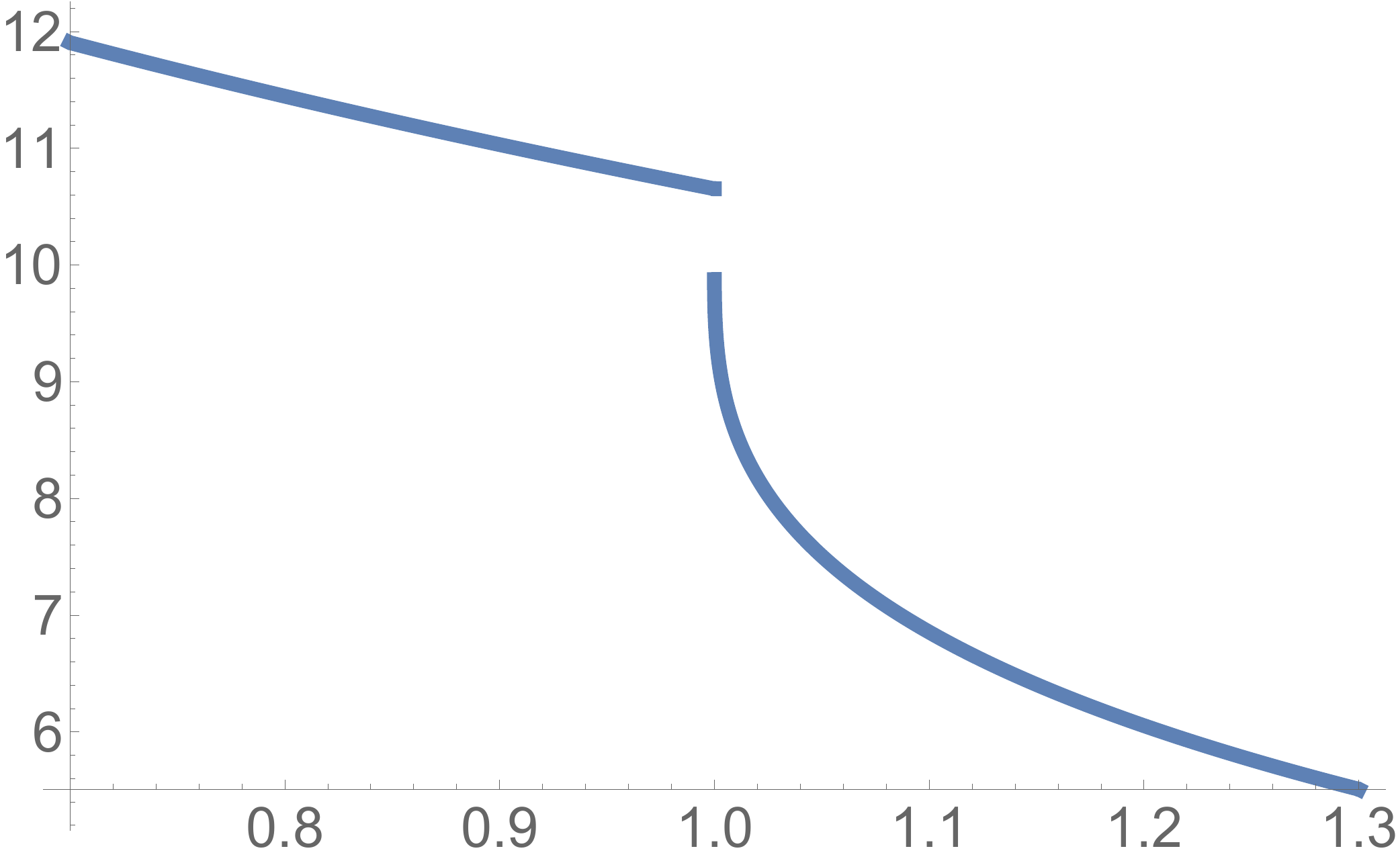}
	\caption{Plots of the limiting 
		height function showing the formation of a traffic jam.
		Left: $\theta<\theta_{\mathsf{cr}}$,
		Center: $\theta=\theta_{\mathsf{cr}}$,
		Right: $\theta>\theta_{\mathsf{cr}}$.}
	\label{fig:traffic_jam_transition}
\end{figure}

The fluctuations of the random height function 
$\mathcal{H}(t,\chi)$ around the traffic jam for every fixed $\chi$
on both sides of $\boldsymbol\upchi$
are governed by the Airy kernel as in the first part of
\Cref{thm:cont_TASEP_fluctuations_in_all_regimes_intro}.
However, the normalizing factor 
$\mathfrak{w}^\circ(\theta,\chi) \mathfrak{d}_{TW}(\theta,\chi)$
has a jump discontinuity at $\chi=\boldsymbol\upchi$.

To further explore behavior of fluctuations
around a traffic jam, we 
consider
a more general regime when $\chi=\chi(L)>\boldsymbol\upchi$ depends 
on $L$ and converges to $\boldsymbol\upchi$ as $L\to+\infty$.
To simplify notation and computations let us take a particular case of a
piecewise constant speed function
\begin{equation}\label{eq:traffic_jam_xi_nice_example}
	\xi(u)=\begin{cases}
		1,& 0\le u\le 1;\\
		1/2,& u>1.
	\end{cases}
\end{equation}
The critical time corresponding to formation of the 
traffic jam at $\boldsymbol\upchi=1$
is $\theta_{\mathsf{cr}}=12$,
see \eqref{eq:critical_traffic_jam_equation_cont}.
We find that there is a particular scale 
at which the fluctuations of the height function 
are governed by a deformation
of the Tracy-Widom distribution
(defined in \Cref{sub:GUE_deformed_kernels}).
This deformation can be obtained in a limit from 
kernels considered in 
\cite{BorodinPeche2009}
and thus has a random matrix interpretation
(see \Cref{ssub:Borodin_Peche_connection} for details).
At other scales the fluctuations lead to the usual Airy
kernel, but close to the slowdown the constants are affected
by the change in $\xi(\cdot)$ as well. Far from the slowdown
the constants are the same as in 
\eqref{eq:Airy_multipoint_convergence_cont_TASEP}
with $\chi$ depending on $L$.
In detail, we show the following:

\begin{theorem}
	\label{thm:traffic_jam_deformed_GUE}
	With the above notation,
	let $\chi=\chi(L)=1+10\upepsilon(L)$, where $\upepsilon(L)>0$
	and $\upepsilon(L)\to0$ as $L\to+\infty$
	(the factor $10$ makes final formulas simpler).
	Let 
	$\mathfrak{w}^\circ=\mathfrak{w}^\circ(12,1+10\upepsilon(L))$,
	$\mathfrak{h}=\mathfrak{h}(12,1+10\upepsilon(L))$,
	$\mathfrak{d}_{TW}=\mathfrak{d}_{TW}(12,1+10\upepsilon(L))$
	be the quantities defined in 
	\Cref{sub:limit_shape_cont_TASEP,sub:cont_TASEP_asympt_fluct_intro}.
	Fix $s_1,\ldots,s_\ell,r_1,\ldots,r_\ell\in \mathbb{R}$.
	Depending on the rate at which $\upepsilon(L)\to0$
	there are three fluctuation regimes:
	\begin{enumerate}[\bf1.]
		\item (close to the slowdown)
			Let $\upepsilon(L)\ll L^{-4/3-\gamma}$ for some $\gamma>0$.
			Define 
			\begin{equation*}
				t_i=12L+\mathfrak{w}^\circ\mathfrak{d}_{TW}^2 2^{-1/3}s_i L^{2/3}.
			\end{equation*}
			Then 
			\begin{multline*}
				\lim_{L\to+\infty}
				\mathop{\mathrm{Prob}}
				\left( 
					\frac{\mathcal{H}(t_i,1+10\upepsilon(L))
					-4L-(\mathfrak{w}^\circ)^2\mathfrak{d}_{TW}^2 2^{-1/3}s_i L^{2/3}}
					{\mathfrak{w}^\circ\mathfrak{d}_{TW}2^{-2/3}L^{1/3}}
					>s_i^2-r_i,\ i=1,\ldots,\ell  \right)
				\\=
				\det(\mathbf{1}-\mathsf{A}^{\mathrm{ext}})_{\sqcup_{i=1}^{\ell}
				\left\{ s_i \right\}\times(r_i,+\infty)}.
			\end{multline*}
		\item (far from the slowdown)
			Let $\upepsilon(L)\gg L^{-4/3+\gamma}$ for some $\gamma\in(0,\frac{4}{3})$.
			Define
			\begin{equation*}
				t_i=12L+2\mathfrak{w}^\circ\mathfrak{d}_{TW}^2 s_i L^{2/3}
				.
			\end{equation*}
			Then
			\begin{multline*}
				\lim_{L\to+\infty}
				\mathop{\mathrm{Prob}}
				\left( 
					\frac{\mathcal{H}(t_i,1+10\upepsilon(L))
					-\mathfrak{h}(12,1+10\upepsilon(L))L
					-
					2(\mathfrak{w}^\circ)^2\mathfrak{d}_{TW}^2 s_i L^{2/3}}
					{\mathfrak{w}^\circ\mathfrak{d}_{TW}L^{1/3}}
					>s_i^2-r_i,\ i=1,\ldots,\ell   \right)
				\\=
				\det(\mathbf{1}-\mathsf{A}^{\mathrm{ext}})_{\sqcup_{i=1}^{\ell}
				\left\{ s_i \right\}\times(r_i,+\infty)}.
			\end{multline*}
		\item (critical scale)
			Let $\upepsilon(L)=10^{-4/3}\delta L^{-4/3}$, where $\delta>0$ is fixed.
			Define
			\begin{equation*}
				t_i=12L+\mathfrak{w}^\circ \mathfrak{d}_{TW}^2 2^{-1/3}s_i L^{2/3}.
			\end{equation*}
			The joint fluctuations at different times
			of the random height function around the limit shape
			$4L$ are described by a deformation
			of the extended Airy kernel
			defined by \eqref{eq:tilde_A_2_ext_deformed}:
			\begin{multline*}
				\lim_{L\to+\infty}
				\mathop{\mathrm{Prob}}
				\left( 
					\frac{
						\mathcal{H}(t_i,1+10^{-1/3}\delta L^{-4/3})
						-4L-(\mathfrak{w}^\circ)^2\mathfrak{d}_{TW}^2 2^{-1/3}s_i L^{2/3}
					}
					{\mathfrak{w}^\circ\mathfrak{d}_{TW}2^{-2/3}L^{1/3}}
					>s_i^2+2s_i\delta^{1/4}-r_i
					,\ i=1,\ldots,\ell  
				\right)
				\\=
				\det(\mathbf{1}-\widetilde{\mathsf{A}}^{\mathrm{ext}, \delta})
				_{\sqcup_{i=1}^{\ell}
				\left\{ s_i \right\}\times(r_i,+\infty)}.
			\end{multline*}
			In particular, for $\ell=1$ and\footnote{The 
			deformed Airy kernel is not invariant
			with respect to simultaneous translations of the 
			$s_i$'s, so we specialize $s_1=0$ to get the simplest 
			one-point distribution $F_{GUE}^{(\delta,0)}$.} 
			$s_1=0$ we have the 
			convergence to a deformation of the GUE Tracy-Widom distribution
			\eqref{eq:F_2_GUE_definition_deformed}:
			\begin{equation*}
				\lim_{L\to+\infty}
				\mathop{\mathrm{Prob}}
				\left( 
					\frac{\mathcal{H}(12 L,1+10^{-1/3}\delta L^{-4/3})-4L}
					{\mathfrak{w}^\circ\mathfrak{d}_{TW}2^{-2/3}L^{1/3}}
					>
					-r
				\right)
				=F_{GUE}^{(\delta,0)}(r),
				\qquad r\in \mathbb{R}.
			\end{equation*}
	\end{enumerate}
\end{theorem}

We prove \Cref{thm:traffic_jam_deformed_GUE} in 
\Cref{sub:cont_phase_trans_fluctuations}.

\section{Asymptotics of continuous space TASEP. Proofs}
\label{sec:asymptotics}

\subsection{Critical points}
\label{sub:cont_double_critical_point}

Recall the notation
\eqref{eq:cont_G_L_function_def}:
\begin{equation*}
	G(v)=G(v;\theta,\chi,h)=-\theta v+h\log v
	+
	\displaystyle\int_0^\chi \frac{\xi(u)du}{\xi(u)-v}.
\end{equation*}
The 
correlation kernel from \Cref{thm:intro_convergence_discrete_to_continuous_TASEP}
takes the form
\begin{align}
	&\contKernel(t, x; t', x')
	=
	-
	\mathbf{1}_{t>t'}\mathbf{1}_{x\ge x'}
	\frac{(t-t')^{x-x'}}{(x-x')!}
	\nonumber\\
	&\hspace{30pt}+
	\frac{1}{(2\pi\iu)^2} \oint \oint \frac{dw dz}{z(z-w)}
	\exp
	\left\{ 
		L\bigl(
			G(w;\tfrac{t'}L,\chi,\tfrac{x'}L)-G(z;\tfrac{t}L,\chi,\tfrac{x}L)
		\bigr)
	\right\}
	\label{eq:cont_K_through_G_fucntion}
	\\\nonumber
	&\hspace{190pt}
	\times
	\frac{(\xi(0)-z)}{(\xi(0)-w)}
	\prod\limits_{b\in\RoadblockSet\colon b < \chi} \frac{\xi(b)- p(b)w}{\xi(b)- p(b) z} 
	\cdot \frac{\xi(b)-z}{\xi(b)-w},
\end{align}
where we used the observation 
$\int_0^\chi \frac{\xi(u)du}{\xi(u)-v}=\chi+\int_0^\chi\frac{v\,du}{\xi(u)-v}$,
and the additional summand $\chi$ cancels out in $G(w)-G(z)$.
The integration contours in \eqref{eq:cont_K_through_G_fucntion}
are described in \Cref{thm:intro_convergence_discrete_to_continuous_TASEP}.

The asymptotic behavior of the kernel as $\contKernel$
as $L\to+\infty$ 
is analyzed via steepest descent method
which in turn relies on finding double critical points
of the function $G$, i.e., 
those $v$ for which 
$\frac{\partial}{\partial v}G(v)=\frac{\partial^2}{\partial v^2}G(v)=0$
and $\frac{\partial^3}{\partial v^3}G(v)\ne 0$.
We turn to double critical points
because we are 
interested in the left edge of the determinantal
point process $\widetilde{\mathfrak{L}}_\ell$.
The equations for the double critical points
of $G(v;\theta,\chi,h)$ can be rewritten the following form:
\begin{align}
	\label{eq:cont_crit_pts_eq_1}
	&\int_0^\chi
	\frac{\xi(u)(v+\xi(u))}{\left( \xi(u)-v \right)^3}\,du=\theta;
	\\
	\label{eq:cont_crit_pts_eq_2}
	&
	h=\theta v-\int_0^\chi \frac{\xi(u)v}{\left( \xi(u)-v \right)^2}\,du.
\end{align}
Recall that 
$\mathcal{W}^\circ_\chi$
is the essential minimum of the 
function $\xi(u)$ for $0<u<\chi$,
and 
$\mathcal{W}_\chi$
is the minimum of $\mathcal{W}^\circ_\chi$, $\xi(0)$,
and values of $\xi$ at all the roadblocks on $(0,\chi)$,
see \eqref{range_of_Speed_notation}--\eqref{essential_range_of_Speed_notation}.
By \Cref{lemma:critical_point_exists_continuous},
for $(\theta,\chi)$ in the curved part
(\Cref{def:curver_part_continuous_regular})
the first equation 
\eqref{eq:cont_crit_pts_eq_1} 
has a unique solution (denoted by $\mathfrak{w}^\circ=\mathfrak{w}^\circ(\theta,\chi)$)
in $v$ belonging to $(0,\mathcal{W}^\circ_\chi)$.

Recall the notation
$\mathfrak{w}(\theta, \chi)
=
\min(\mathfrak{w}^\circ(\theta,\chi),\mathcal{W}_\chi)$
and limit shape
$\mathfrak{h}(\theta,\chi)$
from \Cref{def:limit_shape_continuous}.
In the Tracy-Widom phase the limit shape
$\mathfrak{h}(\theta,\chi)$
is defined by plugging 
$\mathfrak{w}^\circ(\theta,\chi)$ into the
second double critical point equation \eqref{eq:cont_crit_pts_eq_2},
so that 
$\mathfrak{w}(\theta,\chi)=\mathfrak{w}^\circ(\theta,\chi)$ is a double critical
point of
$G(v;\theta,\chi,\mathfrak{h}(\theta,\chi))$.
In the Gaussian phase and at the BBP transition,
$\mathfrak{w}(\theta, \chi)=\mathcal{W}_\chi$ is a single critical
point of 
$G(v;\theta,\chi,\mathfrak{h}(\theta,\chi))$.

\subsection{Estimates on contours}
\label{sub:cont_estimates_on_contours}

Here we
prove estimates of the real part of the function $G(v;\theta,\chi,\mathfrak{h}(\theta,\chi))$ on 
the following contours:

\begin{definition}
	\label{def:contours_for_discrete_AND_continuous}
	For $r>0$ let
	$\Gamma_r$ be the counterclockwise 
	circle centered at zero and 
	passing through $r$.
	Let $C_{r,\varphi}$ (where $0<\varphi<\pi/2$) be the contour 
	\begin{equation*}
		C_{r,\varphi}:=
		\{r-\iu ye^{\iu \varphi \mathop{\mathrm{sgn}}(y)}\colon y\in \mathbb{R}  \}
	\end{equation*}
	composed of two lines passing through $r$
	which form angle $\varphi$ with the vertical axis.
	In this section we mostly need the contour $C_{r,\frac{\pi}{4}}$ 
	which will be denoted simply by $C_r$.
	See \Cref{fig:steep_contours} for an illustration.
\end{definition}

\begin{figure}[htpb]
	\centering
	\includegraphics[width=.3\textwidth]{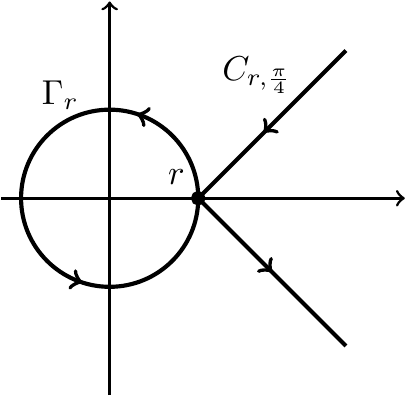}
	\caption{Contours $\Gamma_r$ and $C_{r}=C_{r,\frac{\pi}{4}}$.}
	\label{fig:steep_contours}
\end{figure}

We need slightly different arguments depending on the phase
(\Cref{def:TW_G_BBP_phases}). We start from the Tracy-Widom one.

\begin{lemma}
	\label{lemma:cont_G_estimates_z_circ_contour_TW}
	Let $(\theta,\chi)$ be in the Tracy-Widom phase.
	The contour $\Gamma_{\mathfrak{w}^\circ(\theta,\chi)}$ is 
	steep ascent for the function $\Re G(z;\theta,\chi,\mathfrak{h}(\theta,\chi))$
	in the sense that the function attains its minimal value 
	at $z=\mathfrak{w}^\circ(\theta,\chi)$.
\end{lemma}
\begin{proof}
	For shorter notation we denote $\mathfrak{h}=\mathfrak{h}(\theta,\chi)$
	and $\mathfrak{w}=\mathfrak{w}^\circ(\theta,\chi)$ in the proof of this lemma 
	and \Cref{lemma:cont_G_estimates_w_opened_contour_TW} below.

	From \eqref{eq:cont_crit_pts_eq_1}--\eqref{eq:cont_crit_pts_eq_2} we can write
	\begin{equation*}
		G(z;\theta,\chi,\mathfrak{h})=\int_{0}^{\chi}S(z;\xi(u))du,
	\end{equation*}
	where
	\begin{equation*}
		S(z)=S(z;\xi(u)):=\log z \left ( \frac{\xi(u)(\xi(u)+\mathfrak{w})
			\mathfrak{w}}{(\xi(u)-\mathfrak{w})^3}-\frac{\xi(u)\mathfrak{w}}{(\xi-\mathfrak{w})^2}\right
		)-\frac{z \xi(u) (\xi(u)+\mathfrak{w})}{(\xi(u)-\mathfrak{w})^3}+\frac{\xi(u)}{\xi(u)-z}.
	\end{equation*}
	Denote $\gamma(u)=\xi(u)- \mathfrak{w}.$ We know that $\xi(u)\geq
	 \mathfrak{w},$ thus, $\gamma(u)$ is nonnegative. We get
	 \begin{equation}\label{eq:cont_proofs_S_function}
		 S(z;\xi(u))=(\mathfrak{w}+\gamma(u))\left(\frac{1}{\mathfrak{w}+\gamma(u)-z}-
			\frac{z (2\mathfrak{w}+\gamma(u))}{\gamma(u)^3}+\frac{2 \mathfrak{w}^2 \log z}{\gamma(u)^3}\right).
	\end{equation}
	Let us prove
	$\frac{\partial}{\partial \varphi}\Re S(\mathfrak{w}e^{\iu \varphi})>0$
	for $0<\varphi<\pi$. The case $-\pi<\varphi<0$ is symmetric.
	Straightforward computation gives
	(we are omitting the dependence on $u$ in the notation)
	\begin{equation}
		\label{eq:der_Re_S_on_cicle}
			\frac{\partial}{\partial \varphi}
			\Re S(\mathfrak{w}e^{\iu \varphi})
			=
			\frac{16 \mathfrak{w}^2 (\gamma +\mathfrak{w})^2 (\gamma +2 \mathfrak{w}) \sin
			^3\left(\frac{\varphi}{2}\right) \cos
			\left(\frac{\varphi}{2}\right) 
			\left(
				\gamma ^2+
				\mathfrak{w}^2(1-\cos\varphi)+
				\gamma\mathfrak{w}(1-\cos\varphi)
			\right)
			}
			{\gamma^3 \left(\gamma ^2+2 \mathfrak{w}^2+2 \gamma  \mathfrak{w}-2 \mathfrak{w} (\gamma
			+\mathfrak{w}) \cos \varphi\right)^2}.
  \end{equation}
	We see that for $\pi>\varphi\geq 0$ this quantity is positive,
	which implies the statement.
\end{proof}

\begin{lemma}
	\label{lemma:cont_G_estimates_w_opened_contour_TW}
	Let $(\theta,\chi)$ be in the Tracy-Widom phase.
	The contour $C_{\mathfrak{w}^\circ(\theta,\chi)}$ is 
	steep descent for the function $\Re G(w;\theta,\chi,\mathfrak{h}(\theta,\chi))$
	in the sense that the function attains its maximal value 
	at $w=\mathfrak{w}^\circ(\theta,\chi)$.
\end{lemma}
\begin{proof}
	Using the notation from the proof of
	\Cref{lemma:cont_G_estimates_z_circ_contour_TW}
	we will show that
	$\frac{\partial}{\partial s}\Re S(\mathfrak{w}+s e^{\iu \frac{\pi}{4}})<0$ for $s>0$
	(the case $s<0$ and $-\frac{\pi}{4}$ is symmetric). 
	This would imply the statement of the
	proposition. 
	A straightforward computation gives that this derivative
	is (up to an obviously positive denominator) equal to
	\begin{equation*}
	\begin{split}
		&
	-s^2(\mathfrak{w}+\gamma)\bigl(\gamma(\sqrt 2 s^2-4 s\gamma+3\sqrt
	2\gamma^2)(s^2+\mathfrak{w}^2)+2Q(s,\gamma )\mathfrak{w}\bigr), \quad \text{where}\\
	&\hspace{180pt}
	Q(s,\gamma )=\sqrt 2 s^4-3 s^3 \gamma+2\sqrt 2 s^2 \gamma^2-s \gamma^3+\sqrt 2 \gamma^4
	\end{split}
	\end{equation*}
	(here we omitted the dependence on $u$).
	The discriminant of $\sqrt 2 s^2-4 s\gamma+3\sqrt
	2\gamma^2$ in $s$ is $-8\gamma^2,$ so this expression is
	positive.
	The discriminant of $Q(s,\gamma )$ in
	$\gamma$ is $1684 s^{12}>0,$ so $Q(s,\gamma )$ either has all real or all
	nonreal complex
	roots in $\gamma$. 
	Note that
	$$\frac{\partial}{\partial \gamma} Q(s,\gamma )=(4\sqrt 2 \gamma -3
	s)(s^2+\gamma^2),$$
	which has only one root in $\gamma$. 
	Therefore, $Q(s,\gamma )$ has only nonreal roots and thus preserves sign.
	It is always positive because it is positive for $\gamma=0$.
	This shows that 
	$\frac{\partial}{\partial s}\Re S$ is negative, which 
	implies the claim.
\end{proof}

Let us now turn to the Gaussian phase.

\begin{lemma}
	\label{lemma:cont_G_estimates_z_circ_contour_Gauss}
	Let $(\theta,\chi)$ be in the Gaussian phase or at a BBP transition.
	The contour $\Gamma_{\mathcal{W}_\chi}$ is 
	steep ascent for the function $\Re G(z;\theta,\chi,\mathfrak{h}(\theta,\chi))$
	in the sense that the function attains its minimal value 
	at $z=\mathcal{W}_\chi$.
\end{lemma}
\begin{proof}
	Throughout the proof (and in the proof of \Cref{lemma:cont_G_estimates_w_opened_contour_Gauss} below)
	we use the shorthand notation 
	$\mathfrak{w}^\circ=\mathfrak{w}^\circ(\theta,\chi)$
	and $\mathcal{W}=\mathcal{W}_\chi$.

	Let us write $G(z;\theta,\chi,\mathfrak{h}(\theta,\chi))$
	again as an integral from $0$ to $\chi$.
	While $\mathfrak{h}$
	depends on $\mathcal{W}$ (\Cref{def:limit_shape_continuous}),
	we cannot express $\theta$ through $\mathcal{W}$.
	However, we can still write $\theta$ in terms of
	the solution $\mathfrak{w}^\circ(\theta,\chi)\ge \mathcal{W}_\chi$
	to equation \eqref{eq:cont_crit_pts_eq_1}.
	This allows to write
	\begin{equation*}
		G(z;\theta,\chi,\mathfrak{h}(\theta,\chi))=
		\int_0^\chi \tilde S(z;\xi(u))du,
	\end{equation*}
	where
	\begin{equation*}
		\tilde S(z;\xi)
		:=
		-\log z\,\frac{\xi \mathcal{W}}{(\xi-\mathcal{W})^2}
		+\frac{\xi}{\xi-z}
		+(\mathcal{W}\log z-z)\frac{\xi(\xi+\mathfrak{w}^\circ)}{(\xi-\mathfrak{w}^\circ)^3}.
	\end{equation*}

	We estimate for $0\le \varphi<\pi$ (the case $-\pi<\varphi<0$ is symmetric)
	\begin{align*}
		\frac{\partial}{\partial\varphi}\Re\tilde S(\mathcal{W}e^{\iu \varphi},\xi)
		&=
		\frac{\mathcal{W} \xi  \left(\mathcal{W}^2-\xi ^2\right) \sin \varphi}
		{\left(\mathcal{W}^2-2 \mathcal{W} \xi  \cos \varphi+\xi ^2\right)^2}
		+
		\frac{\mathcal{W} \xi  (\mathfrak{w}^\circ+\xi ) \sin \varphi}{(\xi -\mathfrak{w}^\circ)^3}
		\\&\ge
		\frac{\mathcal{W} \xi  \left(\mathcal{W}^2-\xi ^2\right) \sin \varphi}
		{\left(\mathcal{W}^2-2 \mathcal{W} \xi  \cos \varphi+\xi ^2\right)^2}
		+
		\frac{\mathcal{W} \xi  (\mathcal{W}+\xi ) \sin \varphi}{(\xi - \mathcal{W})^3},
	\end{align*}
	where we used $\mathcal{W}\xi\sin\varphi\ge 0$,
	$\mathfrak{w}^\circ\ge\mathcal{W}$, and that the function
	$u\mapsto \frac{\xi+u}{(\xi-u)^3}$ is increasing for $0<u<\xi$.
	The right-hand side coincides with 
	\eqref{eq:der_Re_S_on_cicle}
	with $\mathfrak{w}$ replaced by $\mathcal{W}$,
	and thus is positive as shown in the proof of \Cref{lemma:cont_G_estimates_z_circ_contour_TW}.
	Therefore, 
	$\frac{\partial}{\partial\varphi}\Re\tilde S(\mathcal{W}e^{\iu \varphi},\xi)>0$
	for $0< \varphi<\pi$,
	and we are done.
\end{proof}

\begin{lemma}
	\label{lemma:cont_G_estimates_w_opened_contour_Gauss}
	Let $(\theta,\chi)$ be in the Gaussian phase or at a BBP transition.
	The contour $C_{\mathcal{W}_\chi}$ is 
	steep descent for the function $\Re G(w;\theta,\chi,\mathfrak{h}(\theta,\chi))$
	in the sense that the function attains its maximal value 
	at $w=\mathcal{W}_\chi$.
\end{lemma}
\begin{proof}
	Using the notation from the proof of \Cref{lemma:cont_G_estimates_z_circ_contour_Gauss}
	let us show that 
	$\frac{\partial}{\partial s}\tilde S(\mathcal{W}+s e^{\iu\frac{\pi}{4}},\xi)<0$
	for $s<0$ (the case of the line at angle $-\frac{\pi}{4}$ in the lower 
	half plane is symmetric).
	This derivative is equal to 
	\begin{multline*}
		-
		\frac{\xi  \mathcal{W} \left(2 s+\sqrt{2} \mathcal{W}\right)}
		{2 \left(s^2+\sqrt{2} s \mathcal{W}+\mathcal{W}^2\right) (\mathcal{W}-\xi )^2}
		+
		\frac{\xi  \left(\sqrt{2} s^2+4 s (\mathcal{W}-\xi )+\sqrt{2} (\mathcal{W}-\xi )^2\right)}
		{2 \left(s^2+\sqrt{2} s (\mathcal{W}-\xi )+(\mathcal{W}-\xi )^2\right)^2}
		\\-
		\frac{\xi  s^2 (\mathfrak{w}^\circ+\xi )}{\sqrt{2} (\xi -\mathfrak{w}^\circ)^3 \left(s^2+\sqrt{2} s \mathcal{W}+\mathcal{W}^2\right)}.
	\end{multline*}
	Again, in the last summand we can replace $\mathfrak{w}^\circ$ by $\mathcal{W}$ 
	by the monotonicity of
	$u\mapsto \frac{\xi+u}{(\xi-u)^3}$
	as in the previous lemma, and 
	the whole expression may only decrease.
	Then we use the proof of \Cref{lemma:cont_G_estimates_w_opened_contour_TW}
	which implies that 
	$\frac{\partial}{\partial s}\tilde S(\mathcal{W}+s e^{\iu\frac{\pi}{4}},\xi)<0$,
	as desired.
\end{proof}

We need two more statements about higher derivatives of the function $G$.
\begin{lemma}
	\label{lemma:cont_3_derivative}
	Let $(\theta,\chi)$ be in the curved part.
	We have 
	\begin{equation*}
		\frac{\partial^3}{\partial^3 v} \Big\vert_{v=\mathfrak{w}(\theta,\chi)}G(v;\theta,\chi,\mathfrak{h}(\theta,\chi))>0.
	\end{equation*}
\end{lemma}
\begin{proof}
	We have
	$G^{(3)}(\mathfrak{w})
	=
	\frac{2\mathfrak{h}}{\mathfrak{w}^3}+\int_0^\chi \frac{6\xi(u)du}{(\xi(u)-\mathfrak{w})^4}>0$, as desired.
\end{proof}
\begin{lemma}
	\label{lemma:4th_der_ReG}
	Let $(\theta,\chi)$ be in the curved part.
	Along the contour
	$\Gamma_{\mathfrak{w}(\theta,\chi)}$ 
	the first $m$ derivatives of $\Re G(z;\theta,\chi,\mathfrak{h}(\theta,\chi))$ at
	$\mathfrak{w}(\theta,\chi)$
	vanish while the $(m+1)$-st one is nonzero,
	where $m=3$ in the Tracy-Widom phase and at a BBP transition, 
	and $m=1$ in the Gaussian phase.
	Along the contour
	$C_{\mathfrak{w}(\theta,\chi)}$
	the first two derivatives
	of $\Re G(z;\theta,\chi,\mathfrak{h}(\theta,\chi))$ at 
	$\mathfrak{w}(\theta,\chi)$
	vanish while the third one is nonzero.
\end{lemma}
\begin{proof}
	This is checked in a straightforward way.
\end{proof}

\subsection{Deformation of contours and behavior of the kernel}
\label{sub:cont_deformation_of_contours}

Assume that $(\theta,\chi)$ is in the curved part
and we scale 
$t,t'=\theta L+o(L)$, 
$x,x'=\mathfrak{h}(\theta,\chi)L+o(L)$
(more precise scaling depends on the phase and
is described below in this subsection).
Let us deform the $z$ and $w$ integration contours
in the correlation kernel 
\eqref{eq:cont_K_through_G_fucntion}
to the steep ascent/descent contours $\Gamma_{\mathfrak{w}(\theta,\chi)}$ and
$C_{\mathfrak{w}(\theta,\chi)}$, respectively.

Since $\mathfrak{w}(\theta,\chi)\le \mathcal{W}_\chi=\min \Xi_\chi$,
see \eqref{range_of_Speed_notation},
the $z$ contour can be deformed to $\Gamma_{\mathfrak{w}(\theta,\chi)}$
without passing through any 
singularities.

To deform the $w$ contour we need to open it
up to infinity. Fix sufficiently large $L$.
Since $(\theta,\chi)$ is in the
curved part, we have
$\theta,\chi>0$.
Then the terms $-L\theta w-L\int_0^\chi\frac{\xi(u)du}{w-\xi(u)}$
in the exponent in the integrand
have large negative real part 
for $\Re w\gg1$, and thus dominate the behavior of the integrand for large $|w|$
if $w$ is
in the right half plane.
Therefore, we can deform the $w$ contour to the desired one.
(In the Gaussian phase or at a BBP transition we 
require, in addition, that locally $w$ passes
strictly to the right
of the pole at $w=\mathcal{W}_\chi$.)

We can now obtain the asymptotic behavior of the correlation kernel
$\contKernel$ \eqref{eq:cont_K_through_G_fucntion} close to the left edge of the determinantal
point process 
$\widetilde{\mathfrak{L}}_{\ell}$.
Recall the quantity $\mathfrak{d}_{TW}=\mathfrak{d}_{TW}(\theta,\chi)>0$ \eqref{eq:cont_d_variance_TW_phase}.
\begin{proposition}[Kernel asymptotics, Tracy-Widom phase]
	\label{prop:cont_K_behavior}
	Let $(\theta,\chi)$ be in the Tracy-Widom phase and
	scale the parameters
	as
	\begin{equation}
		\label{eq:ttxx_cont_scaling_TW_phase}
		\begin{split}
			t=\theta L+2\mathfrak{w}^\circ \mathfrak{d}_{TW}^2 s' L^{2/3},
			\qquad 
			x&=\lfloor\mathfrak{h}L +
			2(\mathfrak{w}^\circ)^2 \mathfrak{d}_{TW}^2 s'
			L^{2/3}+  \mathfrak{w}^\circ\mathfrak{d}_{TW} (s'^2-h') L^{1/3}\rfloor,\\
			t'=\theta L+2\mathfrak{w}^\circ \mathfrak{d}_{TW}^2 s L^{2/3},
			\qquad 
			x'&=\lfloor \mathfrak{h}L +2(\mathfrak{w}^\circ)^2 \mathfrak{d}_{TW}^2 s L^{2/3}+
			\mathfrak{w}^\circ\mathfrak{d}_{TW} (s^2-h) L^{1/3}  \rfloor,
		\end{split}
	\end{equation}
	where $s,s',h,h'\in \mathbb{R}$ are arbitrary.
	Then as $L\to+\infty$ 
	we have
	\begin{equation}\label{eq:cont_K_TW_behavior_statement}
		\contKernel(t,x;t',x')
		=
		e^{f_2L^{2/3}+f_1L^{1/3}}
		\,
		\frac{1+O(L^{-1/3})}
		{L^{1/3}\mathfrak{d}_{TW}\mathfrak{w}^\circ}\,\widetilde{\mathsf{A}}^{\mathrm{ext}}(s,h;s',h')
		,
	\end{equation}
	where the constant in $O(L^{-1/3})$ is uniform 
	in $h,h'$ belonging to compact intervals,
	but may depend on $s,s'$.
	Here
	$\widetilde{\mathsf{A}}^{\mathrm{ext}}$
	is (a version of) the extended Airy$_2$ kernel
	\eqref{eq:tilde_A_2_ext}, and 
	\begin{equation}
		\label{eq:cont_f1_f2_gauge_factors}
		\begin{split}
			f_1&:=(h'-h+s^2-s'^2)\mathfrak{d}_{TW}\mathfrak{w}^\circ \log \mathfrak{w}^\circ
			,
			\\
			f_2&:=2(s'-s)\mathfrak{d}_{TW}^2(\mathfrak{w}^\circ)^2(1-\log \mathfrak{w}^\circ)
			.
		\end{split}
	\end{equation}
\end{proposition}
\begin{figure}[htpb]
	\centering
	\includegraphics[width=.3\textwidth]{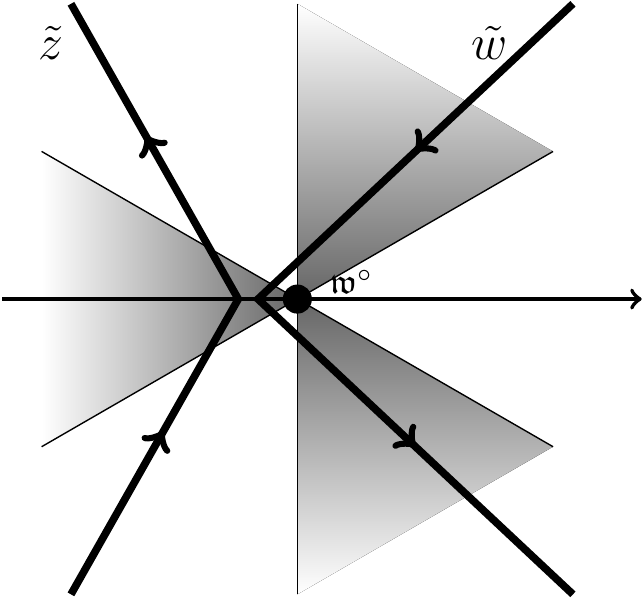}
	\caption{Local behavior of the integration contours
	in a neighborhood of the double critical point~$\mathfrak{w}^\circ$.
	The regions where $\Re (\tilde w^3)<0$ are shaded.}
	\label{fig:TW_contours_pic}
\end{figure}

\begin{remark}
	\textbf{1.}
	Here and below 
	in scalings like \eqref{eq:ttxx_cont_scaling_TW_phase}
	we essentially transpose the pre-limit kernel by
	assigning the primed scaled variables $s',h'$ to the non-primed 
	$t,x$. This transposition is needed so that \eqref{eq:cont_K_TW_behavior_statement}
	holds without switching $(s,h)\leftrightarrow(s',h')$.
	Transposing a correlation kernel does not change the 
	determinantal point process and thus does not affect 
	our asymptotic results.%

	\smallskip
	\noindent\textbf{2.}
	The factor $e^{f_2L^{2/3}+f_1L^{1/3}}$
	is a \emph{gauge transformation} 
	of a determinantal correlation kernel
	which in general looks as
	$K(x,y)\mapsto \frac{f(x)}{f(y)}K(x,y)$
	(with nonvanishing $f$).
	Gauge transformations do not change determinants associated with the kernel.
	\label{rmk:gauge_transform}
\end{remark}

\begin{proof}[Proof of \Cref{prop:cont_K_behavior}]
	One can readily check that $\mathfrak{d}_{TW}=\sqrt[3]{\frac12G^{(3)}(\mathfrak{w}^\circ)}$.
	Deform the integration contours as explained in the beginning of the subsection in the kernel 
	\eqref{eq:cont_K_through_G_fucntion}
	so that they 
	are $\Gamma_{\mathfrak{w}^\circ}$ and $C_{\mathfrak{w}^\circ}$, 
	respectively, and 
	change the variables in a neighborhood of size $L^{-1/6+\varepsilon}$
	(where $\varepsilon>0$ is small and fixed)
	of the double critical point
	$\mathfrak{w}^\circ$
	as
	\begin{equation}
	\label{eq:z_w_1/3_change_of_variables}
		z=\mathfrak{w}^\circ+\frac{\tilde z}{\mathfrak{d}_{TW}L^{1/3}},\qquad 
		w=\mathfrak{w}^\circ+\frac{\tilde w}{\mathfrak{d}_{TW}L^{1/3}},
	\end{equation}
	where $\tilde z,\tilde w$ belong to the contours given in
	\Cref{fig:TW_contours_pic} and are bounded in absolute value by $L^{1/6+\varepsilon}$.
	The exponent in the kernel
	behaves as 
	\begin{equation}
		\label{eq:cont_K_TW_behavior_proof_G_expansion}
		\begin{split}
			&
			L\bigl(G(w;\tfrac{t'}L,\chi,\tfrac{x'}L)-G(z;\tfrac{t}L,\chi,\tfrac{x}L)\bigr)
			=
			L\bigl(G(w;\theta,\chi,\mathfrak{h}(\theta,\chi))
			-
			G(z;\theta,\chi, \mathfrak{h}(\theta,\chi))\bigr)
			\\&\hspace{140pt}
			-(t'-\theta L)w+(t-\theta L)z+(x'-\mathfrak{h} L)\log w-(x-\mathfrak{h} L)\log z
			\\&\hspace{15pt}
			=
			L^{2/3}f_2+L^{1/3}f_1+
			\tfrac{1}{3}\tilde w^3-s \tilde w^2-(h-s^2)\tilde w
			-\tfrac{1}{3}\tilde z^3+s' \tilde z^2+(h'-s'^2)\tilde z
			+o(1),
		\end{split}
	\end{equation}
	where $f_1,f_2$ are given by \eqref{eq:cont_f1_f2_gauge_factors}.
	The remaining factors in the integrand 
	are
	\begin{equation}
		\label{eq:cont_K_TW_behavior_proof1}
		\frac{dwdz}{z(z-w)}=-
		\frac{d\tilde w d\tilde z}{L^{1/3}\mathfrak{w}^\circ \mathfrak{d}_{TW}(\tilde w-\tilde z)}
		\,\bigl( 1+o(1) \bigr)
	\end{equation}
	(the negative sign in the right-hand side is absorbed by reversing one of the contours
	in \Cref{fig:TW_contours_pic}),
	and
	\begin{equation}
		\label{eq:cont_K_TW_behavior_proof2}
		\frac{(\xi(0)-z)}{(\xi(0)-w)}
		\prod\limits_{b\in\RoadblockSet\colon b < \chi} \frac{\xi(b)- p(b)w}{\xi(b)- p(b) z} 
		\cdot \frac{\xi(b)-z}{\xi(b)-w}
		=1+o(1)
		.
	\end{equation}

	Next, with the help of the Stirling
	asymptotics for the Gamma function
	(cf. \cite[1.18.(1)]{Erdelyi1953})
	one readily sees that
	the additional summand 
	in \eqref{eq:cont_K_through_G_fucntion}
	behaves
	as 
	\begin{equation}
		\label{eq:cont_K_TW_behavior_proof_Striling}
		-\mathbf{1}_{t>t',\, x\ge x'}\frac{(t-t')^{x-x'}}{(x-x')!}
		=
		-\mathbf{1}_{s'>s}\,
		e^{f_2L^{2/3}+f_1L^{1/3}}
		\frac{\exp\left\{ -\frac{(h-h'-s^2+s'^2)^2}{4(s'-s)} \right\}}
		{L^{1/3}\mathfrak{d}_{TW}\mathfrak{w}^\circ\sqrt{4\pi (s'-s)}}\,(1+O(L^{-1/3})).
	\end{equation}
	We thus get
	$e^{-f_2L^{2/3}-f_1L^{1/3}}
	\contKernel(t,x;t',x')\approx 
	(L^{1/3}\mathfrak{d}_{TW}\mathfrak{w}^\circ)^{-1}\widetilde{\mathsf{A}}^{\mathrm{ext}}(s,h;s',h')$,
	as desired.

	\medskip

	It remains to show that the behavior 
	of the double contour integral 
	coming from the neighborhood of size $L^{-1/6+\varepsilon}$ of the 
	double critical point
	$\mathfrak{w}^\circ$
	indeed determines the asymptotics
	of the kernel, and show the uniformity of the constant in the error 
	$O(L^{-1/3})$ in \eqref{eq:cont_K_TW_behavior_statement}.

	First, note that 
	both $\mathfrak{w}^\circ(\theta,\chi)$ and $\mathfrak{d}_{TW}(\theta,\chi)$ are 
	uniformly bounded away from $0$ 
	for $(\theta,\chi)$ in a compact subset of the 
	curved part. 
	One can check that 
	in 
	\eqref{eq:cont_K_TW_behavior_proof_Striling}
	the constant by the error $L^{-1/3}$
	contains powers of $\mathfrak{w}^\circ$, $\mathfrak{d}_{TW}$,
	and $s'-s$ in the denominator,
	and thus the error is uniform in $\theta,\chi,h,h'$
	in compact sets.
	
	Let us now turn to the double contour integral,
	and first consider the case when $z,w$ are inside
	the
	$L^{-1/6+\varepsilon}$-neighborhood of $\mathfrak{w}^\circ$.
	Note that the contours $\tilde z,\tilde w$ are separated from each other.
	The $o(1)$ errors coming from
	\eqref{eq:cont_K_TW_behavior_proof_G_expansion},
	\eqref{eq:cont_K_TW_behavior_proof1}, and 
	\eqref{eq:cont_K_TW_behavior_proof2}
	combined produce 
	in front of the exponent
	a function bounded in absolute value by a polynomial 
	in $\tilde z, \tilde w$
	times
	$\mathrm{const}\cdot L^{-1/3}$.
	The Airy-type double contour integral with such additional polynomial factors
	converges, so we get a uniform 
	error of order $L^{-1/3}$.
	Therefore, the double contour integral
	in \eqref{eq:cont_K_through_G_fucntion}
	with $z,w$ in the $L^{-1/6+\varepsilon}$-neighborhood of $\mathfrak{w}^\circ$
	is equal to $1+O(L^{-1/3})$ times
	the
	double contour integral 
	in \eqref{eq:tilde_A_2_ext}
	with $|u|,|v|<L^{1/6+\varepsilon}$.
	The double contour integral
	over the remaining parts of the contours
	can be bounded by $e^{-cL^{1/2+3\varepsilon}}$ and is thus negligible.
	Thus, we get the 
	desired contribution from the 
	small neighborhood of $\mathfrak{w}^\circ$.

	Next, write for the real part similarly to \eqref{eq:cont_K_TW_behavior_proof_G_expansion}:
	\begin{equation}
		\label{eq:cont_K_TW_behavior_proof_G_expansion_real_part}
		\begin{split}
			&L\Re\bigl(G(w;\tfrac{t'}L,\chi,\tfrac{x'}L)-G(z;\tfrac{t}L,\chi,\tfrac{x}L)\bigr)
			=
			L\Re \bigl(G(w;\theta,\chi,\mathfrak{h}(\theta,\chi))
			-
			G(z;\theta,\chi, \mathfrak{h}(\theta,\chi))\bigr)
			\\&\hspace{90pt}
			-(t'-\theta L)\Re w+(t-\theta L)\Re z
			+
			(x'-\mathfrak{h} L)\log |w|-(x-\mathfrak{h} L)\log \mathfrak{w}^\circ.
		\end{split}
	\end{equation}
	By 
	\Cref{lemma:cont_G_estimates_z_circ_contour_TW,lemma:cont_G_estimates_w_opened_contour_TW,lemma:4th_der_ReG}
	there exists $\delta>0$ such that if $z$ or $w$ or both
	are outside the $\delta$-neighborhood of $\mathfrak{w}^\circ$,
	the above quantity is bounded from above by $-c L$ for some $c>0$.
	Indeed, this bound is valid for the first line in 
	\eqref{eq:cont_K_TW_behavior_proof_G_expansion_real_part}
	while the terms in the second line
	as well as the gauge factor 
	$-f_2L^{2/3}-f_1L^{1/3}$
	are of smaller order.

	It remains to consider the case when 
	both $z,w$
	are inside the $\delta$-neighborhood of 
	$\mathfrak{w}^\circ$ but at least one 
	is outise the $L^{-1/6+\varepsilon}$-neighborhood.
	Let use the notation
	$w=\mathfrak{w}^\circ+r(1+\iu)$, 
	$z=\mathfrak{w}^\circ e^{\iu \varphi}$
	where we can assume 
	(by shrinking or enlarging 
	the $\delta$-neighborhood by a constant factor)
	that $0<r<\delta$, $0<\varphi<\delta$,
	$\max (r,\varphi)>L^{-1/6+\varepsilon}$.
	For the first line in the right-hand side 
	of \eqref{eq:cont_K_TW_behavior_proof_G_expansion_real_part}
	we can write
	by \Cref{lemma:4th_der_ReG}:
	\begin{equation}
		\label{eq:cont_K_TW_behavior_proof_G_expansion_real_part_proof3}
			L\Re \bigl(G(w;\theta,\chi,\mathfrak{h}(\theta,\chi))
			-
			G(z;\theta,\chi, \mathfrak{h}(\theta,\chi))\bigr)
			\le -c L (r^3+\varphi^4).
	\end{equation}
	Adding  the terms
	$-f_2 L^{2/3}-f_1 L^{1/3}$ to 
	the second line we 
	can estimate its absolute value as
	\begin{multline*}
		\Bigl|-f_2 L^{2/3}-f_1 L^{1/3}
		-(t'-\theta L)\Re w+(t-\theta L)\Re z+(x'-\mathfrak{h} L)\log |w|
		-
		(x-\mathfrak{h} L)\log \mathfrak{w}^\circ
		\Bigr|
		\\
		\le
		c_2 L^{2/3}(r^3+\varphi^2)+
		c_1 L^{1/3} r.
	\end{multline*}
	One readily sees that the terms in
	\eqref{eq:cont_K_TW_behavior_proof_G_expansion_real_part_proof3}
	dominate by at least a factor of $L^{2\varepsilon}$, and thus the 
	contribution to the double contour 
	integral from this remaining case is also 
	asymptotically negligible.
	This completes the proof.
\end{proof}

The next two propositions deal with 
the BBP and the Gaussian cases. 
As justifications of estimates 
in these cases 
are very similar to the
proof of \Cref{prop:cont_K_behavior},
we omit these arguments and 
only present the main computations.
For the next two statements recall the notation 
$m_\chi$ \eqref{eq:m_x_multiplicity_definition_intro}.

\begin{proposition}[Kernel asymptotics, BBP transition]
	\label{prop:cont_K_BBP_behavior}
	Let $(\theta,\chi)$ be at a BBP transition.
	Scale the parameters as 
	\eqref{eq:ttxx_cont_scaling_TW_phase},
	where 
	$s,s',h,h'\in \mathbb{R}$
	are arbitrary. 
	Then as $L\to+\infty$ for fixed 
	$s,s'$
	we have
	\begin{equation}
		\label{eq:BBP_kernel_scaling}
		\contKernel(t,x;t',x')
		=
		e^{f_2L^{2/3}+f_1L^{1/3}}
		\,
		\frac{1+O(L^{-1/3})}
		{L^{1/3}\mathfrak{d}_{TW}\mathfrak{w}^\circ}
		\,
		\widetilde{\mathsf{B}}^{\mathrm{ext}}_{m_\chi,(0,\ldots,0 )}(s,h;s',h')
		,
	\end{equation}
	with 
	the gauge factors \eqref{eq:cont_f1_f2_gauge_factors} and
	the extended BBP kernel \eqref{eq:ext_BBP}.
	The constant in $O(L^{-1/3})$ is uniform 
	in the same way as in \Cref{prop:cont_K_behavior}.
\end{proposition}
\begin{proof}
	Recall that at a BBP transition we have $\mathfrak{w}^\circ(\theta,\chi)=\mathcal{W}_\chi$.
	The proof is very similar to 
	the one of \Cref{prop:cont_K_behavior}.
	We deform the $z$ and $w$
	integration contours in \eqref{eq:cont_K_through_G_fucntion}
	so that they are $\Gamma_{\mathfrak{w}^\circ}$
	and $C_{\mathfrak{w}^\circ}$, respectively,
	as explained in the beginning of the subsection.
	In particular, the
	pole at $w=\mathfrak{w}^\circ$ stays to the right of all the contours.
	We then
	make the change of variables 
	\eqref{eq:z_w_1/3_change_of_variables}
	in
	a $L^{-1/6+\varepsilon}$-neighborhood of $\mathfrak{w}^\circ$.
	The scaled variables $\tilde z, \tilde w$ belong to the 
	contours given in \Cref{fig:TW_contours_pic}.
	
	The asymptotic expansions of the exponent
	\eqref{eq:cont_K_TW_behavior_proof_G_expansion}
	and the factors 
	\eqref{eq:cont_K_TW_behavior_proof1}
	are the same at our phase transition.
	The behavior of the additional summand
	\eqref{eq:cont_K_TW_behavior_proof_Striling}
	also stays the same.
	The difference with the Tracy-Widom phase comes from the asymptotics
	of the product 
	\eqref{eq:cont_K_TW_behavior_proof2}
	which must be replaced by 
	\begin{equation*}
		\frac{(\xi(0)-z)}{(\xi(0)-w)}
		\prod\limits_{b\in\RoadblockSet\colon b < \chi} \frac{\xi(b)- p(b)w}{\xi(b)- p(b) z} 
		\cdot \frac{\xi(b)-z}{\xi(b)-w}
		=(1+o(1))
		\prod_{b\in\RoadblockSet\colon \xi(b)=\mathfrak{w}^\circ}\frac{\mathfrak{w}^\circ-z}{\mathfrak{w}^\circ-w}
		=(1+o(1))
		\left( \frac{\tilde z}{\tilde w} \right)^{m_\chi}
		.
	\end{equation*}
	Combining these expansions (and omitting 
	error estimates outside a small neighborhood
	of the critical point 
	which are analogous to 
	\Cref{prop:cont_K_behavior})
	one gets the claim.
\end{proof}

For the next statement recall the quantity $\mathfrak{d}_G(\theta,\chi)>0$
\eqref{eq:cont_d_variance_G_phase}
and denote
\begin{equation*}
	\mathfrak{d}_G:=\mathfrak{d}_G(\theta,\chi),
	\qquad 
	\mathfrak{d}_G':=\mathfrak{d}_G(\theta',\chi),
	\qquad 
	\mathfrak{h}:=\mathfrak{h}(\theta,\chi),
	\qquad 
	\mathfrak{h}':=\mathfrak{h}(\theta',\chi).
\end{equation*}

\begin{proposition}[Kernel asymptotics, Gaussian phase]
	\label{prop:cont_K_Gaussian_behavior}
	Let $(\theta,\chi)$ and $(\theta',\chi)$ be in the Gaussian phase, and 
	scale the parameters as
	\begin{equation}
		\label{eq:ttxx_cont_scaling_G_phase}
		\begin{split}
			t=\theta' L+\mathfrak{d}_G' s' L^{1/2},\qquad 
			x&=\lfloor \mathfrak{h}'L+\mathfrak{d}_G'\mathcal{W}(s'-h') L^{1/2}\rfloor,
			\\
			t'=\theta L+\mathfrak{d}_G s L^{1/2},\qquad 
			x'&=\lfloor \mathfrak{h}L+\mathfrak{d}_G\mathcal{W}(s-h) L^{1/2}\rfloor,
		\end{split}
	\end{equation}
	where $s,s',h,h'\in \mathbb{R}$ are arbitrary.
	Then as $L\to+\infty$
	we have 
	with $\widetilde{\mathsf{G}}$ given by 
	\eqref{eq:G_limiting_kernel}:
	\begin{equation}
		\label{eq:G_kernel_scaling}
		\contKernel(t,x;t',x')
		=
		e^{\tilde f_0 L+\tilde f_1 L^{1/2}}
		\frac{1+O(L^{-1/2})}
		{L^{1/2}\mathfrak{d}_G\mathcal{W}_\chi}
		\,
		\widetilde{\mathsf{G}}^{\mathrm{ext}}_{m,\mathfrak{d}_G'/\mathfrak{d}_G}
		(
			h;h'
		),
	\end{equation}
	where
	\begin{equation}
		\label{eq:cont_G_gauge_factors}
		\tilde f_0
		:=
		\mathcal{W}(\theta-\theta')(\log \mathcal{W}-1)
		,
		\qquad 
		\tilde f_1
		:=
		\mathcal{W}
		\left( 
			\mathfrak{d}_G's'-\mathfrak{d}_Gs
			+
			\left( 
				\mathfrak{d}_G
				(s-h)-\mathfrak{d}_G'(s'-h') 
			\right)\log \mathcal{W}
		\right)
		,
	\end{equation}
	and
	the constant in 
	$O(L^{-1/2})$ is uniform in
	$h,h'$ belonging to compact intervals, 
	but may depend on $s,s'$.
\end{proposition}
\begin{proof}
	Recall that in the Gaussian phase we have 
	$\mathfrak{w}^\circ(\theta,\chi)>\mathcal{W}_\chi$,
	and the critical point of interest is now $\mathcal{W}:=\mathcal{W}_\chi$
	which does not depend on $\theta$.
	This critical point is single and not double as in the
	previous two statements.
	Deform the $z$ and $w$ contours in \eqref{eq:cont_K_through_G_fucntion}
	to be $\Gamma_{\mathcal{W}}$ and $C_{\mathcal{W}}$, 
	respectively. In a neighborhood
	of $\mathcal{W}$ of size $L^{-1/4+\varepsilon}$
	(for small fixed $\varepsilon>0$)
	make a change of variables
	\begin{equation*}
		z=\mathcal{W}+\frac{\tilde z}{\mathfrak{d}_G' L^{1/2}},
		\qquad 
		w=\mathcal{W}+\frac{\tilde w}{\mathfrak{d}_G L^{1/2}},
	\end{equation*}
	where $\tilde z, \tilde w$ belong to the contours 
	in \Cref{fig:G_contours_pic} and are bounded in absolute value by $L^{1/4+\varepsilon}$.
	One can readily check that 
	$\mathfrak{d}_G
	=
	\sqrt{-G''(\mathcal{W}_\chi;\theta,\chi,\mathfrak{h})}$,
	$\mathfrak{d}_G'
	=
	\sqrt{-G''(\mathcal{W}_\chi;\theta',\chi,\mathfrak{h}')}$.
	\begin{figure}[htpb]
		\centering
		\includegraphics[width=.3\textwidth]{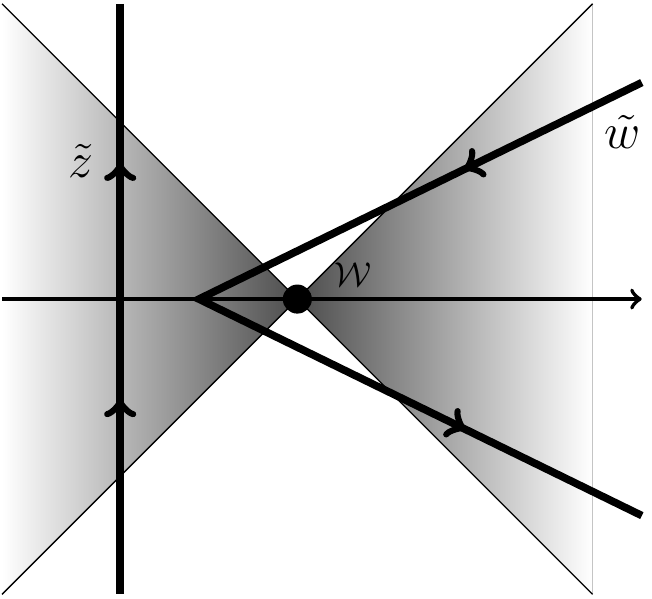}
		\caption{Local behavior of the integration contours
		in a neighborhood of the single critical point~$\mathcal{W}$.
		The contour $\tilde z$ must lie to the left of the 
		contour $\tilde w \mathfrak{d}_G'/\mathfrak{d}_G$.
		Shaded are the regions where $\Re (\tilde w^2)>0$ 
		i.e., $\Re G(w)<\Re G(\mathcal{W})$ locally because $G''(\mathcal{W})<0$.}
		\label{fig:G_contours_pic}
	\end{figure}
		
	Observe that
	$\mathfrak{h}-\mathfrak{h}'=(\theta-\theta')\mathcal{W}$
	and $(\mathfrak{d}_G)^2-(\mathfrak{d}_G')^2=(\theta-\theta')\mathcal{W}^{-1}$.
	The
	exponent in the kernel can be expanded 
	as
	\begin{align*}
		L\bigl(G(w;\tfrac{t'}L,\chi,\tfrac{x'}L)-G(z;\tfrac{t}L,\chi,\tfrac{x}L)\bigr)
		&=
		L\bigl(G(w;\theta,\chi,\mathfrak{h})-G(z;\theta',\chi,\mathfrak{h}')\bigl)
		\\
		&
		\hspace{20pt}
		-(t'-\theta L)w
		+
		(t-\theta' L)z
		+(x'-\mathfrak{h} L)\log w
		-(x-\mathfrak{h}' L)\log z
		\\
		&
		=
		\tilde f_0 L+\tilde f_1 L^{1/2}
		-\tfrac12\tilde w^2+\tfrac12\tilde z^2
		-
		h\tilde w+h'\tilde z+o(1),
	\end{align*}
	where $\tilde f_0, \tilde f_1$ are given by \eqref{eq:cont_G_gauge_factors}.
	The remaining factors in the integrand in \eqref{eq:cont_K_through_G_fucntion}
	are
	\begin{equation*}
		\frac{dwdz}{z(z-w)}=
		-
		\frac{d\tilde w d\tilde z}{L^{1/2}\mathcal{W}\mathfrak{d}_G
			(-\tilde z+\tilde w\mathfrak{d}_G'/\mathfrak{d}_G)
		}
		\,\bigl( 1+o(1) \bigr)
	\end{equation*}
	(the negative sign is absorbed by reversing one of the contours in 
	\Cref{fig:G_contours_pic}),
	and
	\begin{equation*}
		\frac{(\xi(0)-z)}{(\xi(0)-w)}
		\prod\limits_{b\in\RoadblockSet\colon b < \chi} \frac{\xi(b)- p(b)w}{\xi(b)- p(b) z} 
		\cdot \frac{\xi(b)-z}{\xi(b)-w}
		=(1+o(1))
		\left( \frac{\mathfrak{d}_G\tilde z}{\mathfrak{d}_G'\tilde w} \right)^{m_\chi}.
	\end{equation*}

	For the additional summand, the conditions $t>t'$, $x\ge x'$ become simply
	$\theta'>\theta$. Then we have 
	using the Stirling
	asymptotics 
	\cite[1.18.(1)]{Erdelyi1953}:
	\begin{equation*}
		-\mathbf{1}_{t>t',\, x\ge x'}\frac{(t-t')^{x-x'}}{(x-x')!}
		=
		-\mathbf{1}_{\theta'>\theta}
		e^{\tilde f_0 L+\tilde f_1 L^{1/2}}\,
		\frac{\exp\left\{ 
				-\frac{\mathcal{W}(\mathfrak{d}_Gh-\mathfrak{d}_G'h')^2}{2(\theta'-\theta)}
		\right\}}{L^{1/2}\sqrt{2\pi \mathcal{W}(\theta'-\theta)}}
		\left( 1+o(1) \right).
	\end{equation*}
	To match with \eqref{eq:G_limiting_kernel}
	note that $\frac{\theta'-\theta}{\mathcal{W}\mathfrak{d}_G^2}
	=
	\left( {\mathfrak{d}_G'}/{\mathfrak{d}_G} \right)^2-1$.
	
	Via estimates outside the small neighborhood
	of the critical point similar to 
	\Cref{prop:cont_K_behavior}
	one gets the desired claim.
\end{proof}
\begin{remark}
	Since right-and side of 
	\eqref{eq:G_kernel_scaling}
	does not depend on $s$ or $s’$ for the Gaussian
	asymptotics, below in the Gaussian phase
	we will assume $s =s’ =0$.
\end{remark}

\subsection{Asymptotics of Fredholm determinants}
\label{sub:cont_Fredholm_asymptotics}

Having asymptotics of the kernel in each phase, we are now in a 
position to prove 
\Cref{thm:cont_TASEP_fluctuations_in_all_regimes_intro,thm:continuous_intro_limit_shape_theorem}
on the limit shape of the height function
of the continuous space TASEP and its joint 
fluctuations at a fixed location. We begin with the fluctuation statement.

By \Cref{thm:intro_convergence_discrete_to_continuous_TASEP}
(see also \Cref{sub:Fredholm_from_Schur}), for 
fixed $\chi>0$,
any $\ell\in \mathbb{Z}_{\ge1}$,
real $0\le t_1< \ldots< t_\ell $, and $h_1,\ldots,h_\ell\in \mathbb{Z}_{\ge0}$,
the probability 
$\mathop{\mathrm{Prob}}(\mathcal{H}(t_i,\chi)>h_i,\; i=1,\ldots,\ell )$
is expressed as a Fredholm determinant of $\mathbf{1}-\contKernel$
on the union of $\{0,1,\ldots,h_i \}\times\{t_i\}$.
To deal with the asymptotic behavior of this
Fredholm determinant, we need additional estimates of
$|\contKernel(t,x;,t',x')|$
when $x'$ is far to the left of the 
values in the scalings \eqref{eq:ttxx_cont_scaling_TW_phase}
or \eqref{eq:ttxx_cont_scaling_G_phase}.

First we consider the double contour integral 
in \eqref{eq:cont_K_through_G_fucntion}
which we denote by $\mathcal{I}(t,x;t',x')$:
\begin{lemma}[Double contour integral in Tracy-Widom or BBP regime]
	\label{lemma:cont_scaling_TW_phase_Fredholm_remainder}
	Let the space-time point $(\theta,\chi)$ be in the Tracy-Widom phase
	or at a BBP transition.
	Let $t,t'$ scale as in 
	\eqref{eq:ttxx_cont_scaling_TW_phase} with arbitrary fixed $s,s'$.
	Also, take $x$ to be arbitrary, and 
	\begin{equation*}
		x'<\mathfrak{h}L 
		+
		2(\mathfrak{w}^\circ)^2 \mathfrak{d}_{TW}^2 s L^{2/3}
		+
		\mathfrak{w}^\circ \mathfrak{d}_{TW}(s^2-\kappa_0) L^{1/3}
	\end{equation*}
	for some fixed $\kappa_0>s^2>0$ (independent of $L$).
	Then for all large enough $L$ we have
	\begin{multline}
		\label{eq:statement_cont_scaling_TW_phase_Fredholm_remainder}
		e^{-f_2L^{2/3}-f_1L^{1/3}}\bigl|\mathcal{I}(t,x;t',x')\bigr|
		\\\le
		C
		\left( 
			\frac{e^{-c_1L^{\varepsilon_1}}}
			{c_2 \left(
					x'-\mathfrak{h}L-2(\mathfrak{w}^\circ)^2 \mathfrak{d}_{TW}^2 s L^{2/3}
			\right)-1} 
			+
			L^{-1/3}
			e^{c_3 L^{-1/3}(x'-\mathfrak{h}L-2(\mathfrak{w}^\circ)^2 \mathfrak{d}_{TW}^2 s L^{2/3})}
		\right)
		,
	\end{multline}
	where $C,c_i,\varepsilon_1>0$ are constants, 
	and $f_1,f_2$ are the gauge factors 
	\eqref{eq:cont_f1_f2_gauge_factors}
	corresponding to $(t,x;t',x')$.
\end{lemma}
\begin{proof}
	Parametrize 
	$x'=\mathfrak{h}L 
	+
	2(\mathfrak{w}^\circ)^2 \mathfrak{d}_{TW}^2 s L^{2/3}
	+
	\mathfrak{w}^\circ \mathfrak{d}_{TW}(s^2-\kappa) L^{1/3}$,
	where $\kappa>\kappa_0$ and 
	$x$ as in \eqref{eq:ttxx_cont_scaling_TW_phase}
	with $h'$ possibly depending on $L$.
	The gauge factors are
	as in \eqref{eq:cont_f1_f2_gauge_factors}
	but with $h$ replaced by $\kappa$.
	Let the integration contours in $\mathcal{I}$
	pass through the double critical point
	$\mathfrak{w}^\circ(\theta,\chi)$
	and be as in the proofs 
	of \Cref{prop:cont_K_behavior,prop:cont_K_BBP_behavior}.
	To estimate 
	$|\mathcal{I}(t,x;t',x')|$, we bring the absolute value inside
	and consider the real part of the exponent
	which has the form 
	\eqref{eq:cont_K_TW_behavior_proof_G_expansion_real_part}.
	Parametrizing the contours
	$w=\mathfrak{w}^\circ+r(1+\iu)$,
	$z=\mathfrak{w}^\circ e^{\iu \varphi}$ and
	adding the gauge factors $-f_2L^{2/3}-f_1L^{1/3}$
	we see that the resulting expression in the exponent
	does not depend on $h'$
	(which is why $x$ is arbitrary in the hypothesis).
	Moreover, $\kappa$ appears only in the terms multiplied by $L^{1/3}$
	which have the form
	\begin{equation*}
		\frac{1}{2}\mathfrak{d}_{TW}
		\mathfrak{w}^\circ
		L^{1/3}
		(s^2-\kappa)
		\left( 2\log \mathfrak{w}^\circ-
		\log(r^2+(r+\mathfrak{w}^\circ)^2)\right)<-
		cL^{1/3}
		(s^2-\kappa)
		\log (r+1)
	\end{equation*}
	for some $c>0$ depending only on $\theta,\chi,\mathfrak{w}^\circ$
	provided that $\kappa_0>s^2$.
	Arguing as in the proof of 
	\Cref{prop:cont_K_behavior}
	we see that if $z$ or $w$ is outside 
	an $L^{-1/6+\varepsilon}$-neighborhood
	of $\mathfrak{w}^\circ$, 
	the exponent can be bounded from above by 
	$-c L^{2\varepsilon}$
	times an integral of $(r+1)^{-c L^{1/3}(s^2-\kappa)}$
	over $r$ from $0$ to $+\infty$, times a polynomial factor in $L$
	which can be incorporated into the exponent.
	This corresponds to the first term
	in the estimate 
	\eqref{eq:statement_cont_scaling_TW_phase_Fredholm_remainder}.

	When both integration variables are inside the
	$L^{-1/6+\varepsilon}$-neighborhood of $\mathfrak{w}^\circ$,
	make the change of variables 
	\eqref{eq:z_w_1/3_change_of_variables}
	and Taylor expand as in the proof of \Cref{prop:cont_K_behavior}.
	The integral of the absolute value 
	of the integrand converges, and the part depending on $\kappa$
	produces an estimate of the form
	$\le C e^{-c \kappa}$ 
	(after taking into account the gauge factors).
	This corresponds to the second term
	in the right-hand side of \eqref{eq:statement_cont_scaling_TW_phase_Fredholm_remainder}
	where the factor 
	$L^{-1/3}$ in front comes from the change of variables in the double integral.
	This completes the proof.
\end{proof}

A similar estimate can be written down in the Gaussian phase.
Its proof is analogous to \Cref{lemma:cont_scaling_TW_phase_Fredholm_remainder}
therefore we omit it.
\begin{lemma}
	\label{lemma:cont_scaling_G_phase_Fredholm_remainder}
	Let the space-time points 
	$(\theta,\chi)$, $(\theta',\chi)$ be in the Gaussian phase. 
	Let $t,t'$ scale as in \eqref{eq:ttxx_cont_scaling_G_phase} with $s=s'=0$,
	$x$ be arbitrary, and 
	\begin{equation*}
		x'< \mathfrak{h}L-\mathfrak{d}_G \mathcal{W}\kappa_0 L^{1/2}
	\end{equation*}
	for some $\kappa_0>0$ (independent of $L$). 
	Then for all large enough $L$ we have
	\begin{equation*}
		e^{-\tilde f_0 L-\tilde f_1 L^{1/2}}
		\bigl|
		\mathcal{I}(t,x;t',x')
		\bigr|
		\le C
		\left( 
			\frac{e^{-c_1 L^{\varepsilon_1}}}{c_2(x'-\mathfrak{h}L)-1}
			+L^{-1/2}e^{c_3 L^{-1/2}(x'-\mathfrak{h}L)}
		\right),
	\end{equation*}
	where the gauge factors $\tilde f_0, \tilde f_1$ are 
	as in \eqref{eq:cont_G_gauge_factors} with $s=s'=0$, 
	and $C,c_i,\varepsilon_1>0$ are constants.
\end{lemma}

\begin{proof}[Proof of \Cref{thm:cont_TASEP_fluctuations_in_all_regimes_intro}]
	We are now in a position to prove
	\Cref{thm:cont_TASEP_fluctuations_in_all_regimes_intro}
	about fluctuations.
	First, 
	due to the connection to the 
	determinantal point process 
	(\Cref{thm:intro_convergence_discrete_to_continuous_TASEP})
	we can 
	write the probabilities in the left-hand side of 
	\eqref{eq:Airy_multipoint_convergence_cont_TASEP}
	(in Tracy-Widom or BBP regime)
	and 
	\eqref{eq:Gaussian_multipoint_convergence_cont_TASEP}
	(in Gaussian regime)
	as Fredholm determinants of $\mathbf{1}-\mathcal{K}$
	on the space 
	$\mathfrak{X}:=\sqcup_{i=1}^{\ell}\{t_i\}\times \{0,1,\ldots,x_i \}$,
	where $t_i,x_i$
	scale corresponding to the right-hand sides of
	\eqref{eq:Airy_multipoint_convergence_cont_TASEP} or
	\eqref{eq:Gaussian_multipoint_convergence_cont_TASEP},
	and $\mathcal{K}$ is given in \eqref{eq:cont_K_through_G_fucntion}.
	In more detail, the Fredholm determinant has the form
	(cf. \Cref{sub:Fredholm_from_Schur})
	\begin{equation}\label{eq:cont_TASEP_fluct_proof1}
		1+\sum_{n=1}^{\infty}\frac{(-1)^n}{n!}
		\sum_{y^1,\ldots,y^n }
		\det
		\left[ \mathcal{K}(y^p;y^q) \right]_{p,q=1}^{n},
	\end{equation}
	where each $y^p=(t^p,x^p)$ runs over the space 
	$\sqcup_{i=1}^{\ell}\{t_i\}\times \{0,1,\ldots,x_i \}$.
	
	We separate the summation 
	over $y^1,\ldots,y^n $ in \eqref{eq:cont_TASEP_fluct_proof1}
	into two parts,
	when 
	all $y^p$ are close to the right edge of 
	$\mathfrak{X}$ (composed of left neighborhoods of 
	$\{t_i\}\times \{x_i\}$),
	and when at least one $y^p$ is sufficiently 
	far from the right edge of $\mathfrak{X}$, 
	cf.~\Cref{fig:Fredholm_proof_pic}.
	Let us show that 
	the first part of the sum converges to the 
	Fredholm determinant of the corresponding
	limiting kernel,
	and that the second part of the sum
	is negligible.
	
	\begin{figure}[htpb]
		\centering
		\includegraphics{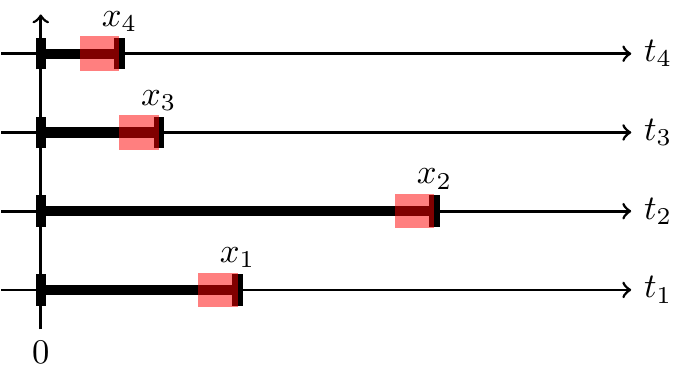}
		\caption{The set $\mathfrak{X}$ over which the 
		Fredholm determinant \eqref{eq:cont_TASEP_fluct_proof1}
		of $\mathbf{1}-\mathcal{K}$
		is taken. 
		Highlighted is the right edge of $\mathfrak{X}$, i.e., 
		the subset contributing to 
		the limiting Fredholm determinant.
		Here $\ell=4$.}
		\label{fig:Fredholm_proof_pic}
	\end{figure}
	
	Consider the Tracy-Widom phase, the other cases are analogous.
	The scaling is 
	\begin{equation*}
		t_i=\theta L+2\mathfrak{w}^\circ \mathfrak{d}_{TW}^2 s_i L^{2/3},
		\qquad 
		x_i=\lfloor 
		\mathfrak{h}L+2(\mathfrak{w}^\circ)^2\mathfrak{d}_{TW}^2
		s_i L^{2/3}+\mathfrak{w}^\circ\mathfrak{d}_{TW}(s_i^2-r_i)L^{1/3}
		\rfloor 
		,
	\end{equation*}
	where $(\theta,\chi)$ is in the Tracy-Widom phase and 
	$s_1,\ldots,s_\ell,r_1,\ldots,r_\ell\in \mathbb{R}$
	are fixed.
	The Fredholm determinant \eqref{eq:cont_TASEP_fluct_proof1}
	expresses the probability
	$\mathop{\mathrm{Prob}}(\mathcal{H}(t_i,\chi)>x_i,\,i=1,\ldots,\ell )$.
	Fix sufficiently large positive $\kappa_1,\ldots,\kappa_\ell $,
	and define the right edge $\mathfrak{X}_{re}$ of $\mathfrak{X}$
	to be disjoint union of segments from 
	$\mathfrak{h}L+2(\mathfrak{w}^\circ)^2\mathfrak{d}_{TW}^2
	s_i L^{2/3}+\mathfrak{w}^\circ\mathfrak{d}_{TW}(s_i^2-\kappa_i)L^{1/3}$
	to $x_i$ on each level $t_i$.
	By \Cref{prop:cont_K_behavior} we have 
	for all $n$:
	\begin{multline*}
		\frac{(-1)^n}{n!}
		\sum_{y^1,\ldots,y^n\in \mathfrak{X}_{re} }
		\det
		[\mathcal{K}(y^p;y^q)]_{p,q=1}^{n}
		\\=
		(1+O(L^{-1/3}))\,
		\frac{(-1)^n}{n!}
		\int\ldots\int
		\det[
			\mathsf{A}^{\mathrm{ext}}
			(Y^p;Y^q)
		]_{p,q=1}^{n}\,
		dh^1 \ldots dh^n,
	\end{multline*}
	where each of the integrals 
	is over $Y^p=(s^p,h^p)\in\sqcup_{i=1}^{\ell}
	\{s_i\}\times[r_i,\kappa_i]$.
	The prefactors $(\mathfrak{w}^\circ\mathfrak{d}_{TW}L^{1/3})^{-1}$
	are absorbed when we pass from sums to integrals 
	due to our scaling.
	We also ignored the gauge factors in \Cref{prop:cont_K_behavior}
	because they do not change the determinants.
	Taking $\kappa_i$ sufficiently large 
	and using the decay of the Airy kernel 
	(e.g., see \cite{tracy_widom1994level_airy}) 
	leads to the desired
	Fredholm determinant of $\mathbf{1}-\mathsf{A}^{\mathrm{ext}}$.
	In the BBP and Gaussian regime we 
	use
	\Cref{prop:cont_K_BBP_behavior,prop:cont_K_Gaussian_behavior},
	respectively, to get similar convergence with 
	the corresponding limiting kernels.
	(In the Gaussian phase the right edge has scale $L^{1/2}$ and 
	not $L^{1/3}$).

	Let us show that the contribution 
	to the Fredholm determinant
	is negligible 
	when at least one $y^p$ is outside $\mathfrak{X}_{re}$.
	We again consider only the Tracy-Widom phase
	as the other ones are analogous.
	Fix $p_0$ such that $y^{p_0}$ is 
	summed over $\mathfrak{X}\setminus \mathfrak{X}_{re}$.
	In \eqref{eq:cont_TASEP_fluct_proof1}
	consider the $n$-th sum, and expand the $n\times n$ determinant
	as a sum over permutations $\sigma\in S(n)$.
	In each of the resulting 
	$n!$ terms single out the factor containing $y^{p_0}$ in the second place:
	\begin{equation}
		\label{eq:cont_TASEP_fluct_proof2}
		\prod_{j=1}^{n}
		\mathcal{K}
		(y^j;y^{\sigma(j)})
		=
		\cdots 
		\mathcal{K}(y^p;y^{p_0})
		\cdots.
	\end{equation}
	We are interested in 
	$\mathcal{K}(y^p;y^{p_0})$
	which is a sum of the additional term
	and the double contour integral 
	$\mathcal{I}$, cf. 
	\eqref{eq:cont_K_through_G_fucntion}.
	For $\mathcal{I}$ we use the estimate 
	of \Cref{lemma:cont_scaling_TW_phase_Fredholm_remainder}
	(in the Gaussian phase we would need
	\Cref{lemma:cont_scaling_G_phase_Fredholm_remainder}).
	Namely, 
	the sum of
	the right-hand side of \eqref{eq:statement_cont_scaling_TW_phase_Fredholm_remainder}
	over $y^{p_0}=(t',x')$
	outside $\mathfrak{X}_{re}$
	can be bounded in absolute value by 
	$C(e^{-c_1 L^{\varepsilon_1}}\log L+e^{-c\kappa})$,
	where $\kappa=\min_{1\le i\le \ell}\kappa_i$,
	and this is small for large $L$ 
	as we take large enough $\kappa_i$.

	The additional term 
	in 
	$\mathcal{K}(y^p;y^{p_0})$
	is nonzero when 
	$t^p>t^{p_0}$ and
	$x^p\ge x^{p_0}$. 
	One can see similarly to 
	\eqref{eq:cont_K_TW_behavior_proof_Striling}
	that when 
	$x^{p_0}-x^p<-\tilde \kappa L^{1/3}$
	for sufficiently large $\tilde \kappa>0$,
	the additional term is negligible.
	Otherwise (when $x^p$ and $x^{p_0}$ are close to each other
	within a constant multiple of $L^{1/3}$)
	it is not negligible,
	and in this case,
	$y^p$
	(which we now call $y^{p_1}$)
	is also outside of $\mathfrak{X}_{re}$.
	We then proceed by finding the factor 
	in \eqref{eq:cont_TASEP_fluct_proof2}
	with $y^{p_1}$ in the second place,
	say, 
	$\mathcal{K}(y^{p_2};y^{p_1})$.
	If $t^{p_2}>t^{p_1}$, 
	this factor can also contribute
	a non-negligible additional summand
	if $x^{p_2}$ is close to $x^{p_1}$, 
	and we can repeat the argument by finding $\mathcal{K}(y^{p_3};y^{p_2})$.
	However, 
	due to the indicators in front of the additional 
	term in $\mathcal{K}$,
	we must take the double contour integral $\mathcal{I}$
	from at least one of the $n$
	factors in \eqref{eq:cont_TASEP_fluct_proof2}.
	Therefore, this procedure 
	of finding non-negligible contributions
	will eventually terminate
	and these additional summands are multiplied by 
	an integral factor. When the additional summands are not 
	small, the corresponding $y^{p_i}$'s are outside
	$\mathfrak{X}_{re}$, and thus the integral factor becomes small.
	We conclude that any non-negligible additional
	summands are multiplied by 
	at least one double contour integral factor which is asymptotically negligible.
	This establishes the desired convergence of the Fredholm
	determinants, and completes the proof 
	of \Cref{thm:cont_TASEP_fluctuations_in_all_regimes_intro}.
\end{proof}

\begin{proof}[Proof of \Cref{thm:continuous_intro_limit_shape_theorem}]
	Let us now prove the limit shape theorem that
	$\lim_{L\to+\infty}L^{-1}\mathcal{H}(\theta,\chi)
	=
	\mathfrak{h}(\theta,\chi)$
	in probability
	for each fixed
	$(\theta,\chi)$.
	If $(\theta,\chi)$ is in the 
	curved part (\Cref{def:curver_part_continuous_regular}),
	then this convergence in probability immediately follows from 
	the (single-point)
	fluctuation results 
	of \Cref{thm:cont_TASEP_fluctuations_in_all_regimes_intro}.
	When $(\theta,\chi)$ is outside the curved part, 
	consider the first particle $x_1$
	of the continuous space TASEP.
	Since this particle performs a 
	simple Poisson random walk 
	(in inhomogeneous space), 
	its location satisfies 
	a Law of Large Numbers. Namely, for fixed $\theta>0$:
	\begin{equation*}
		\lim_{L\to+\infty}
		\mathop{\mathrm{Prob}}
		\left( |x_1(\theta L)-\chi_e(\theta)|>\varepsilon \right)=0
		\quad
		\textnormal{for all $\varepsilon>0$},
	\end{equation*}
	where $\chi_e(\theta)$ is 
	the unique solution to 
	$\theta=\int_0^\chi du/\xi(u)$.
	This implies that 
	$\mathop{\mathrm{Prob}}\left( \mathcal{H}(\theta L,\chi)>\varepsilon L \right)\to0$
	for all $\chi>\chi_e(\theta)$.
	For $\chi=\chi_e(\theta)$ the 
	critical point equation
	\eqref{eq:cont_crit_pts_eq_1}
	has a unique solution $\mathfrak{w}^\circ=0$,
	and thus $\mathfrak{h}=0$. 
	One can check that then $G(v)$
	\eqref{eq:cont_G_L_function_def}
	has a single critical point at $v=0$, 
	and so $\mathcal{H}(\theta L,\chi_e(\theta))$
	has Gaussian type fluctuations of order $L^{1/2}$
	around the limiting value $\mathfrak{h}(\theta,\chi_e(\theta))=0$.
	Thus, the limit shape for the height function
	$L^{-1}\mathcal{H}$ at $\chi_e(\theta)$ is also 
	zero, which completes the proof.
\end{proof}

\subsection{Fluctuations around a traffic jam}
\label{sub:cont_phase_trans_fluctuations}

In this subsection we analyze fluctuations in the 
continuous space TASEP around a down jump 
of the speed function $\xi(\cdot)$
at $\boldsymbol\upchi=1$, see \eqref{eq:traffic_jam_xi_nice_example}.
For this particular choice of $\xi(\cdot)$ the correlation kernel 
\eqref{eq:cont_K_through_G_fucntion}
has the form
\begin{equation}
	\begin{split}
		&\contKernel(t, x; t', x')
		=
		-
		\mathbf{1}_{t>t'}\mathbf{1}_{x\ge x'}
		\frac{(t-t')^{x-x'}}{(x-x')!}
		\\
		&\hspace{80pt}+
		\frac{1}{(2\pi\iu)^2} \oint \oint \frac{dw dz}{z(z-w)}
		\exp
		\left\{ 
			L\bigl(
				G(w;\tfrac{t'}L,\chi,\tfrac{x'}L)-G(z;\tfrac{t}L,\chi,\tfrac{x}L)
			\bigr)
		\right\}
		\frac{1-z}{1-w},
	\end{split}
	\label{eq:traffic_jam_cont_K}
\end{equation}
where $G$ for $\chi>1$ (the regime we're interested in) is given by 
\begin{equation*}
	G(v;\theta,\chi,h)
	=
	-\theta v+h\log v
	+
	\frac{1}{1-v}+\frac{\chi-1}{1-2v}.
\end{equation*}
The $z$ contour is a small circle around $0$, and the $w$ contour encircles
$1/2$ and $1$.

The scaled time is assumed to be critical $\theta_{\mathsf{cr}}=12$
(given by the right-hand side of \eqref{eq:critical_traffic_jam_equation_cont}).
Recall that as $\theta$ passes $\theta_{\mathsf{cr}}$ 
the limit shape loses continuity at $\boldsymbol\upchi=1$.
Set
\begin{equation*}
	\chi-1=10\upepsilon>0 
\end{equation*}
(the factor $10$ 
is convenient in the formulas below)
and let $\upepsilon=\upepsilon(L)\to 0$ as $L\to+\infty$.
Let us expand the double critical point 
$\mathfrak{w}^\circ$
of
$G$ and the limit shape $\mathfrak{h}$
in powers of $\upepsilon$.

\begin{lemma}
	\label{lemma:traffic_jam_epsilon_expansions}
	For small $\upepsilon>0$,
	the double critical point 
	and the limiting height function
	behave as 
	\begin{align}
		\label{eq:traffic_w_epsilon_expansion}
		\mathfrak{w}^\circ(12,1+10\upepsilon)
		&=
		\frac{1}{2}-
		\frac{1}{2}\,\upepsilon^{\frac{1}{4}}
		-
		\frac{1}{5}\,\upepsilon^{\frac{1}{2}}
		+
		\frac{7}{100}\,\upepsilon^{\frac{3}{4}}
		+
		\frac{191}{2000}\,\upepsilon
		+O(\upepsilon^{\frac{5}{4}});
		\\
		\label{eq:traffic_height_epsilon_expansion}
		\mathfrak{h}(12,1+10\upepsilon)
		&=4-10 \upepsilon^{\frac{1}{2}}
		+6 \upepsilon^{\frac{3}{4}}+
		O(\upepsilon^{\frac{5}{4}}).
	\end{align}
\end{lemma}
\begin{proof}
	The double critical point $\mathfrak{w}^\circ$ satisfies
	equation \eqref{eq:cont_crit_pts_eq_1}
	which for our particular $\xi(\cdot)$
	and $\theta=12$
	becomes (after removing the denominator $(v-1)^3 (2v-1)^3$)
	\begin{equation}\label{eq:jam_6_order_eq}
		-11+20\upepsilon+(103-20\upepsilon)v-30(13+2\upepsilon)v^2+
		20(38+5\upepsilon)v^3
		-40(20+\upepsilon)v^4
		+432v^5-96v^6=0.
	\end{equation}
	When $\upepsilon=0$, \eqref{eq:jam_6_order_eq}
	has root $v=\frac{1}{2}$ of multiplicity $4$ which 
	after taking the denominator into account corresponds
	to a single root.\begin{figure}[htpb]
		\centering
		\includegraphics[width=.4\textwidth]{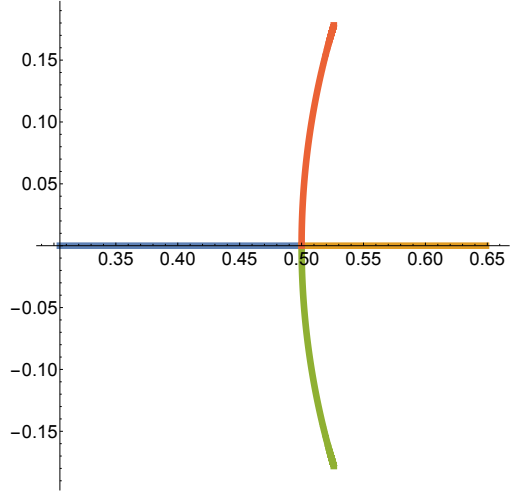}
		\caption{Behavior of four roots of \eqref{eq:jam_6_order_eq}
			in the complex plane which become
			$v=\frac{1}{2}$ for $\upepsilon=0$.}
		\label{fig:roots_epsilon}
	\end{figure}
	For small $\upepsilon>0$ there are four roots close to $\frac{1}{2}$
	two of which are complex conjugate and two of which are real, 
	see
	\Cref{fig:roots_epsilon}.
	We are interested in the unique root $\mathfrak{w}^\circ\in(0,\frac{1}{2})$.
	Using Implicit Function Theorem to find derivatives of $\mathfrak{w}^\circ$ in $\upepsilon$,
	we get the desired expansion~\eqref{eq:traffic_w_epsilon_expansion}.
	The expansion of $\mathfrak{h}$ is obtained using 
	\eqref{eq:cont_crit_pts_eq_2} 
	which now takes the form
	\begin{equation*}
		\mathfrak{h}(12,1+10\upepsilon)
		=
		12\mathfrak{w}^\circ
		-
		\frac{\mathfrak{w}^\circ}{(1-\mathfrak{w}^\circ)^2}
		-
		\upepsilon\,
		\frac{5\mathfrak{w}^\circ}{\left( \frac{1}{2}-\mathfrak{w}^\circ \right)^2}
	\end{equation*}
	together with the expansion of $\mathfrak{w}^\circ$.
	This completes the proof.
\end{proof}

Expansion \eqref{eq:traffic_height_epsilon_expansion}
implies that 
$\frac{\partial}{\partial \upepsilon}
\mathfrak{h}(12,1+ 10\upepsilon)\big\vert_{\upepsilon=0}
=
-\infty$ but $\mathfrak{h}$ is continuous at $\boldsymbol\upchi=1$
(and $\mathfrak{h}(12,1)=4$).
This behavior corresponds to the middle picture in
\Cref{fig:traffic_jam_transition}.

In the rest of the subsection we prove \Cref{thm:traffic_jam_deformed_GUE}.
The analysis of fluctuations of
the random height function 
is similar to 
\Cref{sub:cont_estimates_on_contours,sub:cont_deformation_of_contours,sub:cont_Fredholm_asymptotics}.
The main difference is in the asymptotic expansion 
in the exponent under the integral 
in the kernel $\mathcal{K}$
\eqref{eq:traffic_jam_cont_K} which leads to different limiting
kernels.
The large $L$ behavior
of this exponent
depends on the relative speeds at which $\upepsilon\to0$
and $L\to\infty$. 
To shorten notation
let, by agreement,
$\mathfrak{w}^\circ$, 
$\mathfrak{h}$, and $\mathfrak{d}_{TW}$ 
depend on the parameters
$\theta=12$ and $\chi=1+10\upepsilon$.
The corresponding $\upepsilon=0$ pre-slowdown
values
$\mathfrak{w}^\circ(12,1)=\frac{1}{2}$,
$\mathfrak{h}(12,1)=4$,
and
$\mathfrak{d}_{TW}(12,1)=4\cdot 5^{1/3}$
will be used explicitly.

We consider three cases based on how $\upepsilon$ compares with $L^{-4/3}$.
Indeed, $(L^{-4/3})^{1/4}=L^{-1/3}$ 
corresponds to the scaling of the integration variables
around the double critical point $\mathfrak{w}^\circ$
which itself is close to $1/2$ within $\upepsilon^{1/4}$, see 
\eqref{eq:traffic_w_epsilon_expansion}.
The interplay of these two effects leads to the 
three cases below.

\subsubsection{Close to the slowdown}

Let $0<\upepsilon\ll L^{-\frac{4}{3}-\gamma}$ for some $\gamma>0$.
Scale the parameters as follows
(note the differences with \eqref{eq:ttxx_cont_scaling_TW_phase})
\begin{equation}
	\begin{split}
		t&=12L+\mathfrak{w}^\circ \mathfrak{d}_{TW}^22^{-1/3}s' L^{2/3},
		\qquad 
		t'=12L+\mathfrak{w}^\circ \mathfrak{d}_{TW}^22^{-1/3}s L^{2/3},
		\\
		x&=\lfloor
		4 L+(\mathfrak{w}^\circ)^2\mathfrak{d}_{TW}^2 2^{-1/3} s' L^{2/3}+
		\mathfrak{w}^\circ \mathfrak{d}_{TW}(s'^2-h')2^{-2/3}L^{1/3}
		\rfloor 
		,
		\\
		x'&=
		\lfloor   
		4 L+(\mathfrak{w}^\circ)^2\mathfrak{d}_{TW}^22^{-1/3} s L^{2/3}+
		\mathfrak{w}^\circ \mathfrak{d}_{TW}(s^2-h)2^{-2/3}L^{1/3}
		\rfloor,
		\\
		z&=
		\frac{1}{2}+
		\frac{\tilde z}{2\cdot 10^{1/3} L^{1/3}}
		,
		\qquad \qquad 
		w=
		\frac{1}{2}+
		\frac{\tilde w}{2\cdot 10^{1/3} L^{1/3}},
	\end{split}
	\label{eq:traffic_jam_ttxx}
\end{equation}
where $\tilde z,\tilde w$ belong to the Airy integration contours
as in
\Cref{fig:TW_contours_pic}.
One can check that
\begin{multline}
	\label{eq:traffic_far_close_from_slowdown}
	L
	\bigl( 
		G(w;\tfrac{t'}L,1+10\upepsilon,\tfrac{x'}L)-
		G(z;\tfrac{t}L,1+10\upepsilon,\tfrac{x}L)
	\bigr)
	\\=
	(\textnormal{gauge terms})+
	\frac{\tilde w^3}{3}
	-s\tilde w^2
	-(h-s^2)\tilde w
	-
	\frac{\tilde z^3}{3}
	+s'\tilde z^2+(h'-s'^2)\tilde z
	+
	o(1),
\end{multline}
where ``$(\textnormal{gauge terms})$''
stand for terms 
which do not depend on $\tilde z, \tilde w$ and
can be removed by a suitable gauge transformation
of the kernel
(cf. \Cref{rmk:gauge_transform}).
These terms 
do not affect the asymptotics of probabilities in 
question, and we do not write them down explicitly.
One can also check that the gauge terms coming from the 
non-integral summand in 
\eqref{eq:traffic_jam_cont_K}
are the same as the ones arising from the 
integral.
Moreover, the prefactor $10^{-1/3}L^{-1/3}$ in front of the 
non-integral summand is the same as 
$\frac{1}{2}\cdot2\cdot 10^{-1/3}L^{-1/3}$ coming from the 
change of variables in the double contour integral,
and also coincides with 
$\mathfrak{w}^\circ\mathfrak{d}_{TW}2^{-2/3}L^{-1/3}$
corresponding to rescaling the space variable 
$x$ to $h$.
Repeating the rest of the
argument from
\Cref{sub:cont_estimates_on_contours,sub:cont_deformation_of_contours,sub:cont_Fredholm_asymptotics}
we see that when 
$\chi=1+10\upepsilon$ is close to the slowdown of $\xi(\cdot)$ at $1$,
the fluctuations of the height function around the pre-slowdown value 
$\mathfrak{h}(12,1)=4$
are given by the Airy kernel. 

\subsubsection{Far from the slowdown}
Let $\upepsilon\gg L^{-\frac{4}{3}+\gamma}$ for some $\gamma\in(0,\frac{4}{3})$.
Consider the scaling 
\begin{equation*}
	\begin{split}
		t&=
		12L+2\mathfrak{w}^\circ\mathfrak{d}_{TW}^2s' L^{2/3},
		\qquad 
		x=\lfloor
		\mathfrak{h}L+2(\mathfrak{w}^\circ)^2\mathfrak{d}_{TW}^{2}s' L^{2/3}
		+\mathfrak{w}^\circ\mathfrak{d}_{TW}(s'^2-h')L^{1/3}
		\rfloor 
		;
		\\
		t'&
		=
		12L+2\mathfrak{w}^\circ\mathfrak{d}_{TW}^2s L^{2/3},
		\qquad 
		x'=
		\lfloor   
		\mathfrak{h}L+2(\mathfrak{w}^\circ)^2\mathfrak{d}_{TW}^{2}s L^{2/3}
		+\mathfrak{w}^\circ\mathfrak{d}_{TW}(s^2-h)L^{1/3}
		\rfloor;
		\\
		z&=
		\mathfrak{w}^\circ+\frac{\tilde z}{\mathfrak{d}_{TW}L^{1/3}}
		,
		\qquad \qquad 
		w=
		\mathfrak{w}^\circ+\frac{\tilde w}{\mathfrak{d}_{TW}L^{1/3}}
		.
	\end{split}
\end{equation*}
The 
new integration variables $\tilde z,\tilde w$ belong to the 
contours
as in
\Cref{fig:TW_contours_pic}.
This scaling is 
the same as in the 
general Tracy-Widom fluctuation regime
\eqref{eq:ttxx_cont_scaling_TW_phase} but
the coefficients also depend on $L$. That is,
in contrast with \eqref{eq:traffic_jam_ttxx}
here we include corrections of order larger than
$\upepsilon^{1/4}\gg L^{-1/3}$ directly into 
$t,x,t',x'$ and 
the integration variables.
One can readily check that 
with this scaling 
the same expansion 
\eqref{eq:traffic_far_close_from_slowdown}
holds (with different gauge terms). 
The gauge terms coming from the
additional summand in \eqref{eq:traffic_jam_cont_K}
are also compatible with the ones in the integral.
In this way we again get the Airy kernel 
describing the fluctuations.

\subsubsection{Critical scale at the traffic jam}
This case arises when smaller order terms in 
$\mathfrak{w}^\circ(12,1+10\upepsilon)$, $\mathfrak{h}(12,1+10\upepsilon)$, 
and $\mathfrak{d}_{TW}(12,1+10\upepsilon)$
coincide in scale with 
the natural Airy corrections of orders
$L^{-1/3}$ and $L^{-2/3}$.
Let $\upepsilon=10^{-4/3}\delta L^{-4/3}$, where $\delta>0$ is fixed.
Consider the scaling
\begin{equation}
	\begin{split}
		t&=12L+\mathfrak{w}^\circ \mathfrak{d}_{TW}^22^{-1/3}s' L^{2/3},
		\qquad 
		t'=12L+\mathfrak{w}^\circ \mathfrak{d}_{TW}^22^{-1/3}s L^{2/3},
		\\
		x&=\lfloor
		4 L+(\mathfrak{w}^\circ)^2\mathfrak{d}_{TW}^2 2^{-1/3} s' L^{2/3}+
		\mathfrak{w}^\circ \mathfrak{d}_{TW}(s'^2+2s'\delta^{1/4}-h')2^{-2/3}L^{1/3}
		\rfloor 
		,
		\\
		x'&=
		\lfloor   
		4 L+(\mathfrak{w}^\circ)^2\mathfrak{d}_{TW}^22^{-1/3} s L^{2/3}+
		\mathfrak{w}^\circ \mathfrak{d}_{TW}(s^2+2s\delta^{1/4}-h)2^{-2/3}L^{1/3}
		\rfloor,
		\\
		z&=
		\frac{1}{2}+
		\frac{\tilde z}{2\cdot 10^{1/3} L^{1/3}}
		,
		\qquad \qquad 
		w=
		\frac{1}{2}+
		\frac{\tilde w}{2\cdot 10^{1/3} L^{1/3}},
	\end{split}
	\label{eq:traffic_jam_ttxx_critical_regime}
\end{equation}
where $\tilde z,\tilde w$ belong to the Airy contours
(\Cref{fig:TW_contours_pic}).
We have the following expansion:
\begin{multline}
	\label{eq:traffic_jam_critical_derivation}
	L
	\bigl( 
		G(w;\tfrac{t'}L,1+10\upepsilon,\tfrac{x'}L)-
		G(z;\tfrac{t}L,1+10\upepsilon,\tfrac{x}L)
	\bigr)
	\\=
	(\textnormal{gauge terms})+
	\frac{\tilde w^3}{3}
	-s\tilde w^2
	-(h-s^2)\tilde w
	-
	\frac{\tilde z^3}{3}
	+s'\tilde z^2+(h'-s'^2)\tilde z
	+\frac{\delta}{\tilde z}-\frac{\delta}{\tilde w}
	+o(1).
\end{multline}
The additional summand in \eqref{eq:traffic_jam_cont_K}
has the expansion:
\begin{equation*}
	-
	\mathbf{1}_{t>t'}\mathbf{1}_{x\ge x'}
	\frac{(t-t')^{x-x'}}{(x-x')!}
	=
	-
	\mathbf{1}_{s'>s}
	\exp\{\textnormal{gauge terms}\}\,
	\frac{\exp\left\{ -
	\frac{(h-h'-s^2+s'^2)^2}{4(s'-s)}\right\}}
	{10^{1/3}L^{1/3}\sqrt{4\pi(s'-s)}}
\end{equation*}
with the same gauge terms as in \eqref{eq:traffic_jam_critical_derivation}.
We see that the kernel is approximated by the 
deformed Airy$_2$ kernel 
defined in \Cref{sub:GUE_deformed_kernels}.
The rest of the argument 
for convergence of 
fluctuations can be copied from
the proofs in the Tracy-Widom phase in 
\Cref{sub:cont_estimates_on_contours,sub:cont_deformation_of_contours,sub:cont_Fredholm_asymptotics}.
This completes the proof of
\Cref{thm:traffic_jam_deformed_GUE}.

\subsubsection{Remark. Relation to deformations of the Airy kernel from \cite{BorodinPeche2009}}
\label{ssub:Borodin_Peche_connection}

The deformed Airy kernel that we obtain arises in the
edge scaling limit of a certain multiparameter 
Wishart-like ensemble of random
matrices in the spirit of 
\cite{BorodinPeche2009}.
In that paper the authors consider an Airy-like time-dependent
correlation kernel with two finite sets of real parameters.
In order to
arrive at the kernel of the form
$\widetilde{\mathsf{A}}^{\mathrm{ext},\delta}$
\eqref{eq:tilde_A_2_ext_deformed}
one needs to consider two
infinite sequences 
of 
perturbation parameters 
$x_i,y_j$,
and perform a double limit
transition.
This construction is essentially described in
Remark 2 in \cite{BorodinPeche2009},
and our kernel corresponds to setting all parameters
except $c^{-}$ to zero.

\section{Homogeneous doubly geometric corner growth}
\label{sec:hom_DGCG_asymp}

In this section we consider the limit shape and fluctuations of the 
homogeneous DGCG model (defined in \Cref{sub:DGCG_intro}). Our results are one-parameter deformations of the 
corresponding results for the 
celebrated 
geometric corner growth (equivalently, geometric 
last-passage percolation) model.

Set $a_i\equiv 1$, $\beta_t\equiv \beta>0$, and $\nu_j\equiv \nu\in[-\beta,1)$,
and let $H_T(N)$ denote the height function in this homogeneous DGCG model.
Let $L\to+\infty$ be a large parameter, the location and time scale
linearly as $N=\lfloor \eta L \rfloor $, $T=\lfloor \tau L\rfloor $,
where $\eta$ and $\tau$ are the scaled location and time, respectively.
Fix $\tau>0$ and define the limiting height function 
$\eta\mapsto\mathsf{h}(\tau,\eta)$
as the following parametric curve:
\begin{equation}
	\label{eq:homogeneous_discrete_limit_shape_parametrization}
	\eta(z)=
	\tau\,\frac{\discbeta (1-z)^2 (1-z\discnu)^2}
	{(1-\discnu)(1-z^2\discnu)(1+z\discbeta)^2},
	\qquad 
	\dischlim(z)=
	\tau\,
	\frac{
		\discbeta z^2 
		\left(
		\discbeta(1-z^2\discnu)
		+
		\discnu(1-2 z)
		+1
		\right)
	}
	{
		(1+\discbeta z)^2 
		(1-z^2\discnu)
	},
\end{equation}
where $0\le z\le 1$.
See \Cref{fig:discrete_limit_shapes_Johansson} for examples.
In more detail, we say that $(\tau,\eta)$ is in the
\emph{curved part} if $\tau \beta>\eta(1-\nu)$.
One can show that for $(\tau,\eta)$ in the curved part 
there exists a unique solution to $\eta=\eta(z)$ in $z$
belonging to $(0,1)$.

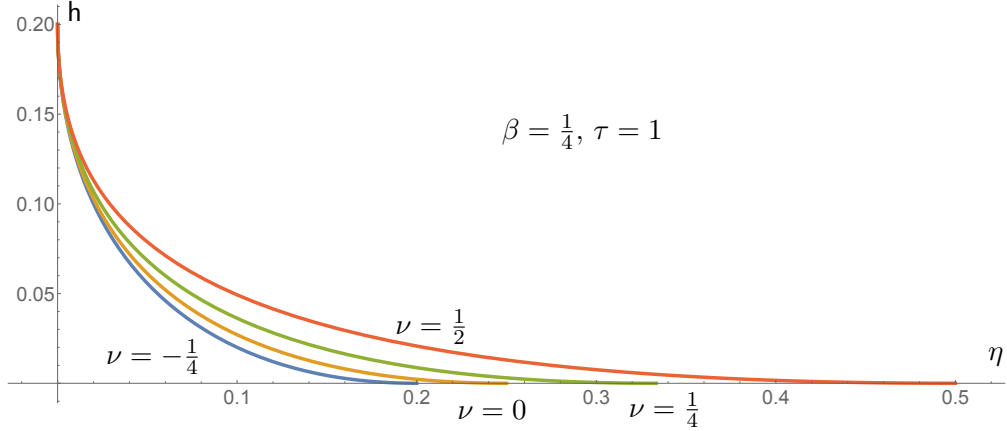
\begin{figure}[htpb]
	\centering
	\begin{tikzpicture}
		[scale=1]
		\node at (0,0) {\includegraphics[width=.8\textwidth]{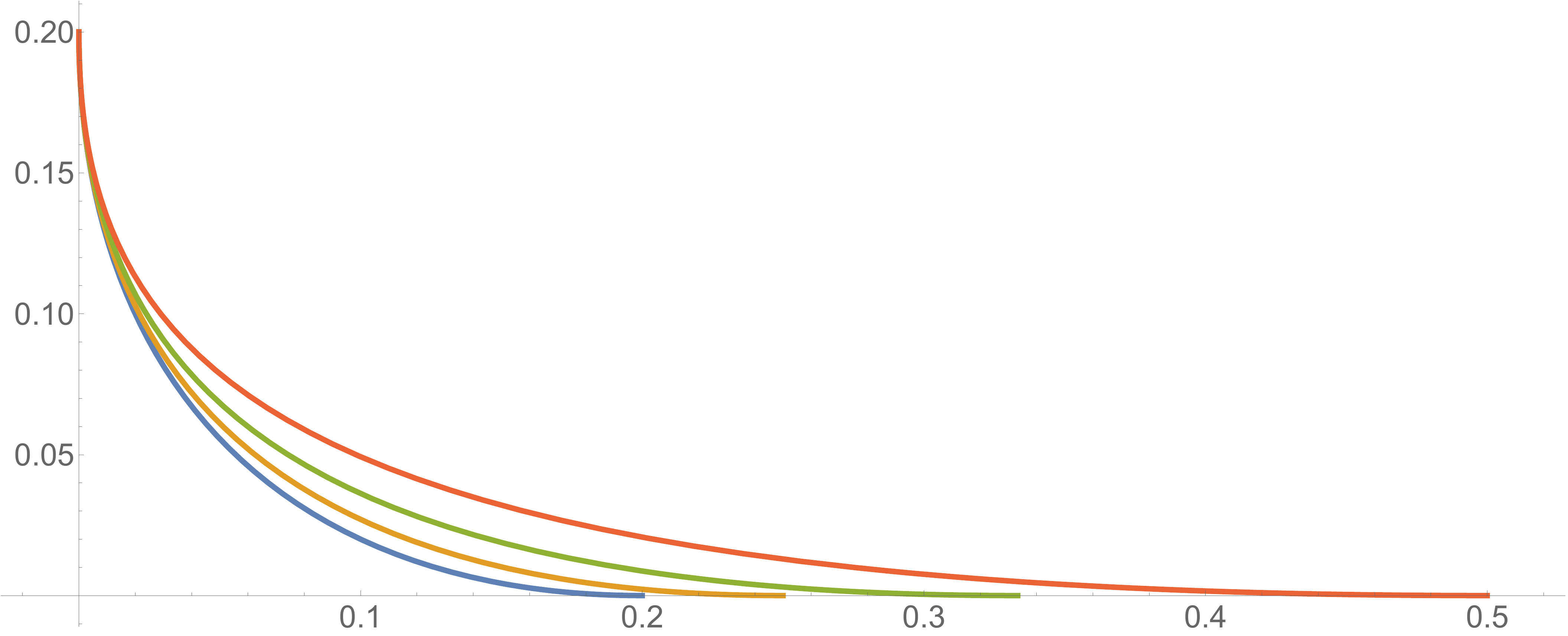}};
		\node at (1,1) {$\discbeta=\frac{1}{4}$, $\tau=1$};
		\node at (-4.7,-2) {$\discnu=-\frac{1}{4}$};
		\node at (-.2,-2.7) {$\discnu=0$};
		\node at (2.1,-2.7) {$\discnu=\frac{1}{4}$};
		\node at (-1,-1.6) {$\discnu=\frac{1}{2}$};
		\node at (6.5,-2) {$\eta$};
		\node at (-5.75,2.6) {$\mathsf{h}$};
	\end{tikzpicture}
	\caption{Limit shapes for varying parameter $\discnu$. The case $\discnu=-\discbeta=-\frac{1}{4}$
		coincides with the limit shape parabola in the geometric corner growth model.}
	\label{fig:discrete_limit_shapes_Johansson}
\end{figure}

Keeping the same parameter $z$ (with $z\in (0,1)$ corresponding to the 
curved part), define 
\begin{equation*}
	\mathsf{d}(z)
	:=
	\Biggl[
		\frac{(1-\discnu)\eta(z)}
		{
			z
			(1+z \beta)
			(1-z)^3
			(1-\discnu z)^3
		}
		\Bigl(
			\beta +
			1+\nu
			-3 z\nu
			(1+z \beta)+\beta z^3\nu
			(1+\nu)+z^3 \discnu^2
		\Bigr)
	\Biggr]^{1/3}.
\end{equation*}
One can show that $\mathsf{d}(z)>0$.
Also define
\begin{equation*}
		\mathscr{A}(z)
		:=
		\frac{2z \mathsf{d}(z)^2(1+\beta z)^2}{\beta}
		,\qquad 
		\mathscr{B}(z):=
		2z \mathsf{d}(z) (1+\beta z)
		,\qquad 
		\mathscr{C}(z):=
		z \mathsf{d}(z).
\end{equation*}

\begin{theorem}
	\label{thm:homogeneous_DGCG}
	As $L\to+\infty$, for all 
	$\tau,\eta>0$
	the scaled DGCG height function $L^{-1}H_{\lfloor \tau L \rfloor }
	(\lfloor \eta L \rfloor )$ 
	converges to $\mathsf{h}(\tau,\eta)$ in probability
	(some examples are given in \Cref{fig:discrete_simulations}).

	Fix $(\tau,\eta)=(\tau,\eta(z))$ in the curved part corresponding 
	to some parameter value $z\in(0,1)$ (recall that $\eta(z), \mathsf{h}(z)$ also
	depend on $\tau$). 
	For any $s_1,\ldots,s_{\ell},r_1,\ldots,r_{\ell}\in \mathbb{R}$ 
	we have
	\begin{multline*}
		\lim_{L\to+\infty}\mathop{\mathrm{Prob}}
		\left( 
			\frac{
				H_{\lfloor \tau L+\mathscr{A}(z)s_i L^{2/3} \rfloor }
				(\lfloor \eta(z) L \rfloor )
			-L \dischlim(z)-\mathscr{B}(z)s_i L^{2/3}}{\mathscr{C}(z)L^{1/3}}
			>
			s_i^2-r_i,\,
			i=1,\ldots,\ell 
		\right)
		\\=
		\det\left( \mathbf{1}-\mathsf{A}^{\mathrm{ext}} \right)_{
		\sqcup_{i=1}^{\ell} \{s_i\}\times(r_i,+\infty) },
	\end{multline*}
	where $\mathsf{A}^{\mathrm{ext}}$ is the extended Airy
	kernel (\Cref{sub:Airy_GUE}).
	In particular, for $\ell=1$ we have
	convergence to the Tracy-Widom GUE distribution:
	\begin{equation*}
		\lim_{L\to+\infty}
		\mathop{\mathrm{Prob}}
		\left( 
			\frac{
				H_{\lfloor \tau L \rfloor }(\lfloor \eta(z) L \rfloor )-
				L \dischlim(z)
			}
			{
				\mathscr{C}(z)L^{1/3}
			}>-r
		\right)=F_{GUE}(r),\qquad  r\in \mathbb{R}.
	\end{equation*}
\end{theorem}

\begin{figure}[htpb]
	\centering
	\begin{tabular}{cc}
		(a) 
		\includegraphics[width=.3\textwidth]{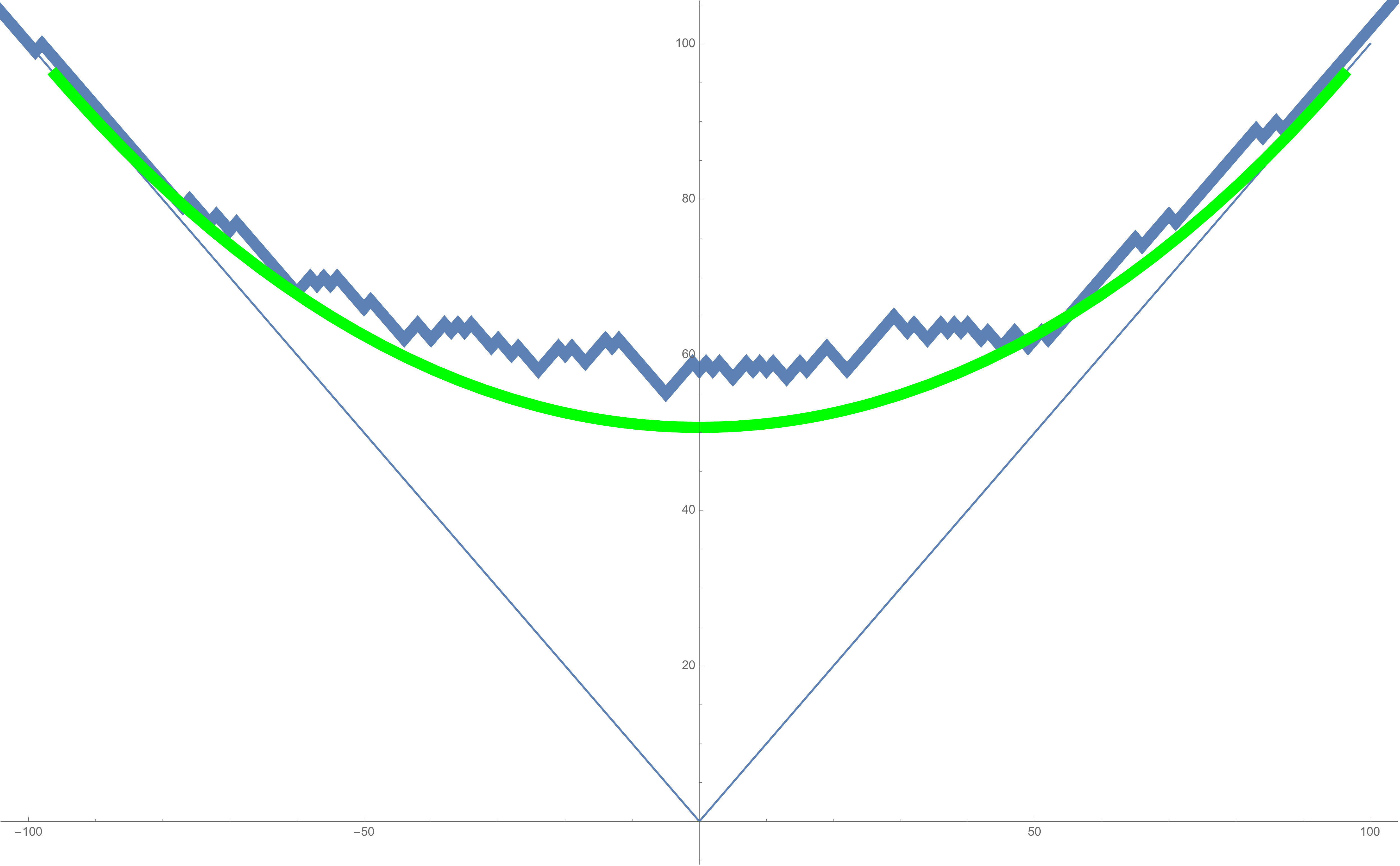}
		&\hspace{40pt}
		(b)
		\includegraphics[width=.3\textwidth]{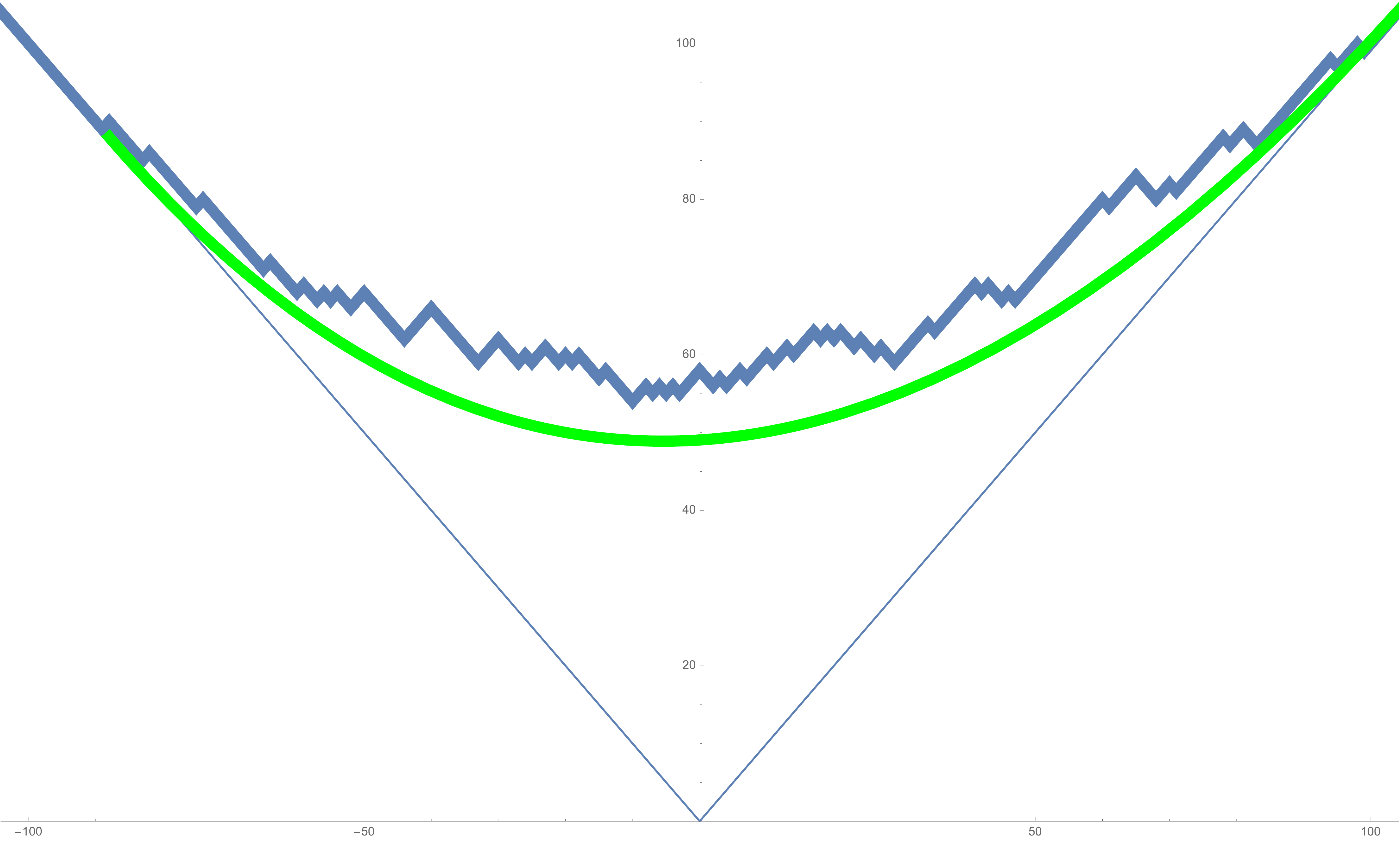}
		\\
		(c)
		\includegraphics[width=.3\textwidth]{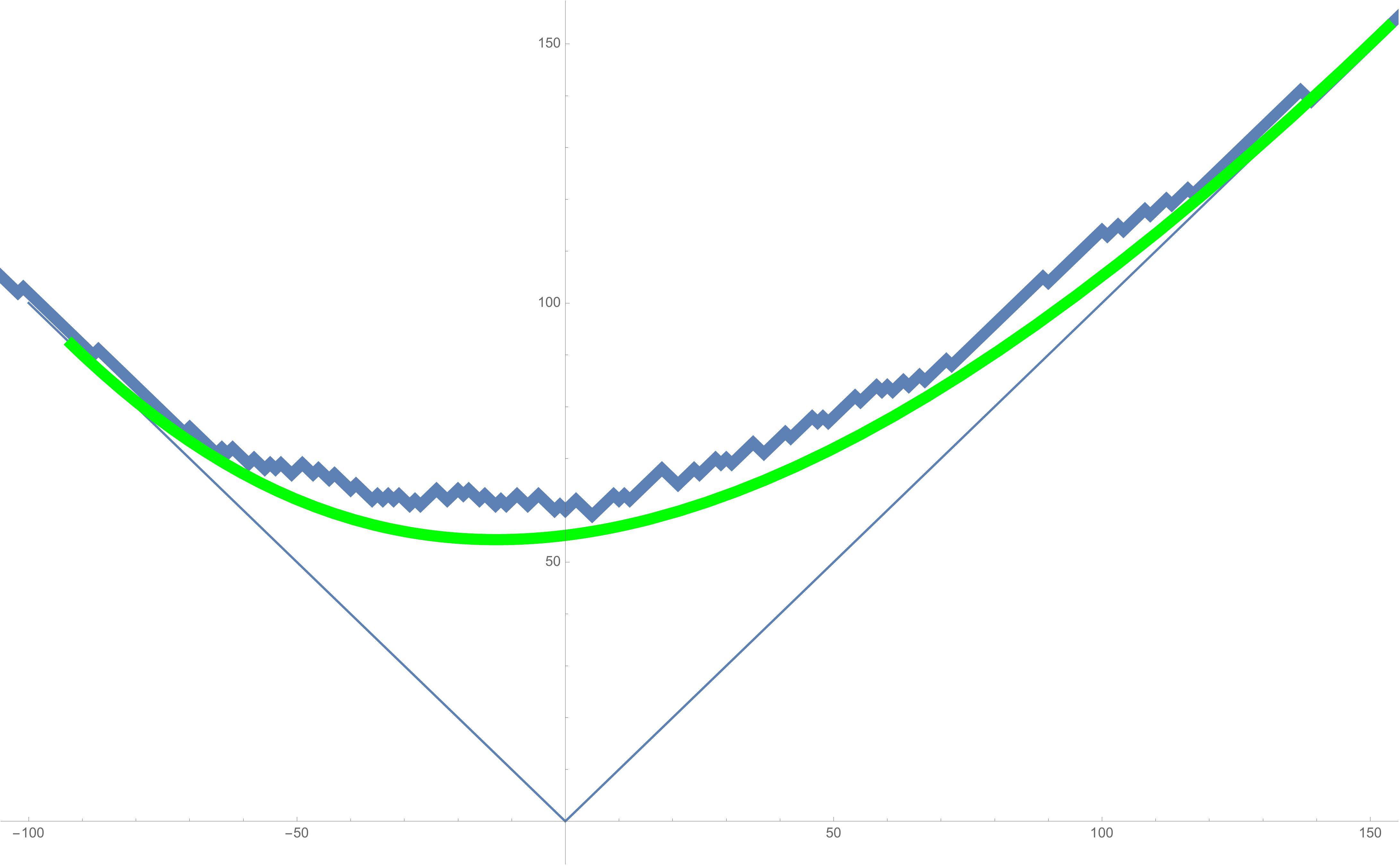}
		&\hspace{40pt}
		(d)
		\includegraphics[width=.3\textwidth]{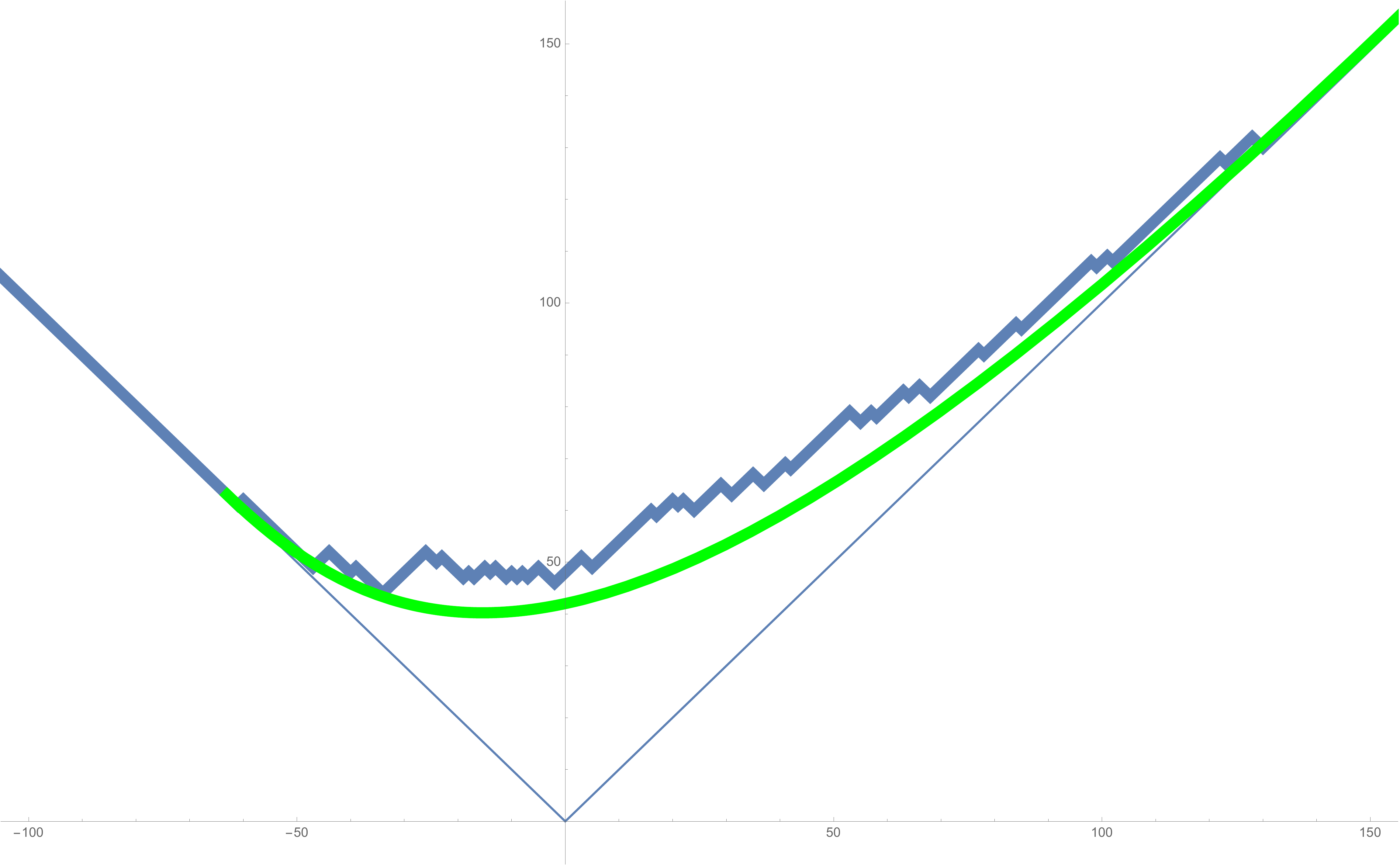}
	\end{tabular}
	\caption{Simulations of the homogeneous DGCG 
	with $\beta=\frac14$ (as in \Cref{fig:discrete_limit_shapes_Johansson}),
	unscaled time $T=500$, and
	(a)~$\discnu=-\frac14$ (parabolic limit shape), 
	(b)~$\discnu=0$, (c)~$\discnu=\frac14$, (d)~$\discnu=\frac12$.
	These figures use the interpretation of 
	DGCG as parallel TASEP (\Cref{sub:gb_TASEP})
	and are thus rotated by $45^\circ$.}
	\label{fig:discrete_simulations}
\end{figure}

\begin{remark}[Reduction to classical corner growth]
	\label{rmk:reduce_classical_corner}
	For $\nu=-\beta$ the homogeneous DGCG model turns into the standard
	corner growth model. 
	Explicit limit 
	shape in the simpler exponential corner growth model goes back to \cite{Rost1981}.
	For the geometric corner growth, the 
	limit shape was obtained
	in
	\cite{jockusch1998random}, 
	\cite{cohn-elki-prop-96},
	and
	\cite{sepp98mprf}
	using various approaches.
	GUE Tracy-Widom fluctuations for the geometric corner growth
	are due to Johansson
	\cite{johansson2000shape}.

	For $\discnu=-\discbeta$ the curve \eqref{eq:homogeneous_discrete_limit_shape_parametrization}
	becomes
	\begin{equation*}
		\eta(z)=\frac{\tau\discbeta(1-z)^2}{(1+\discbeta)(1+z^2\discbeta)}
		,\qquad 
		\dischlim(z)=
		\frac{\tau \discbeta z^2}{1+z^2 \discbeta}
		,
	\end{equation*}
	which after excluding $z$ reduces to
	\begin{equation*}
		\tau=\frac{\eta+\dischlim+2\sqrt{\mathsf{q} \dischlim \eta}}{1-\mathsf{q}},
	\end{equation*}
	under the identification of the parameters
	$\discbeta=\mathsf{q}^{-1}-1$, where $\mathsf{q}\in(0,1)$ is the 
	parameter of the geometric waiting time in the notation of 
	\cite{johansson2000shape}.
	Setting $\eta=1$ and $\dischlim=\gamma$ turns the right-hand side
	into the limiting value of 
	the last-passage time $N^{-1}G^*(\lfloor \gamma N \rfloor ,N)$ 
	from \cite{johansson2000shape}.
	The latter corresponds to the parabolic 
	limit shape in the geometric corner growth / last-passage percolation.
	See \Cref{fig:discrete_limit_shapes_Johansson,fig:discrete_simulations} 
	for an illustration of how the DGCG limit shapes
	form a one-parameter extension of this
	parabola.
\end{remark}

In the rest of the 
section we outline a proof of \Cref{thm:homogeneous_DGCG},
mainly focusing on the contour estimates required for the 
steepest descent analysis.
In view of \Cref{rmk:reduce_classical_corner},
we will not consider the particular case $\nu=-\beta$ extensively
studied previously, and 
will assume that $\nu\in(-\beta,1)$.

\medskip

First, note that for $\nu<0$
the connection of DGCG to Schur measures
decribed in \Cref{sub:new_Schur_connection_discrete_system_determinantal_structure}
breaks 
since Schur processes are not
well-defined for negative parameters.
However, both the homogeneous DGCG model and the 
limit shape curve \eqref{eq:homogeneous_discrete_limit_shape_parametrization}
depend on $\nu\in[-\beta,1)$ in a continuous way.
Moreover, 
the kernel 
$\mathsf{K}_N$ \eqref{eq:new_kernel_full}
and its Fredholm determinants like 
\eqref{eq:one-point-Fredholm-discrete}
clearly make sense for negative $\discnu$. 
The
probability distribution of the height function $H_T(N)$
of the homogeneous
DGCG depends on $\discnu$ in a polynomial
(hence analytic) way. 
Therefore, we can analytically continue 
formulas expressing the 
distribution of $H_T(N)$
as Fredholm determinants of $\discKernel_N$
into the range $\discnu\in(-\beta,1)$.
This allows us to study the asymptotic behavior
of the homogeneous DGCG for $\discnu\in(-\beta,1)$ 
by analyzing the same kernel $\discKernel_N$.

Let us write down the specialization of $\discKernel_N$
to the homogeneous case:
\begin{equation}
	\label{eq:new_kernel_full_homogeneous}
	\begin{split}
		&\discKernel_N
		(T,x;T',x')=
		-
		\frac{\mathbf{1}_{T>T'}\mathbf{1}_{x\ge x'}}{2\pi\iu}
		\oint \frac{(1+\beta z)^{T-T'}}{z^{x-x'+1}}\,dz
		\\&\hspace{60pt}+
		\frac{1}{(2\pi \iu)^2}\oint\oint
		\frac{dz\,dw}{z-w}\frac{w^{x'+N}}{z^{x+N+1}}
		\left(\frac{1-w\nu}{1-z\nu}\right)^{N-1}
		\frac{(1+\beta_t z)^T}
		{(1+\beta_t w)^{T'}}
		\left(
		\frac{1-z}
		{1-w}\right)^{N},
	\end{split}
\end{equation}
The $z$ contour is a small positive circle around $0$ which does not include $1/\nu$, 
and the $w$ contour is a small positive circle around $1$ which is to the right of $0$, $-1/\beta$, and
the $z$ contour.

The asymptotic analysis of $\discKernel_N$
follows essentially the same 
steps as performed for the continuous space TASEP
in \Cref{sec:asymptotics}.
That is, we write $\discKernel_N$ as 
in \eqref{eq:cont_K_through_G_fucntion}
with the function in the exponent under the double integral 
looking as
\begin{equation*}
	S_L(z)=S_L(z;T,N,h):=\frac{h}{L}\log z
	+\frac{N-1}{L}\log(1-\discnu z)-
	\frac{T}{L}\log(1+\beta z)-
	\frac{N}{L}\log(1-z),
\end{equation*}
where $h=x+N$.
The scaling of the parameters 
$T=\lfloor \tau L \rfloor$, $N=\lfloor \eta L \rfloor $
means that we can modify the function
$S_L$ to be
\begin{equation}
	\label{eq:neg_nu_S_L}
	S_L(z)=\frac{h}{L}\log z+\eta \log (1-\discnu z)-\tau \log(1+\beta z)-\eta\log(1-z).
\end{equation}
Indeed, the difference in the exponent is either small
or can be removed by a suitable gauge transformation.

We find the double critical point 
$z=\disczcr_L$
of $S_L(z)$, and deform the integration contours so that the 
behavior of the double contour integral 
is dominated by a small neighborhood of $\disczcr_L$.
To complete the argument 
we need to show
the existence of steep ascent/descent
integration contours.
That is, we find new contours $\gamma_{\pm}$
such that 
$\Re S_L(z)$ attains its minimum on $\gamma_+$
at $z=\disczcr_L$, 
and $\Re S_L(w)$ attains its maximum on $\gamma_-$
at $w=\disczcr_L$.

In the sequel we assume that 
$(\tau,\eta)$ is in the curved part:
$\tau\beta>\eta(1-\discnu)$.
Moreover, we will always assume that 
$h< \tau L$ as the corresponding pre-limit inequality
$H_T(N)\le T$ holds almost surely by the very definition of the 
DGCG model.

One readily sees that $S_L(z)$ has three critical points,
up to multiplicity, since the numerator in 
$S'_L(z)$ is a cubic polynomial.
In the curved part 
there exists $\dischlim_L$ such that
$S_L(z;T,N,\dischlim_L)$ has a double critical point
$\disczcr_L\in(0,1)$. 
Taking this double critical point as a parameter 
of the limit shape and expressing
$\mathsf{h}$ and $\eta$ (for fixed $\tau$) 
through this critical point, 
we arrive at the formulas for the limit shape 
\eqref{eq:homogeneous_discrete_limit_shape_parametrization}.

The next two 
\Cref{lemma:negNu_reS_limits,lemma:negNu_S_third_critical_point}
determine the location of the third critical point of $S_L$
(which must also be real).

\begin{lemma}
	\label{lemma:negNu_reS_limits}
	The function $S_L$ \eqref{eq:neg_nu_S_L}
	has the following limits:
	\begin{equation}
		\begin{split}
			\lim_{z \rightarrow \infty} \Re S_L(z) = \lim_{z \rightarrow \discnu^{-1}}
			& 
			\Re S_L(z) 
			=
			\lim_{z \rightarrow 0} \Re  S_L(z) = - \infty,
			\\
			\lim_{z \rightarrow - \beta^{-1}} \Re S_L(z)
			&= 
			\lim_{z \rightarrow 1} \Re S_L(z) = \infty.
		\end{split}
	\end{equation}
\end{lemma}
\begin{proof}
	This follows from the limits $\log |v|  \rightarrow - \infty$ as
	$v \rightarrow 0$ and $\log |v|  \rightarrow  \infty$ as $v \rightarrow
	\infty$. The signs of the infinities are determined by the signs of the parameters.
	At $v\to\infty$ we use $h/L< \tau$.
\end{proof}

\begin{lemma}
	\label{lemma:negNu_S_third_critical_point}
	The function $S_L$ \eqref{eq:neg_nu_S_L} has a real critical point
	$v_0\in(-\infty,\discnu^{-1})$.
\end{lemma}
\begin{proof}
	It suffices to show that
	$S_L^{\prime}(v_1) < 0$ and $S_L^{\prime}(v_2)>0$ for a pair of real points
	$v_1, v_2 \in (- \infty , 1/\discnu)$.
	We have that $S_L^{\prime}(z) \rightarrow - \infty$ as $v \rightarrow
	(1/\discnu)^{-}$. This establishes the existence of $v_1 \in (- \infty , 1/\discnu)$
	such that $S_L^{\prime}(v_1) < 0$. Also, we have that $v S_L^{\prime}(v)
	\rightarrow h/L -\tau <0$ as $v \rightarrow - \infty$. This
	establishes the existence of $v_2 \in (- \infty , 1/\discnu)$, near negative
	infinity on the real axis, such that $S_L^{\prime}(v_2) > 0$. Therefore,
	there is $v_0 \in (v_2 , v_1)$ such that $S_L^{\prime}(v_0) =0$.
\end{proof}

As the new contours $\gamma_{\pm}$ we take the steepest
ascent/descent paths.
Recall that for a meromorphic function $f: \mathbb{C} \rightarrow \mathbb{C}$, an oriented
path $\gamma : [0,1] \rightarrow \mathbb{C}$ is a steepest path with
base point $z_0 \in \mathbb{C}$ if $\gamma$ is smooth, travels along the
gradient of $\Re f$ (i.e. $\gamma^{\prime} (t) \cdot \nabla (\Re f)|_{z =
\gamma (t)} = \lambda \gamma^{\prime} (t) $), and $\gamma(0)= z_0$. If $\Re f$
is increasing or decreasing along $\gamma$, we say that $\gamma$ is a steepest
ascent or descent path, respectively.

\begin{proposition}
	\label{prop:negNu_steep_contours}
	Consider the function 
	$S_L(z;\lfloor \tau L \rfloor ,\lfloor \eta L \rfloor , \dischlim_L)$
	which has a double critical point at $z=\disczcr_L$.
	There is a pair of steepest ascent paths with
	base point $\disczcr_L$, denoted as $\gamma_+^{(1)}$ and
	$\gamma_+^{(2)}$ (symmetric with respect to $\mathbb{R}$), 
	so that 
	$\gamma_+:=\gamma^{(1)}_+\cup\gamma^{(2)}_+$
	is a simple closed curve enclosing the
	origin and traveling through $- \beta^{-1}$.
	There is a pair of
	steepest descent paths with base point $\disczcr_L$, denoted as
	$\gamma_-^{(1)}$ and $\gamma_-^{(2)}$
	(also symmetric with respect to $\mathbb{R}$), so that 
	$\gamma_-:=\gamma^{(1)}_-\cup\gamma^{(2)}_-$
	is a simple
	closed curve (on the Riemann sphere)
	that travels through a real point in $[- \infty, 1/\discnu]$.
	See \Cref{fig:contour} for an illustration.
\end{proposition}

\begin{figure}[htbp]
	\centering
	\includegraphics[width=.45\textwidth]{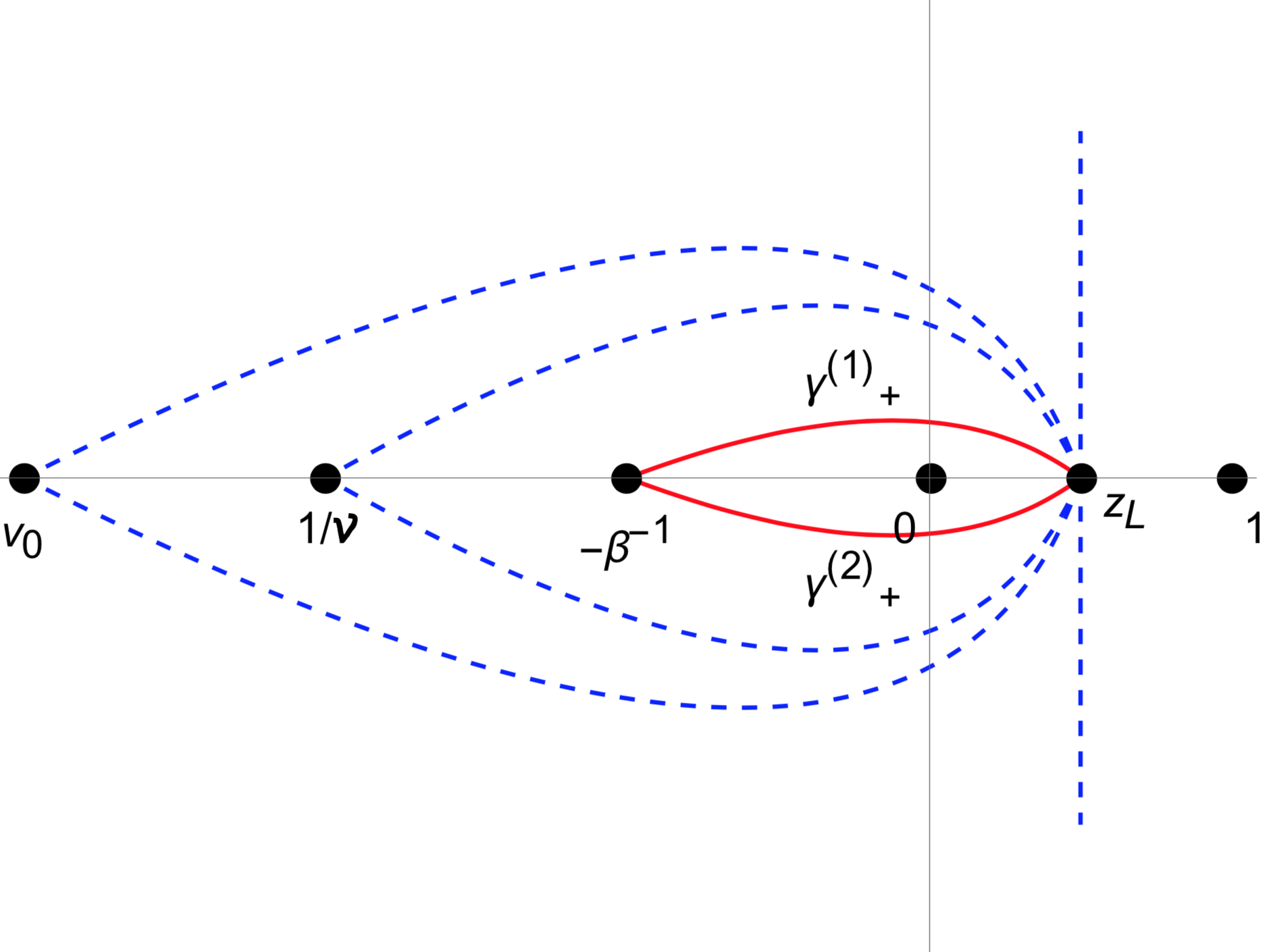}
	\caption{Steepest ascent/descent contours for $S_L(z;T,N,\dischlim_L)$.
	The steepest ascent contours comprising 
	$\gamma_+$ are solid red, and possible options for the
	steepest descent contours comprising $\gamma_-$
	are dashed blue.}
	\label{fig:contour}
\end{figure}

The contours $\gamma_+$ and $\gamma_-$ are assumed to have positive
(counterclockwise) orientation.

\begin{remark}
	\label{rmk:steep_not_steepest}
	In the proofs in \Cref{sec:asymptotics}
	we took concrete integration contours which were not the steepest,
	and this required estimating the derivative of $\Re S_L$
	along the contours. 
	For the relatively simpler function \eqref{eq:neg_nu_S_L}
	we can in fact understand the global configuration
	of the steepest ascent/descent contours,
	and this allows to avoid concrete estimates
	of derivatives of $\Re S_L$.
\end{remark}

\begin{proof}[Proof of \Cref{prop:negNu_steep_contours}]
	Since $S_L(z)$ is analytic at the critical point $\disczcr_L$, we
	know the local shape of all of the steepest paths with base point
	$\disczcr_L$. To establish the global structure of the paths
	we use 
	the following properties:
	\begin{enumerate}[\bf1.]
		\item 
			$\Re S_L(z) = \Re S_L(\bar{z})$;
		\item 
			steepest paths 
			for a meromorphic function
			only intersect at
			critical points or singularities;
		\item 
			the end point of any
			steepest path is a critical point, or a singularity, or infinity.
	\end{enumerate}
	The first property implies that $\gamma_+^{(1)}$ and $\gamma_{+}^{(2)}$ 
	are symmetric with respect to the real line, 
	and the same for $\gamma_{-}^{(1)}$ and $\gamma_{-}^{(2)}$.
	
	Since $\disczcr_L$ is a double critical point and 
	$S_L'''(\disczcr_L)>0$, 
	there are six
	distinct steepest descent paths with base point $\disczcr_L$:
	an
	ascent path
	along the real axis from
	$\disczcr_L$ to $1$, a descent path along the real
	axis from $\disczcr_L$ to $0$,
	and four other paths which we denote by
	$\gamma_+^{(1)}$, $\gamma_+^{(2)}$, $\gamma_-^{(2)}$, and 
	$\gamma_-^{(1)}$ (in counterclockwise order). 
	We know that
	the paths $\gamma_+^{(1)}$ and $\gamma_+^{(2)}$ are (locally) to the left of
	$\gamma_-^{(1)}$ and $\gamma_-^{(2)}$. 

	The end point of
	$\gamma_-^{(1)}$ must be a singularity or a critical point. By the limits of
	\Cref{lemma:negNu_reS_limits} and recalling the simple critical point $v_0 \in (-
	\infty , 1 /\discnu)$ from \Cref{lemma:negNu_S_third_critical_point}, 
	the end point of $\gamma_-^{(1)}$
	must be $0$, $1/\discnu$, $v_0$, or $\infty$. 
	It follows that $\gamma_-=\gamma_-^{(1)}\cup\gamma_{-}^{(2)}$
	must be a simple
	closed curve (on the Riemann sphere) passing through one of the points $0$, $1/\discnu$,
	$v_0$, or $\infty$.
	The union 
	$\gamma_+=\gamma_+^{(1)}\cup \gamma_{+}^{(2)}$
	of
	the steepest ascent paths
	with base point $\disczcr_L$
	is
	a
	simple closed curve
	passing through $-1/\beta$ or $1$.
	
	The
	curves $\gamma_{+}$ and $\gamma_{-}$ cannot intersect 
	outside $\mathbb{R}$
	as this would imply
	existence of 
	additional imaginary critical points or singularities of $S_L(z)$,
	which 
	is not possible.
	Thus, 
	$\gamma_{+}$ cannot pass through $1$, and 
	$\gamma_{-}$ cannot pass through $0$. 
	We are left with the steepest paths
	described by the statement of this proposition, which are depicted in
	\Cref{fig:contour}.
\end{proof}

To finish the proof of \Cref{thm:homogeneous_DGCG},
it remains to show that the 
$z$ and $w$
integration contours
in the kernel $\discKernel_N$
\eqref{eq:new_kernel_full_homogeneous}
can be deformed to 
$\gamma_+$ and $\gamma_-$, respectively.

The old $z$ contour is a small circle around 
$0$, and $-\beta^{-1}$ is not a pole in $z$.
Therefore, we can replace the $z$ contour
by $\gamma_+$ without picking any residues.
We then deform the $w$ contour to $(-\gamma_-)$ by passing over infinity in the
Riemann sphere. 
In this deformation, the only possible residue contribution can
come from infinity
since the integrand is analytic elsewhere along the
deformation. 
Counting the powers of $w$
as $w\to\infty$ in \eqref{eq:new_kernel_full_homogeneous}
(or recalling that $S_L(w)\to-\infty$ as $w\to\infty$)
we conclude that
the integrand does not have a residue at $w=\infty$, and thus the 
deformation can be performed.

The orientation of $\gamma_-$ is negative after the deformation.
This sign is the same extra factor of $(-1)$
arising 
in the proof of 
\Cref{prop:cont_K_behavior}.
Taking this orientation into account we see that the limiting 
Airy fluctuation kernel 
has the correct sign.
We omit the straightforward computation of the 
constants in the Airy kernel limit in 
\Cref{thm:homogeneous_DGCG}. 

\appendix

\section{Equivalent models}
\label{sec:app_B_combinatorics}

Here we discuss a number of equivalent combinatorial 
formulations of our discrete DGCG model.
For simplicity we consider only fully homogeneous models
with $a_i\equiv a$, $\nu_j\equiv \nu$, $\beta_t\equiv \beta$.
In \Cref{sub:contin_FPP_RSK} we also describe an equivalent
formulation of the (homogeneous) continuous space TASEP.

\subsection{Parallel TASEP with geometric-Bernoulli jumps}
\label{sub:gb_TASEP}

Let us interpret the doubly geometric corner growth $H_T(N)$ as a TASEP-like particle system.

\begin{definition}
	\label{def:gB_random_variable}
	The \emph{geometric-Bernoulli} random variable $\mathsf{g}\in \mathbb{Z}_{\ge0}$
	(\emph{gB variable}, for short; notation $\mathsf{g}\sim \mathrm{gB}(a\beta,\nu)$)
	is a random variable with distribution
	\begin{equation*}
		\mathop{\mathrm{Prob}}(\mathsf{g}=j)
		:=
		\frac{\mathbf{1}_{j=0}}{1+a\beta}
		+
		\frac{a\beta\,\mathbf{1}_{j\ge1}}{1+a\beta}
		\left( \frac{\nu+a\beta}{1+a\beta} \right)^{j-1}
		\frac{1-\nu}{1+a\beta}
		,
		\qquad 
		j\in \mathbb{Z}_{\ge0}.
	\end{equation*}
\end{definition}

\begin{definition}
	\label{def:gB_TASEP}
	The \emph{geometric-Bernoulli Totally Asymmetric Simple Exclusion Process}
	(\emph{gB-TASEP}, for short) is a discrete time Markov chain 
	$\{\vec G(T)\}_{T\in \mathbb{Z}_{\ge0}}$
	on 
	the space of particle configurations 
	$\vec G=(G_1>G_2>\ldots)$ in $\mathbb{Z}$, 
	with at most one particle per site allowed,
	and the step initial condition $G_i(0)=-i$, $i=1,2,\ldots $.
	
	The dynamics of gB-TASEP proceeds as follows.
	At each discrete time step, each particle $G_j$ with an empty site to the right
	(almost surely there are finitely many such particles at any finite time)
	samples an independent random variable $\mathsf{g}_j\sim \mathrm{gB}(a\beta,\nu)$, and 
	jumps by $\min(\mathsf{g}_j,G_{j-1}-G_{j}-1)$ steps
	(with $G_0=+\infty$ by agreement).
	See \Cref{fig:gB_TASEP} (in the Introduction) for an illustration.
\end{definition}

\begin{proposition}
	\label{prop:gB_TASEP_corner_growth_correspondence}
	Let $H_T(N)$ be the DGCG height function.
	Then for all $T\in \mathbb{Z}_{\ge0}$ and $N\in \mathbb{Z}_{\ge1}$ we have
	\begin{equation*}
		H_T(N)=\#\{i\in \mathbb{Z}_{\ge1}\colon G_{i}(T)+i+1\ge N\},
	\end{equation*}
	where $\{G_i(T)\}$ is the gB-TASEP with the step initial configuration.
\end{proposition}

\begin{remark}
	\label{rmk:Povolotsky_generalized_TASEP}
	Replacing particles by holes and 
	vice versa in gB-TASEP
	one gets a stochastic particle system 
	of zero range type.
	It is 
	called the
	\emph{generalized TASEP} in
	\cite{povolotsky2015gen_tasep}.
\end{remark}

\subsection{Directed last-passage percolation like growth model}
\label{sub:LPP_formulation}

Let us present another equivalent formulation of DGCG
as a variant of directed last-passage percolation.
For each $N\in \mathbb{Z}_{\ge2}$ and $H\in \mathbb{Z}_{\ge1}$, 
sample two families of independent identically distributed geometric random variables:
\begin{enumerate}[$\bullet$]
	\item $W_{N,H}\in \mathbb{Z}_{\ge1}$ has the geometric distribution with parameter 
		$w:=a \beta/(1+a\beta)$, that is, $\mathop{\mathrm{Prob}}(W_{N,H}=j)=w^{j}(1-w)$, $j\ge 1$.
	\item $U_{N,H}\in \mathbb{Z}_{\ge0}$ has the geometric distribution
		\begin{equation*}
			\mathop{\mathrm{Prob}}(U_{N,H}=j)=
			\frac{1-\nu}{1+a\beta}\left( \frac{\nu+a\beta}{1+a\beta} \right)^{j}, 
			\qquad 
			j\ge0,
		\end{equation*}
		which is the homogeneous version of 
		\eqref{eq:add_boxes_p_inhom_geom_distr_parameter}--\eqref{eq:p_inhom_geom_distr_parameter}.
\end{enumerate}

\begin{figure}[htpb]
	\centering
	\includegraphics[width=.4\textwidth]{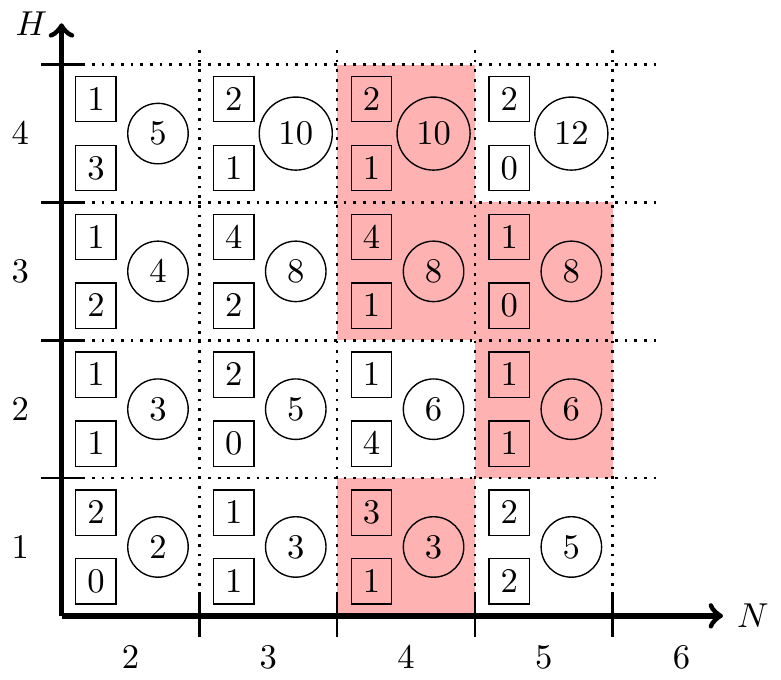}
	\caption{Directed last-passage percolation formulation of DGCG.
	The independent random 
	variables $W_{N,H}$ and $U_{N,H}$ are written in rectangular boxes in each cell,
	and the variables $L_{N,H}$ (times at which each cell is covered by the growing interface) 
	are circled.
	Shaded are the cells which are covered instantaneously during the growth.}
	\label{fig:LPP}
\end{figure}

Define a family of random variables $L_{N,H}\in \mathbb{Z}_{\ge1}$, $N\ge2$, $H\ge1$, 
depending on the $W$'s and the $U$'s via the recurrence 
relation
\begin{equation}
	\label{eq:LPP_definition}
	\begin{split}
		&L_{N,H}:=
		\max(L_{N-1,H},L_{N,H-1})
		+W_{N,H}
		\\
		&\hspace{50pt}-W_{N,H}\mathbf{1}_{L_{N-1,H}>L_{N,H-1}}
		\sum_{j=1}^{N-2}\mathbf{1}_{L_{N-1,H}
		% =L_{N-2,H}
		=\ldots=L_{N-j,H}>L_{N-j-1,H} }
		\mathbf{1}_{U_{N-j,H}\ge j}
		,
	\end{split}
\end{equation}
together with the boundary conditions
\begin{equation}\label{eq:LPP_boundary_cond}
	L_{1,H}=L_{N,0}=0,\qquad H\ge0,\quad N\ge1.
\end{equation}
An example is given in \Cref{fig:LPP}.
\begin{proposition}
	\label{prop:LPP_equivalent_to_growth}
	The time-dependent formulation $\{H_T(N)\}$ (with homogeneous parameters)
	and the last-passage formulation $\{L_{N,H}\}$ are equivalent in the sense that
	\begin{equation*}
		L_{N,H}=\min\left\{ T\colon H_{T}(N)=H \right\}
	\end{equation*}
	for all
	$H\ge1$, $N\ge2$.
\end{proposition}
\begin{proof}
	In $\{H_T(N)\}$ a cell $(N,H)$
	in the lattice can be covered 
	by the growing interface
	at the step $T\to T+1$ in two cases:
	\begin{enumerate}[$\bullet$]
		\item 
			it was an inner corner, and 
			event \eqref{eq:P_add_box_independently} occurred;
		\item 
			it was added to the covered inner corner instantaneously
			according to the probabilities 
			\eqref{eq:add_boxes_p_inhom_geom_distr_parameter}--\eqref{eq:p_inhom_geom_distr_parameter}.
	\end{enumerate}
	Here $W_{N,H}$ is identified with the waiting time 
	to cover $(N,H)$ once this cell becomes an inner corner. 
	The coefficient by $W_{N,H}$ in the second line in
	\eqref{eq:LPP_definition}
	is the indicator of the event that the cell $(N,H)$
	is covered instantaneously by a covered inner corner at some
	$(N-j,H)$. The random variable $U_{N-j,H}$ is precisely the random number of boxes
	which are instantaneously added when $(N-j,H)$ is covered, and it has to be at least $j$ to cover
	$(N,H)$. Moreover, it must be $L_{N-1,H}>L_{N,H-1}$, this corresponds to the truncation in 
	\eqref{eq:add_boxes_p_inhom_geom_distr_parameter}.
	When $(N,H)$ is covered instantaneously (so that the indicator is equal to $1$), 
	we have $L_{N,H}=L_{N-1,H}$, and $W_{N,H}$ is not added.
\end{proof}

\begin{remark}
	\label{rmk:usual_LPP_comparison}
	The first line in \eqref{eq:LPP_definition} corresponds to the usual directed
	last-passage percolation model with geometric weights. Denote it by 
	$\widetilde L_{N,H}$, i.e.,
	$\widetilde L_{N,H}=\max(\widetilde L_{N-1,H},\widetilde L_{N,H-1})+W_{N,H}$
	(with the same boundary conditions \eqref{eq:LPP_boundary_cond}).
	Almost surely we have $\widetilde L_{N,H}\ge L_{N,H}$ for all $N,H$.
	Limit shape and fluctuation results for $\widetilde L_{N,H}$ 
	were obtained in \cite{johansson2000shape} (for the homogeneous case $a_N\equiv a$).
	In \Cref{sec:hom_DGCG_asymp} 
	we compare our limit shape with the one for $\widetilde L_{N,H}$.
\end{remark}

\subsection{Strict-weak first-passage percolation}
\label{sub:gb_RSK}

Any TASEP with parallel update and step initial configuration
can be restated in terms of the First-Passage Percolation (FPP)
on a strict-weak lattice. Let us define the FPP model.
Take a lattice $\{(T,j)\colon T\ge0,\, j\ge1\}$, and draw its 
elements as $(1,0)T+(1,1)j\subset \mathbb{R}^2$, see
\Cref{fig:FPP}.
Assign random weights to the edges of the lattice:
put weight zero at each diagonal edge, and 
independent random weights with gB distribution (\Cref{def:gB_random_variable})
at all horizontal edges. This model (with the gB distributed weights)
appeared in \cite{Matrin-batch-2009} together with a queuing interpretation, see
\Cref{rmk:tandem_queues} below.
Its limit shape was described in \cite{Matrin-batch-2009} in terms of a Legendre dual.

\begin{figure}[htbp]
	\centering
	\includegraphics[width=.4\textwidth]{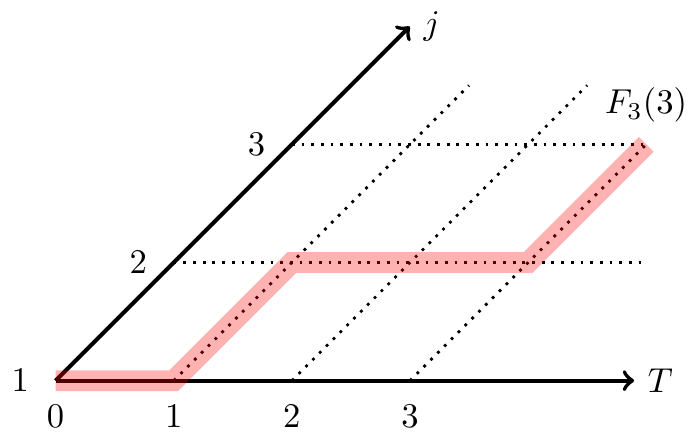}
	\caption{Interpreting $G_j(T)$ as first-passage percolation times.}
	\label{fig:FPP}
\end{figure}

We consider directed paths on our lattice, i.e., paths which are monotone in both $T$ and $j$.
For any path, define its weight to be the sum of weights of all its edges.
Let the \emph{first passage time} $F_j(T)$ from $(0,0)$ to $(T,j)$ 
to be the minimal weight of a path over all directed paths from $(0,0)$ to $(T,j)$.

\begin{proposition}
	\label{prop:FPP_connection}
	We have $F_j(T)=G_j(T+j-1)+j$ for all $j,T$ (equality in distribution
	of families of random variables), where $G_j(T)$ is the coordinate
	of the $j$-th particle in the gB-TASEP started from the step initial configuration.
\end{proposition}
\begin{proof}
	The first passage times
	satisfy the recurrence:
	\begin{equation*}
		F_j(T)=\min(F_{j-1}(T),F_j(T-1)+w_{j,T}),
	\end{equation*}
	where $w_{j,T}$ is the gB random variable at the horizontal edge connecting
	$(j,T-1)$ and $(j,T)$.
	At the same time, the gB-TASEP particle locations satisfy 
	\begin{equation*}
		G_j(T)=\min(G_{j-1}(T-1)-1,G_j(T-1)+\tilde w_{j,T}),
	\end{equation*}
	where $\tilde w_{j,T}$ is the gB random variable corresponding to the 
	desired jump of the $j$-th particle at time step $T-1\to T$.
	One readily sees that the boundary conditions for these recurrences also match, 
	which completes the proof.
\end{proof}

The FPP times $F_j(T)$ have an interpretation 
in terms of column Robinson-Schensted-Knuth (RSK) correspondence.
We refer to \cite{fulton1997young},
\cite{sagan2001symmetric},
\cite{Stanley1999} for details on the RSK correspondences. 
Applying the column RSK to a 
random integer matrix of size $j\times (T+j-1)$
with independent gB entries, one gets a random Young diagram
$\lambda=(\lambda_1\ge \ldots\ge \lambda_j\ge0 )$ of at most $j$
rows. The FPP time is related to this diagram as $F_j(T)=\lambda_j$. 
The full diagram $\lambda$ can also be recovered 
with the help of Greene's theorem
\cite{Greene1974}
by considering 
minima of weights over nonintersecting directed paths in the strict-weak lattice
with edge weights coming from the integer matrix.

To the best of our knowledge, the gB distribution presents a 
new family of random variables for which the corresponding oriented FPP times 
(obtained by applying the column RSK to a random matrix with independent entries)
can be analyzed to the point of asymptotic 
fluctuations.
Other known examples of random variables with 
tractable (to the point of asymptotic fluctuations) behavior of the FPP times
consist of the pure
geometric and Bernoulli distributions.
Under a Poisson degeneration, the question of oriented FPP fluctuations
can be reduced
to the Ulam's problem on asymptotics of the longest increasing 
subsequence in a random permutation. Tracy-Widom fluctuations in the latter
case were obtained in the celebrated work
\cite{baik1999distribution}.

\begin{remark}
	\label{rmk:tandem_queues}
	The oriented FPP model (as well as the TASEP with parallel update)
	is equivalent to a tandem queuing system.
	For our models, the service times in the queues have the gB distribution. 
	We refer to \cite{Baryshnikov_GUE2001},
	\cite{OConnell2003Trans}, \cite{Matrin-batch-2009} for 
	tandem queue interpretation of the usual TASEP 
	as well as of the column RSK correspondence.
	See also the end of \Cref{sub:equiv_intro} 
	for a similar interpretation of the continuous space TASEP.
\end{remark}

\subsection{Continuous space TASEP and semi-discrete directed percolation}
\label{sub:contin_FPP_RSK}

The homogeneous version (i.e., with $\xi(\chi)\equiv 1$)
of the continuous space TASEP with no roadblocks
possesses an interpretation
in the spirit of directed First-Passage Percolation (FPP).
This construction 
is very similar to a well-known interpretation of the
usual continuous time TASEP on $\mathbb{Z}$ via FPP. 
We are grateful to Jon Warren for this observation.

Fix $M\in \mathbb{Z}_{\ge1}$ and 
consider the space $\mathbb{R}_{\ge0}\times \{1,\ldots,M \}$
in which each copy of $\mathbb{R}_{\ge0}$ is equipped with
an independent standard Poisson point process of rate $1$. 
See \Cref{fig:semi-discrete-Poisson} for an illustration.
Let us first recall the connection to the usual 
continuous time, discrete space TASEP $(\tilde X_1(t)>\tilde X_2(t)>\ldots )$, 
$\tilde X_i(t)\in \mathbb{Z}$, $t\in \mathbb{R}_{\ge0}$,
started from the step initial configuration $\tilde X_i(0)=-i$, $i=1,2,\ldots $.
In this TASEP each particle has an independent exponential clock with rate $1$, 
and when the clock rings it jumps to the right by one 
provided that the destination is unoccupied.\begin{figure}[htpb]
	\centering
	\includegraphics[width=.4\textwidth]{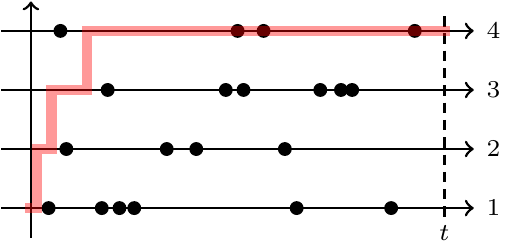}
	\caption{
		A minimal energy up-right path from $(0,1)$ to $(t,4)$
		in the semi-discrete Poisson environment. 
		We have $\tilde X_4(t)+4=3$.}
	\label{fig:semi-discrete-Poisson}
\end{figure}
Fix $t\in \mathbb{R}_{>0}$.
For each $m=1,\ldots,M $ consider up-right paths
from $(0,1)$ to $(t,m)$ as in \Cref{fig:semi-discrete-Poisson}.
The energy of an up-right path is, by definition, the total number of 
points in the Poisson processes lying on this path.
\begin{proposition}
	\label{prop:usual_TASEP_FPP}
	For each $m$ and $t$, the minimal energy 
	of an up-right path 
	from $(0,1)$ to $(t,m)$ in the Poisson environment
	has the same distribution
	as the displacement $\tilde X_m(t)+m$ of 
	the $m$-th particle in the usual TASEP.
\end{proposition}

For the continuous space TASEP consider a variant of this construction
by putting an 
independent exponential random weight with mean $L^{-1}$
at each point of each of the Poisson processes 
as in \Cref{fig:semi-discrete-Poisson}.
That is, let now the weight of each point be random instead of $1$.
One can say that we replace the 
Poisson processes on $\mathbb{R}_{\ge0}\times\left\{ 1,\ldots,M  \right\}$
by \emph{marked Poisson processes}.
This environment corresponds to the continuous space TASEP $(X_1(t)\ge X_2(t)\ge \ldots )$,
$X_i(t)\in \mathbb{R}_{\ge0}$:

\begin{proposition}
	\label{prop:cont_TASEP_FPP}
	For each $t>0$ and $m=1,\ldots,M $ the 
	minimal energy 
	of an up-right path 
	from $(0,1)$ to $(t,m)$ in the marked Poisson environment
	has the same distribution as the coordinate
	$X_m(t)$ of the $m$-th particle in the continuous space TASEP
	with mean jumping distance $L^{-1}$.
\end{proposition}

Both \Cref{prop:usual_TASEP_FPP,prop:cont_TASEP_FPP} are established similarly
to \Cref{prop:FPP_connection} while taking into account the continuous
horizontal coordinate.
The interpretation via minimal energies of up-right paths 
also allows to define 
random Young diagrams depending on the Poisson
or marked Poisson processes, respectively,
by minimizing over collections of nonintersecting up-right paths.
Utilizing Greene's theorem 
\cite{Greene1974},
(see also \cite{fulton1997young},
\cite{sagan2001symmetric}, or
\cite{Stanley1999})
one sees that
in the case of the usual TASEP the distribution of this Young diagram
is the Schur measure $\propto s_\lambda(1,\ldots,1 )s_\lambda(\vec0;\vec0;t)$.
It would be very interesting to understand 
the distribution and asymptotics of random Young diagrams
arising from the marked Poisson environment.

\section{Hydrodynamic equations for limiting densities}
\label{sec:app_hydrodynamics}

Here we present informal derivations of 
hydrodynamic partial differential equations
which the limiting densities and height functions
of the DGCG and continuous space TASEP should satisfy.
These equations follow from constructing families of 
local translation invariant stationary
distributions of arbitrary density for the corresponding dynamics.
The argument could be made rigorous if one shows that these
families exhaust all possible (nontrivial)
translation invariant stationary distributions
(as, e.g., it is for TASEP \cite{Liggett1985}
or PushTASEP \cite{guiol1997resultat}, \cite{andjel2005long}).
We do not pursue this classification question here.

\subsection{Hydrodynamic equation for DGCG}
\label{sub:rmk_hydrodynamics_discrete}

Consider the discrete DGCG model in the asymptotic regime
described in \Cref{sec:hom_DGCG_asymp}. 
Locally around every scaled point 
$\eta$ the distribution of the process should be translation
invariant and stationary under the
homogeneous version of
DGCG on $\mathbb{Z}$ (recall that it depends on the three parameters
$a,\beta,\nu$).
The existence 
(for suitable initial
configurations) of the homogeneous
dynamics on $\mathbb{Z}$
can be established similarly to
\cite{Liggett1973infinite_zerorange}, \cite{andjel1982invariant}.

A supply of translation invariant stationary distributions
on particle configurations on $\mathbb{Z}$ is 
given by product measures. That is, let us independently 
put particles at each site of $\mathbb{Z}$ with the gB probability
(cf.~\Cref{def:gB_random_variable})
\begin{equation}\label{eq:discrete_hydrodyn_distribution}
	\pi(j):=\mathop{\mathrm{Prob}}(\textnormal{$j$ particles at a site})
	=
	\begin{cases}
		\displaystyle \frac{1-c}{1-c\nu}, & j=0;
		\\[10pt]
		\displaystyle
		c^j\,\frac{(1-c)(1-\nu)}{1-c\nu},&
		j\ge1.
	\end{cases}
\end{equation}

\begin{proposition}
	\label{prop:discrete_hydrodyn_distribution}
	The product measure $\pi^{\otimes \mathbb{Z}}$
	on particle configurations in $\mathbb{Z}$ corresponding to 
	the distribution $\pi$ 
	\eqref{eq:discrete_hydrodyn_distribution}
	at each site
	is 
	invariant 
	under the homogeneous DGCG on $\mathbb{Z}$
	with any values of the parameters $a$ and $\beta$. 
\end{proposition}
\begin{proof}
	Let us check directly that $\pi$ is invariant, i.e., satisfies
	\begin{equation}\label{eq:discrete_hydrodyn_distribution_proof}
		\pi(k+1)P(k+1\to k)+\pi(k-1)P(k-1\to k)+\pi(k)P(k\to k)=\pi(k), \qquad k=0,1,2,\ldots,
	\end{equation}
	where $P(k\to l)$ are the one-step 
	transition probabilities 
	of the homogeneous DGCG restricted to a given site (say, we are
	looking at site $0$).
	The probability that a particle coming from the left crosses 
	the bond $-1\to 0$ is equal to\footnote{Note 
		that this calculation
		is greatly simplified by the fact that the update is parallel, 
		otherwise we would have to take into account the full 
		behavior on the left half line. A way to deal with this
		issue for the stochastic six vertex model (which is not parallel update)
		is discussed in, e.g., \cite{Amol2016Stationary}.}
	\begin{equation*}
		u:=\sum_{n=0}^{\infty}\pi(0)^{n}(1-\pi(0))
		\frac{a\beta}{1+a\beta}\left( \frac{\nu+a\beta}{1+a\beta} \right)^n
		=
		\frac{ac\beta}{1+ac\beta},
	\end{equation*}
	where we sum over the number of empty sites to the left of $0$, 
	multiply by the probability that a particle leaves a stack, and 
	then travels distance $n$.
	We have
	\begin{equation*}
		P(k+1\to k)=\frac{a\beta}{1+a\beta}(1-u),
	\end{equation*}
	the probability that a particle leaves the stack at $0$, and another particle 
	does not join it from the left.
	Moreover, for $k\ge1$ we have
	\begin{equation*}
		P(k-1\to k)=
		\frac{u(1-\nu)}{1+a\beta}\,\mathbf{1}_{k=1}+
		\frac{u}{1+a\beta}\,\mathbf{1}_{k\ge2},
	\end{equation*}
	where for $k=1$ we require that the moving particle stops at site $0$,
	and for $k\ge2$ we need the stack at $0$ not to emit a particle.
	Finally, 
	\begin{equation*}
		P(k\to k)=\left( 1-\frac{u(1-\nu)}{1+a\beta} \right)\mathbf{1}_{k=0}
		+
		\left( \frac{a \beta u}{1+a\beta}+\frac{1-u}{1+a\beta} \right)
		\mathbf{1}_{k\ge1},
	\end{equation*}
	where we require that no particle has stopped at site $0$ for $k=0$ and sum
	over two possibilities to 
	preserve the number of particles at $0$ for $k\ge 1$.
	With these probabilities written down,
	checking \eqref{eq:discrete_hydrodyn_distribution_proof} is straightforward.
\end{proof}

The density of particles under the product measure $\pi^{\otimes\mathbb{Z}}$ is 
$\rho(c)=\frac{c(1-\nu)}{(1-c)(1-c\nu)}$,
and the current (i.e., the average number of particles crossing a given bond)
is equal to the quantity $u$ from the proof of \Cref{prop:discrete_hydrodyn_distribution},
that is,
$j(c)=\frac{c a\beta}{1+ca\beta}$.
Thus,
the dependence of the current on the density has the form
(where we recall that the parameters $a,\nu$ depend on the space coordinate $\eta$)
\begin{equation}\label{eq:current_via_density_discrete_model}
	j(\rho)=
	\frac{2 a \beta  \rho }
	{2 a \beta  \rho 
		+\discnu  (\rho -1)+\rho +1+
	\sqrt{(\discnu  (\rho -1)+\rho +1)^2-4 \discnu  \rho ^2}}.
\end{equation}

The partial differential equation
for the limiting density $\rho(\tau,\eta)$
expressing 
the continuity of the hydrodynamic flow
has the form
\cite{Andjel1984}, \cite{Rezakhanlou1991hydrodynamics}, 
\cite{Landim1996hydrodynamics},
\cite{Seppalainen_Discont_TASEP_2010}
\begin{equation}
	\label{eq:DGCG_hydro_eq}
	\frac{\partial}{\partial\tau}\rho(\tau,\eta)
	+
	\frac{\partial}{\partial\eta}j\bigl(\rho(\tau,\eta)\bigr)=0.
\end{equation}
One can readily verify that 
the limit shape \eqref{eq:homogeneous_discrete_limit_shape_parametrization}
satisfies this equation. 
Equation \eqref{eq:DGCG_hydro_eq}
should also hold for the scaling limit of the inhomogeneous
DGCG, when the parameters $a,\beta,\nu$ of the homogeneous
dynamics on the full line $\mathbb{Z}$ depend on the spatial coordinate
$\eta$. That is, one should replace 
$j(\rho(\tau,\eta))$
\eqref{eq:current_via_density_discrete_model}
by 
$j(\rho(\tau,\eta);\eta)$
with $a=a(\eta),a=\beta(\eta),a=\nu(\eta)$ being the scaled values of the parameters.

\subsection{Hydrodynamic equation for continuous space TASEP}
\label{sub:rmk_hydrodynamics_continuous}

Assume that the set of roadblocks $\RoadblockSet$
is empty. Then locally at every point $\chi>0$ the 
behavior of the continuous space TASEP
should be homogeneous.
Locally the parameters can be chosen so that the
mean waiting time to jump is 
$1/\xi\equiv 1/\xi(\chi)$
and the 
mean jumping distance is $1$. 

The local distribution (on the full line $\mathbb{R}$)
should be invariant under space translations,
and stationary under our homogeneous Markov dynamics.
The existence 
(for suitable initial
configurations) of the 
dynamics on $\mathbb{R}$
can be established similarly to
\cite{Liggett1973infinite_zerorange}, \cite{andjel1982invariant}.

A supply of translation invariant stationary distributions of arbitrary density 
may be constructed as follows.
Fix a parameter $0<c<1$ and
consider a Poisson process on $\mathbb{R}$
with rate (i.e.,~mean density) $\frac{c}{1-c}$.
Put a random geometric number of particles at each point of this Poisson process, 
independently at each point, with the geometric distribution 
\begin{equation*}
	\mathrm{Prob}(\textnormal{$j\ge1$ particles})=(1-c)c^{j-1}.
\end{equation*}
Thus we obtain a so-called \emph{marked Poisson process}
--- a distribution of stacks of particles on $\mathbb{R}$.
It is clearly translation invariant.
The stationarity of this process under the dynamics (for any $\xi$) follows by setting $q=0$ in
\cite[Appendix B]{BorodinPetrov2016Exp} so we omit the computation here.

The density of particles under
this marked Poisson process
is 
\begin{equation*}
	\rho=\frac{c}{(1-c)^2}.
\end{equation*}
One can check that the current of particles
(that is, the mean number of particles passing through, say, zero, 
in a unit of time)
has the form
\begin{equation*}
	j=\xi  c=\xi\,\frac{1+2 \rho -\sqrt{1+4 \rho}}{2 \rho }.
\end{equation*}

The partial differential equation
for the limiting density $\rho(\theta,\chi)$
(under the scaling described in \Cref{sub:limit_shape_cont_TASEP})
expressing 
the continuity of the hydrodynamic flow
has the form
$\rho_\theta+(j(\rho))_{\chi}=0$, 
or
\begin{equation}
	\label{eq:hydrodynamic_continuous_equation_for_density}
	\frac{\partial}{\partial\theta}\rho(\theta,\chi)+
	\frac{\partial}{\partial\chi}
	\biggl[ \xi(\chi)
		\,\frac{1+2 \rho(\theta,\chi) -\sqrt{1+4 \rho(\theta,\chi)}}{2 \rho(\theta,\chi) }
	\biggr]
	=0,\qquad 
	\rho(0,\chi)= +\infty\, \mathbf{1}_{\chi=0}.
\end{equation}

The density is related to the limiting height function as 
$\rho(\theta,\chi)=-\frac{\partial}{\partial \chi}\mathfrak{h}(\theta,\chi)$,
and so $\mathfrak{h}$ should satisfy
\begin{equation}
	\label{eq:hydrodynamic_continuous_equation}
	\mathfrak{h}_{\chi}(\theta,\chi)=
	-
	\frac{\xi(\chi)\mathfrak{h}_{\theta}(\theta,\chi)}
	{\bigl(\xi(\chi)-\mathfrak{h}_{\theta}(\theta,\chi)\bigr)^2},\qquad 
	\mathfrak{h}(0,\chi)=+\infty\,\mathbf{1}_{\chi=0}.
\end{equation}
The passage from
\eqref{eq:hydrodynamic_continuous_equation_for_density}
to 
\eqref{eq:hydrodynamic_continuous_equation}
is done via integrating from $\chi$ to $+\infty$ followed by algebraic manipulations.
One can check that 
the limit shape in the curved part
\begin{equation*}
	\mathfrak{h}(\theta,\chi)=
	\theta 
	\mkern2mu
	\mathfrak{w}^\circ(\theta,\chi)-
	\int
	\limits_{0}^{\chi}  \dfrac{\xi (u) \mathfrak{w}^\circ(\theta,\chi) du}{(\xi
	(u) - \mathfrak{w}^\circ(\theta,\chi))^{2}}
\end{equation*}
from \Cref{def:limit_shape_continuous}
indeed satisfies \eqref{eq:hydrodynamic_continuous_equation} whenever all derivatives make sense.
Such a check is very similar to the one performed in the discrete case 
in \Cref{sub:rmk_hydrodynamics_discrete} 
(and also corresponds to setting $q=0$ in \cite[Appendix B]{BorodinPetrov2016Exp}),
so we omit it for the continuous model.

\section{Fluctuation kernels}
\label{sec:app_B1_TW_etc_distributions}

\subsection{Airy$_2$ kernel and GUE Tracy-Widom distribution}
\label{sub:Airy_GUE}

Let $\mathsf{Ai}(x):=\frac{1}{2\pi}\int e^{i\sigma^3/3+i\sigma x}d\sigma$
be the Airy function, where the integration is 
over a contour in the complex plane from $e^{\iu\frac{5\pi}{6}}\infty$
through $0$ to $e^{\iu\frac{\pi}{6}}\infty$.
Define the extended Airy kernel\footnote{In this paper we deal
only with the Airy$_2$ kernel and omit the subscript $2$.}
\cite{macedo1994universal},
\cite{NagaoForresterHonnerExtAiry1998},
\cite{PhahoferSpohn2002}
on $\mathbb{R}\times \mathbb{R}$
by 
\begin{equation}
	\begin{split}
		&
		\mathsf{A}^{\mathrm{ext}}(s,x;s',x')=
		\begin{cases}
			\int_{0}^{\infty}e^{-\mu(s-s')}\mathsf{Ai}(x+\mu)\mathsf{Ai}(x'+\mu)d\mu,&
			\mbox{if $s\ge s'$};\\
			-\int_{-\infty}^{0}e^{-\mu(s-s')}\mathsf{Ai}(x+\mu)\mathsf{Ai}(x'+\mu)d\mu,&
			\mbox{if $s<s'$}
		\end{cases}
		\\[4pt]
		&\hspace{10pt}=
		-
		\frac{\mathbf{1}_{s<s'}}{\sqrt{4\pi(s'-s)}}
		\exp\left(
		- \frac{(x-x')^2}{4(s'-s)}
		-\frac12(s'-s)(x+x')
		+\frac1{12}(s'-s)^3
		\right)
		\\
		&\hspace{30pt}
		+
		\frac1{(2\pi\iu)^2}
		\iint
		\exp\Big(
		s x-s'x'
		-\frac{1}{3}s^3+\frac{1}{3}s'^3
		-(x-s^2)u+(x'-s'^2)v
		\\&\hspace{160pt}
		-s u^2+s' v^2
		+
		\frac13(u^3-v^3)
		\Big)
		\frac{du\,dv}{u-v}.
	\end{split}
	\label{eq:ext_Airy}
\end{equation}
In the double contour integral expression, the $v$ integration contour goes from
$e^{-\iu\frac{2\pi}{3}}\infty$ through $0$ to
$e^{\iu\frac{2\pi}{3}}\infty$,
and the $u$ contour goes from
$e^{-\iu\frac{\pi}{3}}\infty$ through $0$
to
$e^{\iu\frac{\pi}{3}}\infty$,
and the integration contours do not intersect.
This expression for the extended Airy kernel
which is most suitable for our needs appeared in
\cite[Section 4.6]{BorodinKuan2007U},
see also 
\cite{johansson2003discrete}.

We also use the following gauge transformation
of the extended Airy kernel:
\begin{equation}
	\label{eq:tilde_A_2_ext}
	\begin{split}
		&
		\widetilde{\mathsf{A}}^{\mathrm{ext}}(s,x;s',x'):
		=
		e^{-sx+s'x'+\frac13 s^3-\frac13s'^3}
		\mathsf{A}^{\mathrm{ext}}(s,x;s',x')
		\\&\hspace{20pt}=
		-
		\frac{\mathbf{1}_{s<s'}}{\sqrt{4\pi(s'-s)}}
		\exp\left(
			-
			\frac{(s^2-x-s'^2+x')^2}{4(s'-s)}
		\right)
		\\
		&\hspace{40pt}
		+
		\frac1{(2\pi\iu)^2}
		\iint
		\exp\Big(
		-(x-s^2)u+(x'-s'^2)v
		-s u^2+s' v^2
		+
		\frac13(u^3-v^3)
		\Big)
		\frac{du\,dv}{u-v}.
	\end{split}
\end{equation}

When $s=s'$, 
$\mathsf{A}^{\mathrm{ext}}(s,x;s',x')$
becomes the usual Airy kernel (independent of $s$):
\begin{equation}
	\label{eq:usual_Airy}
	\begin{split}
		\mathsf{A}(x;x')
		:=
		\mathsf{A}^{\mathrm{ext}}(s,x;s,x')
		&=
		\frac{1}{(2\pi\iu)^2}
		\iint
		\frac{e^{u^3/3-v^3/3-xu+x'v}du\,dv}{u-v}
		\\&=\frac{\mathsf{Ai}(x)\mathsf{Ai}'(x')-\mathsf{Ai}'(x)\mathsf{Ai}(x')}{x-x'},
		\hspace{45pt}  
		x,x'\in \mathbb{R}.
	\end{split}
\end{equation}

The \emph{GUE Tracy-Widom distribution} function
\cite{tracy_widom1994level_airy}
is the following Fredholm determinant of~\eqref{eq:usual_Airy}:
\begin{equation}
	\label{eq:F_2_GUE_definition}
	F_{GUE}(r)=\det\left( \mathbf{1}-\mathsf{A} \right)_{(r,+\infty)},
	\qquad 
	r\in \mathbb{R},
\end{equation}
defined analogously to \eqref{eq:K_Fredholm_determinant}
with sums replaced by integrals over $(r,+\infty)$.

\subsection{BBP deformation of the Airy$_2$ kernel}
\label{sub:BBP_kernels}

Fix $m$ and a vector $\mathbf{b}=(b_1,\ldots,b_m )\in \mathbb{R}^{m}$.
Define the extended BBP kernel
on $\mathbb{R}\times \mathbb{R}$
by 
\begin{equation}
	\label{eq:ext_BBP}
	\begin{split}
		&
		\widetilde{\mathsf{B}}^{\mathrm{ext}}_{m,\mathbf{b}}(s,x;s',x'):
		=
		-
		\frac{\mathbf{1}_{s<s'}}{\sqrt{4\pi(s'-s)}}
		\exp\left(
			-
			\frac{(s^2-x-s'^2+x')^2}{4(s'-s)}
		\right)
		\\
		&\hspace{10pt}
		+
		\frac1{(2\pi\iu)^2}
		\iint
		\prod_{j=1}^{m}\frac{v-b_i}{u-b_i}
		\exp\Big(
		-(x-s^2)u+(x'-s'^2)v
		-s u^2+s' v^2
		+
		\frac13(u^3-v^3)
		\Big)
		\frac{du\,dv}{u-v}.
	\end{split}
\end{equation}
The integration contours are as in the Airy kernel
\eqref{eq:ext_Airy}
with the additional condition that they both must pass to the 
left of the poles $b_i$.

For $s=s'=0$ this kernel (denote it by 
$\widetilde{\mathsf{B}}_{m,\mathbf{b}}
(x,x')$) 
was introduced
in \cite{BBP2005phase}
the context of spiked random matrices.
The extended version appeared in
\cite{imamura2007dynamics}.
In this paper we are using the gauge transformation
similar to 
\eqref{eq:tilde_A_2_ext}, hence the tilde in the notation.
Denote for $\mathbf{b}=(0,\ldots,0 )$ the corresponding
distribution function by
\begin{equation*}
	F_m(r):=
	\det\bigl( \mathbf{1}-\widetilde{\mathsf{B}}_{m,\mathbf{b}} \bigr)
	_{(r,+\infty)}, \qquad r\in \mathbb{R}.
\end{equation*}

\begin{remark}
	\label{rmk:BBP_kernel_matching}
	Note that in several other papers, e.g.,
	\cite{BorodinCorwinFerrari2012},
	\cite{barraquand2015phase},
	\cite{BorodinPetrov2016Exp}
	the kernel like 
	\eqref{eq:ext_BBP}
	has the reversed product $\prod_{i=1}^{m}\frac{u-b_i}{v-b_i}$,
	but the contours pass to the right of the poles.
	Such a form is equivalent to 
	\eqref{eq:ext_BBP}.
	In \cite{BorodinPeche2009}
	a common generalization with poles on both sides of the 
	contours is considered.
\end{remark}

\subsection{Deformation of the Airy$_2$ kernel arising at a traffic jam}
\label{sub:GUE_deformed_kernels}

For $\delta>0$ introduce 
the following deformation of the extended
Airy$_2$ kernel
\eqref{eq:tilde_A_2_ext}:
\begin{equation}
	\label{eq:tilde_A_2_ext_deformed}
	\begin{split}
		&
		\widetilde{\mathsf{A}}^{\mathrm{ext,\delta}}(s,x;s',x')
		=
		-
		\frac{\mathbf{1}_{s<s'}}{\sqrt{4\pi(s'-s)}}
		\exp\left(
			-
			\frac{(s^2-x-s'^2+x')^2}{4(s'-s)}
		\right)
		\\
		&\hspace{20pt}
		+
		\frac1{(2\pi\iu)^2}
		\iint
		\exp\bigg(
			\frac{\delta}{v}-\frac{\delta}{u}
			-(x-s^2)u+(x'-s'^2)v
			-s u^2+s' v^2
			+
			\frac13(u^3-v^3)
		\bigg)
		\frac{du\,dv}{u-v}
	\end{split}
\end{equation}
with the same integration contours
as in 
the Airy kernel
with the additional condition that they both pass to the left of $0$.
This kernel can be related to certain
random matrix and percolation models
considered in 
\cite{BorodinPeche2009}, 
see 
\Cref{ssub:Borodin_Peche_connection}
for details.
A Fredholm determinant at $s=s'$ of this kernel 
is a deformation of the GUE Tracy-Widom distribution
\eqref{eq:F_2_GUE_definition}:
\begin{equation}
	F_{GUE}^{(\delta,s)}(r)
	=
	\det
	\bigl( 
		\mathbf{1}
		-
		\widetilde{\mathsf{A}}^{\mathrm{ext},\delta}(s,\cdot;s,\cdot) 
	\bigr)_{(r,+\infty)},
	\qquad 
	r\in \mathbb{R}.
	\label{eq:F_2_GUE_definition_deformed}
\end{equation}
Note that this 
deformation additionally \emph{depends} on $s$ in contrast with the undeformed case, so the 
deformation breaks translation invariance of the kernel and the process.
When $\delta=0$, both the extended kernel
\eqref{eq:tilde_A_2_ext_deformed}
and the deformed Tracy-Widom GUE distribution
turn into the corresponding undeformed objects.

One can show by a change of variables in the 
integral in \eqref{eq:tilde_A_2_ext_deformed}
that 
$F_{GUE}^{(\delta,0)}(r+2\delta^{\frac{1}{2}})\to F_{GUE}(2^{-\frac{2}{3}}r)$
as $\delta\to+\infty$.
This explains why the deformed distribution 
$F_{GUE}^{(\delta,0)}$ arises at a phase transition between 
two GUE Tracy-Widom laws. We are grateful to Guillaume Barraquand
for this observation.

\subsection{Fluctuation kernel in the Gaussian phase}
\label{sub:Gaussian_kernels}

Let $m\in \mathbb{Z}_{\ge1}$ and $\gamma>0$ be fixed.
Define the 
kernel on $\mathbb{R}$
as follows:
\begin{equation}
	\label{eq:G_limiting_kernel}
	\begin{split}
		\widetilde{\mathsf{G}}^{\mathrm{ext}}_{m,\gamma}(h;h')
		&:=
		-\mathbf{1}_{\gamma>1}
		\frac{\exp\left\{ -\frac{(h-h'\gamma)^2}{2(\gamma^2-1)} \right\}}{\sqrt{2\pi(\gamma^2-1)}}
		\\&\hspace{50pt}
		+
		\frac{1}{(2\pi\iu)^2}
		\iint
		\exp\left\{ -\tfrac12 w^2+\tfrac12 z^2-hw+h'z \right\}
		\left( \frac{z}{w\gamma} \right)^{m}\frac{dz\,dw}{z-w\gamma}.
	\end{split}
\end{equation}
The $z$ contour is a 
vertical line in the left half-plane
traversed upwards which
crosses the real line to the left of $-\gamma$.
The $w$ contour 
goes from $e^{-\iu\frac{\pi}{6}}\infty$ to 
$-1$ to $e^{\iu\frac{\pi}{6}}$.

For $\gamma=1$
a Fredholm determinant of this kernel 
describes the distribution of the largest eigenvalue
of an $m\times m$ GUE random matrix
$H=[H_{ij}]_{i,j=1}^{m}$, $H^*=H$, 
$\Re H_{ij}\sim \mathcal{N}\bigl(0,\frac{1+\mathbf{1}_{i=j}}{2}\bigr)$, $i\ge j$, 
$\Im H_{ij}\sim \mathcal{N}\bigl(0,\frac{1}{2}\bigr)$, $i>j$.
That is, the distribution function of the largest eigenvalue 
is 
\begin{equation*}
	G_m(r)=
	\det\bigl(
		\mathbf{1}-\widetilde{\mathsf{G}}_{m,1}^{\mathrm{ext}}
	\bigr)_{(r,+\infty)}.
\end{equation*}
The extended version \eqref{eq:G_limiting_kernel} 
appeared in 
\cite{eynard1998matrices},
\cite{imamura2005polynuclear}
(see also \cite{imamura2007dynamics}).

\paragraph{Acknowledgments.}
We are grateful to 
Guillaume Barraquand,
Riddhipratim Basu,
Alexei Borodin,
Eric Cator,
Francis Comets,
Ivan Corwin,
Patrik Ferrari,
Vadim Gorin,
Pavel Krapivsky,
Alexander Povolotsky,
Timo Sepp\"{a}l\"{a}inen,
and Jon Warren 
for helpful discussions. 
A part of the work was completed when the authors
attended the 2017 IAS PCMI Summer Session on Random Matrices, and we are
grateful to the organizers for the hospitality and support.
AK was partially supported by
the NSF grant DMS-1704186.
LP was partially supported by
the NSF grant \mbox{DMS-1664617}.

\printbibliography

\bigskip

\textsc{A. Knizel, Department of Mathematics, Columbia University}

\textit{E-mail address:} \texttt{knizel@math.columbia.edu}

\medskip

\textsc{L. Petrov, Department of Mathematics, University of Virginia, and Institute for Information Transmission Problems}

\textit{E-mail address:} \texttt{lenia.petrov@gmail.com}

\medskip

\textsc{A. Saenz, Department of Mathematics, University of Virginia}

\textit{E-mail address:} \texttt{ais6a@virginia.edu}

\end{document}